\documentclass[11pt]{amsart}
\usepackage{fouriernc} 

\usepackage{Definitions}    
\usepackage{CLMPZ}
\usepackage{PageSetup}      
\usepackage{Environments}   
\usepackage{mathtools}
\usepackage{stmaryrd}       
\usepackage{leftidx}
\usepackage{tensor}
\usepackage{subcaption}
\usepackage{soul}
\usepackage{ifthen}

\newcommand{\swcref}[2]{#1}

\usepackage{tikz} 
\usepackage{tikz-3dplot}
\usetikzlibrary{decorations.pathmorphing}

\newcommand{\mbf}{\mathbf}
\DeclareMathOperator{\Fix}{Fix}

\definecolor{coa}{HTML}{77aadd}
\definecolor{cob}{HTML}{99DDFF}
\definecolor{coc}{HTML}{ee8866}
\definecolor{cod}{HTML}{FFAABB}
\definecolor{coe}{HTML}{bbcc33}
\definecolor{cof}{HTML}{44bb99}
\definecolor{cog}{HTML}{eedd88}
\definecolor{coh}{HTML}{DDDDDD}

\hypersetup{
 pdfkeywords={latex starter,template},
 pdfauthor={Niles Johnson},
}

\usepackage{tikz-cd} 
\usetikzlibrary{decorations.pathmorphing}


\begin{document}

\title{$K$-theory of endomorphisms, the $\mathit{TR}$-trace, and zeta functions}

\author[J. A. Campbell]{Jonathan A. Campbell}
\email{jonalfcam@gmail.com }
\author[J. A. Lind]{John A. Lind}
\email{jlind@csuchico.edu}
\address{Department of Mathematics and Statistics, California State University, Chico, CA USA}
\author[C. Malkiewich]{Cary Malkiewich}
\email{malkiewich@math.binghamton.edu}
\address{Department of Mathematical Sciences, Binghamton University, PO Box 6000, Binghamton, NY 13902}
\author[K. Ponto]{Kate Ponto}
\email{kate.ponto@uky.edu}
\address{Department of Mathematics, University of Kentucky, 719 Patterson Office Tower, Lexington, KY USA}
\author[I. Zakharevich]{Inna Zakharevich}
\email{zakh@math.cornell.edu}
\address{587 Malott, Ithaca, NY 14853}
\maketitle

\begin{abstract}
We show that the characteristic polynomial and the Lefschetz zeta function are manifestations of the trace map from the $K$-theory of endomorphisms to topological restriction homology (TR).   Along the way
we generalize Lindenstrauss and McCarthy's map from $K$-theory of endomorphisms to topological restriction homology, defining it for any Waldhausen category with a compatible enrichment in orthogonal spectra. In particular, this extends their construction from rings to ring spectra.
 We also give a revisionist treatment of the original Dennis trace map from $K$-theory to topological Hochschild homology ($\THH$) and explain its connection to traces in bicategories with shadow (also known as trace theories).
\end{abstract}

\setcounter{tocdepth}{1}
\tableofcontents

\section{Introduction}

The trace of a matrix is one of the most fundamental invariants in mathematics. It is concrete, computable,  easy to define, and ubiquitous. It generalizes to traces of operators, traces of endomorphisms of projective modules, traces in symmetric monoidal categories \cite{dp}, and traces in bicategories with shadow \cite{ponto_thesis,ps:bicat,kaledin_traces}.
The trace is computable because it is additive: given two endomorphisms of $k$-vector spaces $f\colon V \to V$ and $g\colon W \to W$, the trace satisfies
\[\tr(f \oplus g) = \tr(f) + \tr(g).\]
A similar additivity statement holds for exact sequences of $R$-modules, in symmetric monoidal categories \cite{may_additivity}, and in bicategories \cite{ps:linearity}.

Therefore the trace, considered as a function from the set of matrices to the ground ring, can be encoded using a universal additive invariant. The Hattori--Stallings trace
\[
K_0(A) \arr \HH_0(A) \cong A/[A,A],
\]
and its generalization the Dennis trace
$K(A) \to \HH(A),$
make this idea precise.
Here $K(A)$ is the algebraic $K$-theory of a
ring $A$ \cite{quillen, 1126} and $\HH$ is the Hochschild homology.
Following the outline of Goodwillie \cite{goodwillie}, the Dennis trace was further generalized to a map to topological Hochschild homology $\THH(A)$, then to topological restriction homology $\TR(A)$ and topological cyclic homology $\TC(A)$ in the celebrated work of B\" okstedt, Hsiang, and Madsen \cite{bokstedt_thh,bokstedt_hsiang_madsen}. The invariants $\THH, \TR$ and $\TC$ are the source of much of our computational knowledge of algebraic $K$-theory.

The Hattori-Stallings trace is constructed in a concrete way from the ordinary trace of endomorphisms of modules. In this paper we show that the same is true of the Dennis trace and its refinements to $\THH$ and $\TR$: they also encode concrete and computable trace invariants. This is a shift in perspective, because typically $\THH$, $\TR$, and $\TC$ are viewed as tools for computing the whole of $K$-theory, rather than a sequence of natural receptacles for trace maps.
Our goals are two-fold:
\begin{itemize}
	\item To explain why the invariants comprising the Dennis trace $K(A) \rto \THH(A)$ and the $\TR$ trace $K(A) \rto \TR(A)$ are generalized traces arising in the bicategorical duality theory of Ponto and Ponto-Shulman \cite{ponto_thesis,ps:bicat}. These invariants, which include the trace of a matrix, the characteristic polynomial, and the Lefschetz zeta function, are easy to define, frequently computable, and have excellent formal properties.
	\item To carefully explicate the construction of the Dennis trace map and its generalizations. We follow previous accounts of the Dennis trace \cite{dundas_mccarthy,blumberg_mandell_published,dundas_goodwillie_mccarthy}, using shadows in bicategories to simplify and conceptualize the definition.
\end{itemize}

As a result of the first goal, we also show that fixed-point and periodic-point invariants of ``Reidemeister type'' lift along the Dennis trace, as in \cite{iwashita,gn_k_theory}.

In summary, we view $\THH$ not as a stepping stone to $K$-theory computations, but as an important receptacle for invariants in its own right. This shift in perspective is accompanied by a shift in emphasis in the definition of the Dennis trace.  Cyclic invariance has been central to the construction of the Dennis trace since its invention by Dennis \cite[p.36]{waldhausen_2}. In that guise, cyclicity is more commonly called the Dennis--Waldhausen--Morita argument \cite{blumberg_mandell_unpublished}. We expand this idea, putting it in the context of bicategorical traces.

\subsection{Statement of results: Invariants}
In order to relate the Dennis trace to bicategorical traces, we consider a generalization of the Dennis trace of the form
\[ \widetilde K(A;M) \arr \TR(A;M) \arr \THH(A;M) \]
which was studied by Lindenstrauss and McCarthy \cite{LM12} in the case of discrete rings and bimodules. Here $\THH(A;M)$ denotes topological Hochschild homology with coefficients in an $(A,A)$-bimodule $M$, $K(A;M)$ is the $K$-theory of perfect $A$-modules $P$ and twisted endomorphisms
\begin{equation}\label{intro_twisted_endomorphism}
	f\colon P \to M \otimes_A P,
\end{equation}
and $\widetilde K(A;M)$ is the cofiber of the map $K(A) \to K(A;M)$ that sends each perfect $A$-module to its zero endomorphism.

We will recall in \S\ref{sec:bicat_trace} that a twisted endomorphism $f\colon P \arr M \otimes_A P$, with $P$ a dualizable $A$-module, has an associated bicategorical trace (\cref{defn:euler_characteristic})
\[
\tr(f) \colon \bbS \arr \THH(A; M)
\]
Our first result, an elaboration of \cite[7.11]{cp}, says that the Dennis trace encodes the bicategorical trace.

\begin{thm} [\cref{ex:twisted_trace_is_twisted_trace}]\label{intro_dennis_trace_w_coeffs}
	For any ring or ring spectrum $A$ and a $(A,A)$-bimodule $M$, there is a generalized Dennis trace map (\cref{twisted_dennis_trace})
	\[ \xymatrix{ \widetilde K(A;M) \ar[r]^-\trc & \THH(A;M) } \]
	that on $\pi_0$ takes the class of an endomorphism $f\colon P \to M \otimes_A P$ to its bicategorical trace $\tr(f) \colon \bbS \to \THH(A; M)$.
\end{thm}

More generally, topological restriction homology encodes the traces of the iterates of an endomorphism.

\begin{thm} [\cref{thm:main_identify_pi_0_TR}]\label{intro_tr_trace_w_coeffs}
	There is a lift of the Dennis trace to topological restriction homology (\cref{tr_trace})
	\[ \xymatrix{ \widetilde K(A;M) \ar[r]^-\trc & \TR(A;M) } \]
	that on $\pi_0$ takes the class of an endomorphism $f\colon P \to M \otimes_A P$ to the trace of its $n$-fold iterate
	\[ f^{\circ n}\colon P \to M \otimes_A  \dotsm \otimes_A M \otimes_A P \]
	for every $n \geq 1$.
\end{thm}
We call the map in \cref{intro_tr_trace_w_coeffs} the {\bf $\TR$-trace}.
The characteristic polynomial of a matrix is a refinement of the trace, and is encoded by the $\TR$-trace.

\begin{thm} [\cref{tr=char}]\label{intro:tr=char}
	If $A$ is a discrete commutative ring, then the composite
	\[
	\xymatrix{&\widetilde K_0(A;A)\ar[r]^-\trc & \pi_0\TR(A) \ar[r]^-\cong & (1+tA[[t]])^\times }
	\]
	takes the class $[f \colon P \to P]$ of an endomorphism to its characteristic polynomial $\det(1 - tf)$.
\end{thm}
We emphasize that \cref{intro:tr=char} states that the $\TR$-trace is \textit{exactly} the homotopical analogue of the characteristic polynomial. Since zeta functions are built out of characteristic polynomials, we summarize with the slogan:
\begin{quote}
\emph{$K$-theory is the natural home for additive invariants, $\THH$ is the natural home for traces, and $\TR$ is the natural home for zeta functions.}
\end{quote}
A related slogan occurs in topological fixed-point theory:
\begin{quote}
\emph{$\THH$ is the natural home for fixed-point invariants and $\TR$ is the natural home for periodic-point invariants.}
\end{quote}

The following result captures this idea, and is the topological analogue of the algebraic slogan.

\begin{thm}[\cref{fuller_underneath_k_theory,lefschetz_zeta_function}]\label{intro_fuller}
	Every self-map $f\colon X \to X$ of a connected finite complex defines a canonical class in endomorphism $K$-theory
\[ [f] \in K_0(\bbS[\Omega X]; \bbS[\Omega^f X]).\]
The image of this class under the $\TR$-trace coincides with the periodic-point invariant $R(\Psi^{\infty}(f))$ studied in \cite{mp1}.

Composing with the map on $\TR$ induced by the ring map
\[
\bbS[\Omega X] \xarr{\text{collapse}} \bbS \xarr{\text{unit}} H\bbZ,
\]
the image in $\pi_0\TR(\bbZ) \cong (1+t\bbZ[[t]])^\times$ is the Lefschetz zeta function of $f$:
  \[
  \zeta(t) = \exp \left(\sum^\infty_{n=1} L(f^{\circ n}) \frac{t^n}{n} \right).
  \]
\end{thm}

In more detail, the image of $[f]$ in $\pi_0 \TR(\bbS[\Omega X]; \bbS[\Omega^f X])$ is given by the Fuller traces $R(\Psi^n f)^{C_n}$ for all $n \geq 1$. These are the strongest invariants that detect the $n$-periodic points of $f$ up to homotopy, and our work here extends \cite{mp1} by lifting them to the $K$-theory of spherical group rings. This realizes a vision of Klein, McCarthy, Williams and others about the fundamental nature of these periodic-point invariants.

The theorem also suggests that the higher homotopy groups of $K$-theory with coefficients capture parameterized versions of the Lefschetz zeta function, just as $K$-theory without coefficients captures parametrized Euler characteristics \cite{dwyer_weiss_williams}. We intend to return to this idea in future work.

\subsection{Statement of Results: The Dennis Trace}

In order to prove that the trace maps out of $K$-theory encode bicategorical traces, as described in the theorems above, we integrate the perspective of shadows into the construction of the Dennis trace. This has the unexpected benefit of simplifying many aspects of its construction. We emphasize that our definition is similar to and very much motivated by the work in   \cite{dundas_mccarthy,blumberg_mandell_published,dundas_goodwillie_mccarthy}, but the focus on shadows is conceptually clarifying.

To make sense of both the algebraic K-theory of a category and its topological Hochschild homology we need the category to be a spectral category and have a compatible Waldhausen structure.  Applying the building blocks of algebraic $K$-theory (i.e. applying $w_\bullet$ and $S_\bullet$) to a Waldhausen category goes back to Waldhausen's original work, but applying these to a spectral category is the most technically demanding portion of the paper.  For the introduction we will treat this step as a black box.

Given a spectral category $\cC$ and a Waldhausen category $\cC_0$ with appropriate compatibility (\cref{def:spectrally_enriched_waldhausen_category}),
the foundation of the Dennis trace is the inclusion of the zero skeleton in $\THH$:
\[
\bigvee_{f \in \operatorname{End}(\uncat{\cC})} \mathbb{S} \to \THH(\mathcal{C}).
\]
Note that the object on the left depends only on the base category $\cC_0$, which we assumed to be Waldhausen. Since $w_\bullet$ and $S_\bullet$ can be applied to both $\cC_0$ and $\cC$, the inclusion of the zero skeleton gives a map of bisimplicial spectra
\[
\Sigma^\infty \operatorname{ob} w_\bullet S_\bullet \End(\uncat{\cC}) \to \THH(w_\bullet S_\bullet \mathcal{C})
\]
and more generally for each $n \geq 0$ a map of $(n+1)$-fold multisimplicial spectra
\[
\Sigma^\infty \operatorname{ob} w_\bullet S^{(n)}_{\bullet, \dotsc, \bullet} \End(\mathcal{C}_0) \to \THH(w_\bullet S^{(n)}_{\bullet, \dotsc, \bullet} \mathcal{C}).
\]
The Dennis trace is then defined to be a map in the homotopy category
\begin{equation}\label{eq:intro_dennis_trace} \trc \colon K(\End(\uncat\cC)) \arr \THH(\cC) \end{equation}
obtained from a zig-zag of the form
\[
\Sigma^\infty \operatorname{ob} w_\bullet S^*_{\bullet, \dotsc, \bullet} \End(\mathcal{C}_0) \to \THH(w_\bullet S^*_{\bullet, \dotsc, \bullet} \mathcal{C}) \xleftarrow{\simeq} \THH(S^*_{\bullet, \dotsc, \bullet}) \xleftarrow{\simeq} \Sigma^\infty \THH(\cC).
\]
The backwards maps of the zig-zag are provided by the following two theorems.
\begin{thm}[\cref{w_bullet_invariance}] If $\cC$ is a spectral category and $w_k\cC$ is the associated category of flags of weak equivalences in $\cC$, then there is a natural equivalence
	\[
	\THH(w_k \mathcal{C}) \xleftarrow{\simeq} \THH(\mathcal{C}).
	\]
\end{thm}
\begin{thm}[Additivity of $\THH$, \cref{cor:small_add}]\label{intro_thh_additivity}
	Let $\cC$ be a spectral category and let $S_2\cC$ be the associated spectral category of cofiber sequences in $\cC$. Then there is a natural equivalence
	\[
	\THH(S_2 \mathcal{C}) \xleftarrow{\simeq} \THH(\mathcal{C}) \vee \THH(\mathcal{C}).
	\]
	These equivalences inductively define an equivalence
	\[
	\THH(S_{\bullet} \cC) \xleftarrow{\simeq} \Sigma \THH(\cC),
	\]
	and thus an equivalence to the iterated $S_{\bullet}$-construction
	\[
\THH(S^{(n)}_{\bullet, \dotsc, \bullet}\cC) \xleftarrow{\simeq} \Sigma^{n} \THH(\cC).
	\]
\end{thm}
Note that, as a result, the zig-zag defining \eqref{eq:intro_dennis_trace} has \textit{two} spectral directions. One spectral direction comes from the enrichment of $\cC$. The other spectral direction comes from the iterated $S_\bullet$-construction and additivity.

The above two theorems are essential components in the construction of the Dennis trace. They are well known in many different contexts \cite{dundas_mccarthy, blumberg_mandell_unpublished,blumberg_mandell_published,dundas_goodwillie_mccarthy,oberwolfach_report}.  We provide new proofs in the context of spectral Waldhausen categories that highlight how these theorems are \textit{completely formal} consequences of the fact that
\begin{itemize}
	\item $\THH$ is a shadow on the bicategory of spectral categories and spectral bimodules, and that
	\item $\THH$ preserves cofiber sequences in the bimodule slot.
\end{itemize}

We define the Dennis trace for any ring or ring spectrum $A$ by applying the above to the spectral Waldhausen category $\tensor[^A]{\Perf}{}$ of perfect $A$-module spectra:
\begin{equation}\label{eq:intro_dennis_trace_2}  K(\End(A)) \coloneqq K(\End(\tensor[^A]{\Perf}{_0})) \arr \THH(\tensor[^A]{\Perf}{}). \end{equation}
To make this land in $\THH(A)$ we use one final core result, which is also a formal consequence of the shadow property.
\begin{thm}[Morita invariance of $\THH$, \cref{ex:classic_morita_equiv}]\label{intro_thh_morita}
	There is a natural equivalence
	\[
	\THH(\tensor[^A]{\Perf}{}) \xrightarrow{\simeq} \THH(A)
	\]
	defined by a bicategorical trace.
\end{thm}
Again, this is well known, but recognizing that the map underlying the equivalence is itself a bicategorical trace is clarifying and simplifies the proof.

The trace to $\THH(A;M)$ for an $(A, A)$-bimodule $M$ proceeds in the same way, using variants of the above theorems with coefficients. To define the lift to $\TR$ as in \cite{LM12} we perform the same manipulations but replace the endomorphisms $c_0 \to c_0$ in $\uncat\cC$ by length $r$ cycles of maps
\begin{equation}
	a_1 \xarr{f_1} a_2 \xarr{f_2} a_3 \xarr{f_3} \cdots \xarr{f_{r-1}} a_r \xarr{f_f} a_1
\end{equation}
for each $r \geq 1$, and include these into the zero skeleton of $\THH^{(r)}(\spcat\cC)$, a certain $r$-fold subdivision of $\THH$. The resulting traces agree by taking fixed points along the action of a cyclic group that rotates the endomorphisms, and therefore they assemble together into a map to $\TR$.

\subsection{Connection to the literature}
In the case of discrete or simplicial rings $A$, the trace of \cref{intro_dennis_trace_w_coeffs} is not new. The algebraic $K$-theory of parametrized endomorphisms $K(A;M)$ and its trace to $\THH(A;M)$ were first defined in \cite{dundas_mccarthy_stable} for exact categories, see also \cite{iwashita,dundas_goodwillie_mccarthy}. The lift to $\TR(A;M)$ was constructed for discrete rings (or exact categories) by Lindenstrauss and McCarthy in \cite{LM12}. Our contribution is mainly to re-tool the construction so that it works for any ring spectrum, or more generally any spectrally enriched Waldhausen category.

Our reworking uses the Hill-Hopkins-Ravenel equivariant norm of \cite{hhr} and the associated cyclotomic structure on $\THH$ from \cite{ABGHLM_TC,malkiewich_thh_dx,dmpsw}. Many of our arguments are adaptations and conceptualizations of work of Blumberg and Mandell \cite{blumberg_mandell_unpublished,blumberg_mandell_published,blumberg_mandell_cyclotomic}. To identify the image on $\pi_0$ we make heavy use of the main result of \cite{cp}.

In the setting of stable $\infty$-categories, the Dennis trace has a universal characterization \cite{bgt,bgt_endo}.
The point-set model of the Dennis trace for spectrally enriched Waldhausen categories serves as a concrete description of the trace for stable $\infty$-categories. (Note from \cite{bgt} that the two settings are essentially equivalent.) We expect that the generalized Dennis trace constructed here will similarly underlie the $\infty$-categorical Dennis trace with coefficients \cite{bgt_endo,oberwolfach_report}, and the universal characterization of the $\TR$-trace described in forthcoming work of Nikolaus \cite{nikolaus_future}.

Finally, on the subject of fixed-point theory, we note that \cref{intro_fuller} is closely related to the main result of \cite{iwashita}, which lifts the Reidemeister traces of the iterates $R(f^{\circ n})$ to $K_0(\bbZ[\pi_1 X]; \bbZ[\pi_1 X^f])$. They are related because on $\pi_0$, the Fuller trace $R(\Psi^n f)^{C_n}$ is equivalent to the Reidemeister traces $R(f^{\circ k})$ for all $k | n$, by \cite{mp1}. We anticipate that the formulation in \cref{intro_fuller} will be needed for future generalizations to families of endomorphisms, where the Fuller trace becomes a strictly stronger invariant than $R(f^{\circ n})$, and approaches that use discrete rings tend to break down.

\subsection{Organization}
We recall preliminaries on duality and traces in symmetric monoidal categories and bicategories, as well as on equivariant spectra, in \S\ref{sec:bicat_trace}.  \S\ref{sec:spec_wald_cat}--\ref{sec:spectral_bimodules} recall and extend necessary foundations to apply the trace in categories that are compatibly spectrally enriched and have a Waldhausen structure.  In \S\ref{sec:add} we revisit the additivity of $\THH$ using shadows in preparation for the definition of the Dennis trace in \S\ref{sec:dennis_trace}.  We extend this definition to an equivariant trace in \S\ref{sec:equiv_dennis_trace} and use it to define the $\TR$ trace in \S\ref{sec:trace_to_TR}.  Finally in
\S\ref{sec:char_poly_zeta} we describe applications to homotopical characteristic polynomials and periodic point invariants.

\subsection{Acknowledgments}

JC would like to thank Andrew Blumberg, Mike Mandell, and Randy McCarthy for helpful conversations about this paper, and for general wisdom about trace methods.
CM would like to thank Randy McCarthy for persistently telling him about the TR trace for years -- it's beginning to sink in a little.
KP was partially supported by NSF grant DMS-1810779 and the University of Kentucky Royster Research Professorship.
The authors thank Cornell University for hosting the initial meeting which led to this work.

\section{Preliminaries: duality, bicategories, and spectra}\label{sec:bicat_trace}
We begin with a slogan:
\begin{quotation}
  \emph{Every endomorphism of a finite mathematical object defines a class in
  $K$-theory, and the Dennis trace takes its trace.}
\end{quotation}
In this section, we recall many of the fundamental definitions in this slogan.   We define the trace of an endomorphism in a symmetric monoidal category, and then extend the formalism to the noncommutative setting of bicategories.
Ideas suggesting this approach can be found in \cite{nicas_traces}, but the
first successful formalization was the notion of ``shadowed bicategory'' in the
thesis of the fourth author \cite{ponto_thesis, ps:bicat}, later re-discovered
by Kaledin under the name ``trace theory'' \cite{kaledin_traces,kaledin_trace_theories_periodicity}.  In this
section we give a brief introduction to these ideas.  The reader is encouraged
to consult \cite{ponto_thesis, ps:bicat, ps:smc,
  dp}
for more details.

\subsection{Duality and trace in symmetric monoidal categories}
An object $X$ of a symmetric monoidal category $(\cC, \otimes, I)$ is {\bf dualizable} if there exists an object $X^*$, together with an evaluation map $\epsilon \colon X \otimes X^\ast \to I$ and a coevaluation map
$\eta \colon I \to X^\ast \otimes X$, such that both composites
\[
\begin{tikzcd}[row sep=tiny]
  X \ar{r}{\cong} & X \otimes I \ar{r}{\id \otimes \eta} & X \otimes X^\ast \otimes X \ar{r}{\epsilon \otimes \id} & I \otimes X \ar{r}{\cong} & X \\
  X^\ast \ar{r}{\cong} & I \otimes X^\ast \ar{r}{\eta \otimes \id} & X^\ast \otimes X \otimes X^\ast \ar{r}{\id \otimes \epsilon} & X^\ast \otimes I \ar{r}{\cong} & X^\ast
\end{tikzcd}
\]
are identity maps. The dual object $X^*$ is unique up to canonical isomorphism.

Given a dualizable object $X$,  the {\bf trace} of a map $f \colon X \to X$ is the composite
\begin{equation}\label{commutative_trace}
\begin{tikzcd}
\tr(f) \colon I \ar{r}{\eta} & X^\ast \otimes X \ar{r}{\id \otimes f} &  X^\ast \otimes X \ar{r}{\cong} & X \otimes X^\ast \ar{r}{\epsilon} & I.
\end{tikzcd}
\end{equation}
When $f$ is the identity morphism, we call $\tr(\id_{X})$ the {\bf Euler characteristic} of the object $X$ \cite{dp,lms,ps:smc}.

\begin{example} In classical contexts, the above definition becomes familiar.
  \begin{enumerate}
  \item In the category of vector spaces over a field $k$, the trace of an
    endomorphism $f \colon V \arr V$ of a finite dimensional vector space is the $k$-linear map $\tr(f) \colon k \arr k$ given by multiplication by the trace of a matrix representing $f$.
  \item In the stable homotopy category of spectra, the trace of the identity map on the suspension spectrum $\Sigma^{\infty}_{+} X$ of a finite CW complex $X$ is a map $\tr(\id_{\Sigma^{\infty}_{+} X}) \colon \bbS \rto \bbS$ whose degree is the Euler characteristic of $X$.
  \item More generally, if $f \colon X \arr X$ is a self-map of a finite CW complex, then the trace of the stable map $\Sigma^{\infty}_{+} f \colon \Sigma^{\infty}_{+} X \arr \Sigma^{\infty}_{+} X$ is
    the Lefschetz number $L(f)$ \cite{dp,dold:enr}.
  \end{enumerate}
\end{example}

\subsection{Bicategories and shadows} \label{sec:bicats}
If $A$ is a non-commutative ring then $A$-modules do not form a symmetric
monoidal category. Hence the trace as defined in \eqref{commutative_trace}
does not make sense.
To take the trace of an endomorphism $f\colon M \to M$ in this setting, one must circumvent the problem that
\[M \otimes N \ncong N \otimes M\] --- they are not even objects of the same
type. The Hattori--Stallings trace solves this issue in an ad-hoc way, by
modding out by a commutator ideal. The general solution to this issue first
appeared in \cite{ponto_thesis} (and was independently developed in work of
Kaledin on cyclic $K$-theory \cite{kaledin_traces}). The idea is to use
\emph{bicategories} to encode noncommutativity, and create a type
of wrapper $\sh{-}$, called a ``shadow,'' which removes just enough information
to give us commutativity when we need it.

\begin{defn}
A {\bf bicategory} $\mc{B}$ consists of objects, $A, B, \dotsc$, called $0$-cells, and categories $\mc{B}(A, B)$ for each pair of objects $A, B$. Objects in the category $\mc{B}(A, B)$ are called 1-cells and morphisms are called 2-cells. The bicategory is further equipped with horizontal composition functors
\begin{equation}\label{eq:odot_convention}
\odot \colon  \mc{B}(A, B) \times \mc{B}(B, C) \to \mc{B}(A, C),
\end{equation}
that are associative and have units $U_A \in \mc{B}(A, A)$, up to coherent isomorphism.
\end{defn}

In our context, the horizontal composition will substitute for the tensor
product; the following family of examples is used throughout this section as
motivation.

\begin{example} \label{ex:bicat_rings}
There is a bicategory with one 0-cell for each ring $A$. For each pair of rings $A$ and $B$, the category $\mc{B}(A,B)$ is the category of $(A,B)$-bimodules. The horizontal composition is the tensor product $\otimes_B$.
\end{example}

This bicategory serves as motivation for
 the bicategory of spectral categories,
bimodules, and homotopy classes of maps of bimodules, which we describe in
\S\ref{sec:spectral_bimodules}.  The true work of the paper requires the bicategory in \S\ref{sec:spectral_bimodules}.

In order to define the trace, extra structure is
required.

\begin{defn}[\cite{ponto_thesis}]\label{def:shadow}
  Let $\mc{B}$ be a bicategory.  A \textbf{shadow functor} for $\mc{B}$ consists
  of the following data:
  \begin{description}
  \item[a target category] $\mbf{T}$,
  \item[functors]   \[
      \sh{-}\colon  \mc{B}(C,C) \to \mbf{T}
    \]
    for each object $C$ of $\mc{B}$,
  \item[a natural isomorphism] \begin{equation}\label{eq:sh_iso} \theta\colon
      \sh{M\odot N}\xto{\cong} \sh{N\odot M}
    \end{equation}
    for $M\in \mc{B}(C, D)$ and $N\in \mc{B}(D, C).$
  \end{description}
  These must satisfy the condition that the following
  diagrams commute whenever they make sense:
  \begin{description}
  \item[cyclic associativity]
      \[\xymatrix{
          \sh{(M\odot N)\odot P} \ar[r]^\theta \ar[d]_{\sh{a}} &
          \sh{P \odot (M\odot N)} \ar[r]^{\sh{a}} &
          \sh{(P\odot M) \odot N}
          \\
          \sh{M\odot (N\odot P)} \ar[r]^\theta
          & \sh{(N\odot P)\odot M} \ar[r]^{\sh{a}}
          & \sh{N\odot (P\odot M)}\ar[u]_\theta
        }. \]
    \item[unitality]
      \[\xymatrix{
          \sh{M\odot U_C} \ar[r]^\theta \ar[dr]_{\sh{r}} &
          \sh{U_C\odot M} \ar[d]^{\sh{l}} \ar[r]^\theta &
          \sh{M\odot U_C} \ar[dl]^{\sh{r}}
          \\
          &\sh{M}}.\]

    \end{description}
  \end{defn}
  If $\sh{-}$ is a shadow functor on $\mc{B}$, then the composite
  \[\xymatrix{
      \sh{M\odot N} \ar[r]^\theta &
    \sh{N\odot M} \ar[r]^\theta &
    \sh{M\odot N}
   }\]
is the identity \cite[Prop.~4.3]{ps:bicat}. More generally, the circular product $\sh{M_1 \odot \dotsm \odot M_n}$ of any composable list of 1-cells $M_1, \dotsc, M_n$ is well-defined up to canonical isomorphism \cite[1.6]{mp2}.

\begin{example}
The 0th Hochschild homology group $\sh{M} = HH_0(A;M) \coloneqq M/(am - ma)$ defines a shadow on the bicategory of rings and bimodules. The isomorphism
  \[
  \theta\colon  \HH_0 (A, M \otimes_B N) \to \HH_0 (B, N \otimes_A M)
  \]
  is given by observing that both sides are the same quotient of $M \otimes N$.
\end{example}
If we modify the bicategory of \cref{ex:bicat_rings} by taking derived tensor products $\otimes^{\bbL}$ instead of ordinary ones, then the higher Hochschild homology $HH_*(A;M)$ is also a shadow. See \cref{main_homotopy_bicategory} for an analog of this using topological Hochschild homology.

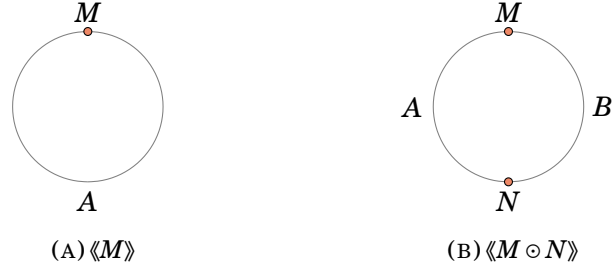
\begin{figure}[h]
     \centering
	\hspace{1in}
     \begin{subfigure}[b]{0.25\textwidth}
         \centering
         \begin{tikzpicture}
\draw [ domain=90:270 ,smooth,variable =\x, gray] plot ({cos \x},{sin \x}); 
\draw [ domain=270:450 ,smooth,variable =\x, gray] plot ({cos \x},{sin \x}); 

\node[below] at ({cos 270},{sin 270}) {$A$};

\node[circle,draw=black, fill=coc, inner sep=0pt,minimum size=3pt] at ({cos 90},{sin 90}) {};
\node[above] at ({cos 90},{sin 90}) {$M$};
\end{tikzpicture}
         \caption{$\sh{M}$}
         \label{fig:shadow_single}
     \end{subfigure}
     \hfill
     \begin{subfigure}[b]{0.25\textwidth}
         \centering
         \begin{tikzpicture}
\draw [ domain=90:270 ,smooth,variable =\x, gray] plot ({cos \x},{sin \x}); 
\draw [ domain=270:450 ,smooth,variable =\x, gray] plot ({cos \x},{sin \x}); 

\node[left] at ({cos 180},{sin 180}) {$A$};
\node[right] at ({cos 0},{sin 0}) {$B$};

\node[circle,draw=black, fill=coc, inner sep=0pt,minimum size=3pt] at ({cos 90},{sin 90}) {};
\node[above] at ({cos 90},{sin 90}) {$M$};
\node[circle,draw=black, fill=coc, inner sep=0pt,minimum size=3pt] at ({cos 270},{sin 270}) {};
\node[below] at ({cos 270},{sin 270}) {$N$};
\end{tikzpicture}
         \caption{$\sh{M\odot N}$}
         \label{fig:shadow_double}
     \end{subfigure}
	\hspace{1in}
        \caption{Graphical representations of shadows}
        \label{fig:pic_shadows}
\end{figure}

While we won't make any formal use of graphical reasoning or string diagram calculi, ``cartoon'' images of the shadow and later generalizations can be useful.
\cref{fig:pic_shadows} contains two examples of this.  We think of a 1-cell $M$ as a vertex with two edges labeled by the 0-cells which are the source and target of $M$.  Then the shadow of $M$ glues the free ends of these edges to each other, as in \cref{fig:shadow_single}.  The shadow of the horizontal composite of compatible 1-cells is displayed in \cref{fig:shadow_double}.

\subsection{Duality and trace}

With a shadow we can now define traces in bicategories.  We start by recalling the generalization of dualizability to bicategories.

\begin{defn}We say that a 1-cell $P\in \mc{B}(C,D)$ in a bicategory is \textbf{left dualizable} if there is a 1-cell $P^\ast\in \mc{B}(D,C)$, called its \textbf{left dual}, and coevaluation and evaluation 2-cells $\eta\colon U_D \to P^\ast \odot P$ and $\epsilon \colon P \odot P^\ast \to U_C$ satisfying the triangle identities. We say that $(P^\ast,P)$ is a \textbf{dual pair}, that $P^\ast$ is \textbf{right dualizable}, and that $P$ is its \textbf{right dual}.
\end{defn}

\begin{example}\hfill
\begin{enumerate}
\item
For rings $C$ and $D$, a $(C,D)$-bimodule $P$ is left dualizable if and only if it is finitely generated and projective as a left $C$-module.
\item The $2$-category of small categories, functors, and natural transformations
  is a bicategory. The functors and their compositions may either be written from right to left (function convention), or from left to right (bimodule convention, \eqref{eq:odot_convention}).

  Under the function convention, a functor $G\colon \mc{C} \to \mc{D}$ is left dualizable if and only if it is a left adjoint. Under the bimodule convention, $G$ is left dualizable if and only if it is a right adjoint.
\end{enumerate}
\end{example}

\begin{defn}[\cite{ponto_thesis}]\label{defn:trace}
  Let $\mc{B}$ be a bicategory with a shadow functor and let $(P^\ast,P)$ be a
  dual pair of 1-cells.  Let $M\in \mc{B}(C,C)$ and $N\in \mc{B}(D,D)$ be 1-cells.
The \textbf{trace} of a 2-cell $f\colon P\odot N\to M\odot P$ is the composite
\[
\resizebox{\textwidth}{!}{$
   \sh{N}\cong \sh{U_D \odot N}\xto{\sh{\eta\odot\id_N}}
  \sh{P^\ast \odot P \odot N} \xto{\sh{\id_{P^\ast}\odot f}}
  \sh{P^\ast \odot M \odot P} \xto{\theta}
  \sh{M \odot P \odot P^\ast} \xto{\sh{\id_M \odot \epsilon}}
    \sh{M \odot U_C} \cong \sh{M}.
$}
\]
The trace for a 2-cell $g\colon N\odot P^\ast \rightarrow P^\ast \odot M$ is defined similarly.
\end{defn}
See \cref{fig:picture_trace}.
\begin{figure}
\resizebox{\textwidth}{!}{\begin{tikzpicture}
\pgfmathsetmacro{\shift}{4}

\draw [ domain=90:270 ,smooth,variable =\x, gray] plot ({-\shift+ cos \x},{sin \x}); 
\draw [ domain=270:450 ,smooth,variable =\x, gray] plot ({-\shift+ cos \x},{sin \x}); 

\node[left] at ({-\shift+ cos 150},{sin 150}) {$A$};
\node[below] at ({-\shift+ cos 270},{sin 270}) {$A$};

\node[circle,draw=black, fill=cob, inner sep=0pt,minimum size=3pt] at ({-\shift+ cos 210},{sin 210}) {};
\node[below left] at ({-\shift+ cos 210},{sin 210}) {$Q$};

\draw [ domain=90:270 ,smooth,variable =\x, gray] plot ({cos \x},{sin \x}); 
\draw [ domain=270:450 ,smooth,variable =\x, gray] plot ({cos \x},{sin \x}); 

\node[left] at ({cos 150},{sin 150}) {$A$};
\node[right] at ({cos 30},{sin 30}) {$B$};

\node[below] at ({cos 270},{sin 270}) {$A$};

\node[circle,draw=black, fill=coc, inner sep=0pt,minimum size=3pt] at ({cos 90},{sin 90}) {};
\node[above] at ({cos 90},{sin 90}) {$M$};
\node[circle,draw=black, fill=coa, inner sep=0pt,minimum size=3pt] at ({cos -30},{sin -30}) {};
\node[below right] at ({cos -30},{sin -30}) {$N$};
\node[circle,draw=black, fill=cob, inner sep=0pt,minimum size=3pt] at ({cos 210},{sin 210}) {};
\node[below left] at ({cos 210},{sin 210}) {$Q$};

\draw [ domain=90:270 ,smooth,variable =\x, gray] plot ({\shift+cos \x},{sin \x}); 
\draw [ domain=270:450 ,smooth,variable =\x, gray] plot ({\shift+cos \x},{sin \x}); 

\node[left] at ({\shift+cos 150},{sin 150}) {$A$};
\node[right] at ({\shift+cos 30},{sin 30}) {$B$};

\node[below] at ({\shift+cos 270},{sin 270}) {$B$};

\node[circle,draw=black, fill=coc, inner sep=0pt,minimum size=3pt] at ({\shift+cos 90},{sin 90}) {};
\node[above] at ({\shift+cos 90},{sin 90}) {$M$};
\node[circle,draw=black, fill=cod, inner sep=0pt,minimum size=3pt] at ({\shift+cos -30},{sin -30}) {};
\node[below right] at ({\shift+cos -30},{sin -30}) {$P$};
\node[circle,draw=black, fill=coa, inner sep=0pt,minimum size=3pt] at ({\shift+cos 210},{sin 210}) {};
\node[below left] at ({\shift+cos 210},{sin 210}) {$N$};

\draw [ domain=90:270 ,smooth,variable =\x, gray] plot ({2*\shift+cos \x},{sin \x}); 
\draw [ domain=270:450 ,smooth,variable =\x, gray] plot ({2*\shift+cos \x},{sin \x}); 

\node[right] at ({2*\shift+cos 30},{sin 30}) {$B$};

\node[below] at ({2*\shift+cos 270},{sin 270}) {$B$};

\node[circle,draw=black, fill=cod, inner sep=0pt,minimum size=3pt] at ({2*\shift+cos -30},{sin -30}) {};
\node[below right] at ({2*\shift+cos -30},{sin -30}) {$P$};

\draw[->, thick] (-2.5,0)-- (-1.5,0) node[midway, above] {$\eta\odot \id$};
\draw[->, thick] (1.5,0)-- (2.5,0) node[midway, above] {$\id\odot f$};

\draw[->, thick] (5.5,0)-- (6.5,0) node[midway, above] {$\id \odot \epsilon$};
\end{tikzpicture}}
\caption{The trace}\label{fig:picture_trace}
\end{figure}

As explained in \cite{ponto_thesis}, there is a conceptual re-interpretation of the Hattori--Stallings trace of an $A$-module endomorphism $f\colon P \to P$ as a bicategorical trace
\[
\begin{tikzcd}[row sep=small]
\bbZ \cong \HH_0(\bbZ) \ar{r}{\eta} & \HH_0(\bbZ;P^\ast \otimes_A P) \ar{r}{f} & \HH_0(\bbZ;P^\ast \otimes_A P) \ar{d}{\cong}&A/[A,A].
\\
&& \HH_0(A;P \otimes_\bbZ P^\ast) \ar{r}{\epsilon} & \HH_0(A) \ar{u}{\cong}
\end{tikzcd}
\]

This formalism is precisely what we need to generalize the classical link between the Dennis trace and the Hattori--Stallings trace so that it also applies to ring spectra.

\begin{defn}\label{defn:euler_characteristic}
  If $P\in \mc{B}(C,D)$ and $P$ is left dualizable with left dual $P^\ast$, the \textbf{Euler characteristic} of $P$ is the trace of its identity 2-cell,
  \[\chi(P)\colon \sh{U_D}\to \sh{U_C} \qquad \chi(P) \coloneqq \tr(\id_P).\]
  (Here, $M = U_C$ and $N = U_D$.)  Similarly, the Euler characteristic
  $\chi(P^\ast)$ is the trace of the identity 2-cell of $P^\ast$. A check of the
  definitions shows that $\chi(P)=\chi(P^\ast)$.
\end{defn}

As for symmetric monoidal categories, there is also a notion of invertible 1-cell that is stronger than being dualizable. It gives us a natural notion of equivalence between the 0-cells.

\begin{defn}[{\cite{benabou}\cite[Def. 4.1]{cp}}]\label{morita_equivalence}
A pair of 1-cells $P\in \mc{B}(C,D)$ and $P^\ast\in \mc{B}(D,C)$ forms a {\bf Morita equivalence} between $C$ and $D$ if $(P^\ast,P)$ is a dual pair whose coevaluation and evaluation maps are isomorphisms.
\end{defn}

\begin{example}\label{ex:morita_rings}
Morita equivalence in the bicategory of rings and bimodules is the usual notion of Morita equivalence between rings.
\end{example}

When $(P^\ast,P)$ is a Morita equivalence, the Euler characteristic $\chi(P)$ is
an isomorphism since it is a composite of isomorphisms. We will make significant use of this observation---it is an essential part of our approach to  \cref{intro_thh_morita}.

We finish this section with a definition which will be used often in this paper:
\begin{defn} \label{def:twisting} Let $\cC$ be a category, such as the category of rings and ring homomorphisms.  A {\bf pre-twisting}
  of an object $C\in \cC$ is a pair of morphisms $F\colon A \to C$ and $G\colon B \to C$;
  this is denoted $\lMr{F}{A\slash C\reflectbox{\slash}B}{G}$. When clear from context we
  often omit $A$ and $B$ from the notation.  A {\bf morphism of pre-twistings}
  $(H,I,J)\colon\lMr{F}{A\slash C\reflectbox{\slash} B}{G} \to \lMr{F'}{A'\slash C'\reflectbox{\slash}
    B}{G'}$ is a commutative diagram
  \[\xymatrix{A  \ar[r]^{F} \ar[d]_{H} & C \ar[d]^J & B \ar[l]_{G} \ar[d]^I \\
      A' \ar[r]^{F'} & C' & B' \ar[l]_{G'}
    }
  \]
  A pre-twisting is a {\bf twisting} if $A = B$; we denote a twisting by
  $\lMr{F}{A/C}{G}$, and often omit $A$ from the notation.  If
  $\lMr{F}{A/C}{G} \to \lMr{F'}{A'/C'}{G'}$ is a morphism of pre-twistings
  between twistings then it is a {\bf morphism of twistings} if $H = I$; this is
  denoted $(I,J)$.
\end{defn}

Many bimodules of interest arise from twistings in the following way.

\begin{example}
  Let $R$ be a ring, and let $f\colon S \to R$ and $g\colon T \to R$ be a pre-twisting of
  $R$.  Then $\lMr{f}{S\slash R\reflectbox{\slash} T}{g}$ gives $R$ the structure of an
  $(S,T)$-bimodule, with $S$ acting on the left through $f$ and $T$ acting on
  the right through $g$.  Morphisms of twistings produce morphisms of bimodules.
  In a similar manner, a twisting $\lMr{f}{S/R}{g}$ gives $R$ an $S$-bimodule
  structure.
\end{example}

We use this perspective in future sections (e.g. \cref{def:base_change}) to produce examples of bimodules
over spectral categories.

\subsection{Review of orthogonal $G$-spectra}\label{sec:equivariant_review}

We will also recall a bit of the theory of equivariant spectra; more details can be found in \cite{mandell_may,hhr,clmpz-specWaldCat}.

For simplicity, let $G$ be a finite abelian group, such as $C_r = \bbZ/r\bbZ$. An {\bf orthogonal $G$-spectrum} is an orthogonal spectrum with an action of $G$. By the point-set change of universe functor, this is the same thing as an orthogonal spectrum indexed on the finite-dimensional representations of $G$ \cite[V.1.5]{mandell_may}. An {\bf equivalence of $G$-spectra} is a map that induces an isomorphism on the homotopy groups that are defined using all of the $G$-representations; these are the weak equivalences in a model structure on orthogonal $G$-spectra from \cite[III.4.2]{mandell_may}, and when we say ``cofibrant'' or ``fibrant'' we are always referring to the notions coming from this model structure.

There is a categorical fixed-points functor $(-)^H$ from orthogonal $G$-spectra to orthogonal $G/H$-spectra. It is right Quillen and its right-derived functor is the \textbf{genuine fixed points} functor. There is also a point-set geometric fixed points functor $\Phi^H$ from orthogonal $G$-spectra to orthogonal $G/H$-spectra \cite[V.4.1]{mandell_may}. It is not a left adjoint, but it is left-deformable and we refer to its left-derived functor $\bbL\Phi^{H}$ as the \textbf{(left-derived) geometric fixed points} functor. On suspension spectra we have canonical isomorphisms $\Phi^H \Sigma^\infty X \cong \Sigma^\infty X^H$. There is a natural transformation $\kappa\colon X^H \to \Phi^H X$ for $G$-spectra $X$ called the \textbf{restriction map}. Making $X$ cofibrant and fibrant gives a corresponding restriction map from the genuine fixed points to the geometric fixed points.

We recall a common tool for checking that a map of $G$-spectra is an equivalence.

\begin{prop}\cite[XVI.6.4]{alaska}\label{equivs_measured_on_geometric_fp}
	A map of $G$-spectra $X \to Y$ is an equivalence if and only if for every $H \leq G$ the induced map on derived geometric fixed points $\bbL\Phi^{H}X \to \bbL\Phi^{H}Y$ is an equivalence of spectra.
\end{prop}

As discussed in \cite[4.1]{dmpsw}, the geometric fixed point functor $\Phi^H$ and its left-derived functor $\bbL\Phi^{H}$ also commute with the forgetful functor from $G$-spectra to $H$-spectra. As a result, to determine whether a map of $G$-spectra is an equivalence, it suffices to forget down to the $H$-action and measure its geometric $H$-fixed points, for each $H \leq G$.

If $X$ is an orthogonal spectrum, then the $r$-fold smash product $X^{\sma r}$ admits a canonical $C_r$-action by rotating the factors.  By the above discussion, we may then consider $X^{\sma r}$ to be an orthogonal $C_r$-spectrum. This is the {\bf Hill--Hopkins--Ravenel norm} of $X$. The following fundamental property of the norm gives us control over its equivariant homotopy type.

\begin{prop}\cite{hhr}\label{norm_diagonal}
	There is a natural diagonal map of orthogonal spectra
	\[ D_r\colon X \arr \Phi^{C_r} X^{\sma r}. \]
	If $X$ is cofibrant, then $X^{\sma r}$ is cofibrant and $D_r$ is an isomorphism on the point-set level. We therefore get a natural equivalence for cofibrant $X$,
	\[ X \simeq \bbL \Phi^{C_r} X^{\sma r}. \]
\end{prop}

In \cite{ABGHLM_TC,malkiewich_thh_dx} this result is used to build a cyclotomic structure on the topological Hochschild homology of an orthogonal ring spectrum, that by \cite{dmpsw} is equivalent to the cyclotomic structure of B\" okstedt \cite{bokstedt_thh}. In this paper we use \cref{norm_diagonal} in much the same way to control the equivariant homotopy type of the $r$-fold topological Hochschild homology spectrum $\THH^{(r)}$ (see \cref{ex:thh_n}).

\begin{rmk}\label{norm_on_suspension_spectra}
	It is especially important for us to note that on suspension spectra, the HHR norm agrees with the more obvious map
	\[ \Sigma^\infty X \cong \Sigma^\infty (X^{\sma r})^{C_r} \cong \Phi^{C_r} \Sigma^\infty X^{\sma r} \cong \Phi^{C_r} (\Sigma^\infty X)^{\sma r}. \]
	This can be checked by tracing through the definitions, but it is much easier to conclude it formally by noting that any point-set automorphism of the functor $\Sigma^\infty X$ has to be the identity when $X = S^0$ and therefore has to be the identity for all $X$. The rigidity theorem for geometric fixed points from \cite[1.2]{malkiewich_thh_dx} is a generalization of this observation.
\end{rmk}

\section{Spectral categories and spectral Waldhausen categories}\label{sec:spec_wald_cat}

In this section we establish our conventions on spectral categories, define the notion of a spectral Waldhausen category, and set up notation.  In later sections, spectral
categories will play the role that rings played in the primary example of
\S\ref{sec:bicat_trace}.

\subsection{Spectral categories}

\begin{defn}
  A {\bf spectral category} $\mc C$ is a category enriched in orthogonal
  spectra.  In other words, for every ordered pair of objects $(a, b)$ there is
  a mapping spectrum $\mc C(a,b)$, a unit map $\bbS \arr \mc C(a,a)$ from the
  sphere spectrum for every object $a$, and multiplication maps
  \[ \mc{C}(a,b) \sma \mc{C}(b,c) \arr \mc{C}(a,c) \] that are strictly
  associative and unital.  A spectral category is {\bf pointwise cofibrant} if
  all mapping spectra are cofibrant in the stable model structure on orthogonal spectra \cite[\S 9]{mandell_may_shipley_schwede}.

  A {\bf functor of spectral categories} $F\colon \mc C \arr \mc D$ consists of
  a map on the object sets and maps of spectra
  $F\colon \mc C(a,b) \arr \mc D(Fa,Fb)$ that agree with the multiplications and
  units.  Such a functor is called a {\bf Dwyer--Kan embedding} if each of these
  maps is an equivalence \cite[5.1]{blumberg_mandell_published}.

  Throughout, we assume that spectral categories are small, meaning that they
  have a set of objects.
\end{defn}

\begin{rmk}
  Our convention that $\mc C(a,b)$ is an orthogonal spectrum imposes no essential restriction. Any category enriched in symmetric or EKMM spectra can be turned into an orthogonal spectral category using the symmetric monoidal Quillen equivalences $(\bbP,\bbU)$ and $(\bbN,\bbN^\#)$ from \cite{mandell_may_shipley_schwede} and \cite{mandell_may}, respectively.
\end{rmk}

\begin{example}\nopagebreak\hfill
  \begin{enumerate}
  \item Every (orthogonal) ring spectrum $A$ is a spectral category with one object.

  \item If $\uncat\cC$ is a pointed category, then there is a spectral category $\Sigma^\infty \uncat\cC$ with the same objects as $\uncat\cC$, mapping spectra given by the suspension spectra $\Sigma^\infty \uncat\cC(a,b)$, and composition arising from $\uncat\cC$.
  \end{enumerate}
\end{example}

\begin{defn}
  A \textbf{base category} of a spectral category $\cC$ is a pair
  $(\uncat \cC, F \colon \Sigma^\infty \uncat\cC \to \cC)$ where $\uncat\cC$ is an
  pointed category and $F$ is a spectral functor that is the identity on object sets.  When the functor is clear from context we omit it
  from the notation.
\end{defn}

We can form such a base category $\uncat\cC$ by restricting each mapping
spectrum to level zero and forgetting the topology.  However,  there are also examples, such as \cref{ex:fun_cat},
which do not arise in this way.

\begin{defn}\label{modules_convention}
  Let $A$ be an orthogonal ring spectrum. The {\bf category of $A$-modules}
  $\tensor[^A]{\Mod}{}$ is a spectral category whose objects are the cofibrant
  module spectra over the ring spectrum $A$. The mapping spectra are the
  right-derived mapping spectra.  When $A$ is clear from context we omit it from
  the notation.
\end{defn}

Perfect modules also form a spectral category.
\begin{example}\label{ex:perfect_modules}

  For a ring spectrum $A$, the category of {\bf perfect $A$-modules}
  $\tensor[^A]{\Perf}{}$ is the full subcategory of $\tensor[^A]{\Mod}{}$
  spanned by the modules that are retracts in the homotopy category of finite
  cell $A$-modules. When $A$ is understood, we call this $\Perf$. There is a functor of spectral categories
  $A \to \tensor[^A]{\Perf}{}$ taking $A$ to $A$ equipped with the
  left-multiplication action.
\end{example}

In both of these cases, we make the mapping spectra right-derived by passing through the category of EKMM spectra, using the symmetric monoidal Quillen adjunction $(\bbN, \bbN^\#)$ from \cite[I.1.1]{mandell_may}. A more explicit treatment appears in \cite[\swcref{Section 3.2}{\cref{swc-sec:spec_mod}}]{clmpz-specWaldCat}.

Another common example of a spectral category is a functor category.

\begin{example} \label{ex:fun_cat}\label{thm:moore_end}
  Let $I$ be a small category, and let $\cC$ be a spectral category with a
  base category $\uncat\cC$. Write $\Fun(I,\uncat\cC)$ for the {\bf category of
  functors} (and natural transformations) $I \to \uncat\cC$. The category
  $\Fun(I, \uncat \cC)$ is a base category of a spectral
  category $\Fun(I, \spcat\cC)$, whose mapping spectra are right-derived from the equalizer
  \[
  \operatorname{eq}\Big( \prod_{i_0 \in \ob I} \spcat{\cC}(\phi(i_0),\gamma(i_0))
    \rightrightarrows \prod_{i_0 \arr i_1}
    \spcat{\cC}(\phi(i_0),\gamma(i_1))\Big).
  \]
  To be more precise, the spectral category $\spcat{\Fun(I,\cC)}$ is defined using the Moore end construction of \cite[2.4]{mcclure_smith} and \cite[2.3]{blumberg_mandell_unpublished}. An explicit and detailed treatment of this construction for spectra appears in \cite[\swcref{Section 4}{\cref{swc-ex:fun_cat}}]{clmpz-specWaldCat}.
\end{example}

Many of our techniques will require the use of
pointwise cofibrant spectral categories. Spectral categories can
always be replaced with equivalent pointwise cofibrant spectral categories using the model structure from \cite[6.1,~6.3]{schwede_shipley}.

\begin{thm} \label{thm:cofib_repl_spcat}
	There is a pointwise cofibrant replacement functor $Q$ and a pointwise fibrant replacement functor $R$ on spectral categories. In particular,
	\[Q\colon \mathbf{SpCat} \to \mathbf{SpCat}\]
	is a functor equipped with a natural transformation $q\colon Q \Rightarrow \id_{\mathbf{SpCat}}$ such that $q_{\cC}$ is a pointwise equivalence for every spectral category $\cC$.
\end{thm}

\subsection{Spectral Waldhausen categories}\label{ssec:sp_wald}

We now extend the definition of Waldhausen categories to spectral categories. Recall that a Waldhausen category $\uncat\cC$ is a category with cofibrations and weak equivalences satisfying the axioms in \cite[\S 1.2]{1126}.

\begin{defn}\label{def:spectrally_enriched_waldhausen_category}
  A \textbf{spectral Waldhausen category} is a spectral category $\cC$
  together with a base category $\uncat\cC$ which is equipped with a Waldhausen category structure. This data is subject to the following three conditions:
\begin{enumerate}
	\item\label{def:spectrally_enriched_waldhausen_category_zero} The zero object of $\uncat{\cC}$ is also a zero object for $\spcat{\cC}$.
	\item\label{def:spectrally_enriched_waldhausen_category_we} Every weak equivalence $c \arr c'$ in $\uncat{\cC}$ induces stable equivalences
	\[ \spcat{\cC}(c',d) \overset\sim\arr \spcat{\cC}(c,d), \qquad \spcat{\cC}(d,c) \overset\sim\arr \spcat{\cC}(d,c'). \]
	\item\label{def:spectrally_enriched_waldhausen_category_pushout} For every pushout square in $\uncat{\cC}$ along a cofibration
	\[ \xymatrix @R=1.5em{
		a \, \ar[d] \ar@{^(->}[r] & b \ar[d] \\
		c \,\ar@{^(->}[r] & d
	} \]
	and object $e$, the resulting two squares of spectra
	\[ \xymatrix @R=1.5em{
		\spcat{\cC}(a,e) \ar@{<-}[d] \ar@{<-}[r] & \spcat{\cC}(b,e) \ar@{<-}[d] \\
		\spcat{\cC}(c,e) \ar@{<-}[r] & \spcat{\cC}(d,e)
	}
	\qquad
	\xymatrix @R=1.5em{
		\spcat{\cC}(e,a) \ar[d] \ar[r] & \spcat{\cC}(e,b) \ar[d] \\
		\spcat{\cC}(e,c) \ar[r] & \spcat{\cC}(e,d)
	} \]
	are homotopy pushout squares.
      \end{enumerate}

      A \textbf{functor of spectral Waldhausen categories}
      $F\colon(\spcat{\cC},\uncat{\cC}) \arr (\spcat{\cD},\uncat{\cD})$ is an exact functor
      $\uncat{F}\colon \uncat{\cC} \arr \uncat{\cD}$ and a spectral functor
      $\spcat{F}\colon \spcat{\cC} \arr \spcat{\cD}$ such that the diagram
      \[\xymatrix @R=1.5em{
          \Sigma^\infty \uncat{\cC} \ar[r]^{\Sigma^\infty \uncat{F}} \ar[d] & \Sigma^\infty \uncat{\cD}
          \ar[d] \\
          \spcat{\cC} \ar[r]^{\spcat{F}}& \spcat{\cD}
      }\]
      commutes.
      When it is clear from context, we omit $\uncat\cC$ from the notation and refer simply to the spectral Waldhausen category $\cC$.
    \end{defn}

\begin{example}
	\label{ex:Perf_as_spectral_wald_cat}
	The categories $\tensor[^A]{\Perf}{}$ of perfect $A$-modules and
        $\tensor[^A]{\Mod}{}$ of all $A$-modules are both spectral Waldhausen
        categories.
\end{example}

\begin{example} \label{ex:BM} If $\uncat{\cC}$ is a simplicially enriched
	Waldhausen category in the sense of \cite{blumberg_mandell_unpublished} then
	the spectral enrichment $\cC^\Gamma$ from
	\cite[2.2.1]{blumberg_mandell_unpublished} is compatible with the Waldhausen
	structure in our sense. The same is true for the non-connective enrichment
	$\cC^{\mc S}$ from \cite[2.2.5]{blumberg_mandell_unpublished} if $\uncat{\cC}$
	is enhanced simplicially enriched.
\end{example}

\begin{prop}[{\cite[\swcref{Theorem 4.1}{\cref{swc-thm:moore_end}}]{clmpz-specWaldCat}}]
	\label{prop:fun_cat_wald}
	The category of functors construction from \cref{thm:moore_end} respects Waldhausen structures and  defines a functor
	\[\Fun\colon \mathbf{Cat}^\op \times \mathbf{SpWaldCat} \to \mathbf{SpWaldCat},\]
	by giving $\Fun(I, \uncat\cC)$ the levelwise Waldhausen structure.
\end{prop}

By levelwise Waldhausen structure, we mean that a map of diagrams $\phi \to \gamma$ is a cofibration (resp. weak equivalence) if $\phi(i) \to \gamma(i)$ is a cofibration (resp. weak equivalence) for every object $i$ of the indexing category $I$. In practice, this is usually more cofibrations than we need, but we can always restrict the class of cofibrations:
\begin{lem} \label{lem:restrict_cofib}
	If $(\spcat\cC,\uncat\cC)$ is a spectral Waldhausen category, and $\uncat\cC'$ is a different Waldhausen structure on $\uncat\cC$ with the same weak equivalences and fewer cofibrations, then $(\spcat\cC, \uncat\cC')$ is also a
	spectral Waldhausen category.
\end{lem}

\subsection{The $S_\bullet$ construction and the $K$-theory of spectral Waldhausen categories}\label{subsec:Sdot}

Let $[k] = \{ 0 < 1 < \cdots < k \}$ denote the totally ordered set on $k+1$ elements. Recall that for a Waldhausen category $\uncat\cC$, the $S_\bullet$ construction produces a simplicial category whose $k$th level is the full subcategory
\[ S_k \uncat\cC \subseteq \Fun([k] \times [k],\uncat\cC) \]
consisting of functors that vanish on all pairs $(i,j)$ with $i > j$, and on the remaining pairs (in other words the category of arrows $\Arr[k]$) form a sequence of cofibrations and their quotients \cite[\S 1.3]{1126}.

We extend this definition to spectral Waldhausen categories $\spcat\cC$ by defining $S_k \spcat\cC$ to be the full subcategory of $\spcat{\Fun([k] \times [k],\cC)}$ on the objects that define $S_k \uncat\cC$. We define the $n$-fold $S_\bullet$ construction and the simplicial category $w_\bullet$ of composable sequences of weak equivalences for spectral Waldhausen categories in a similar way; see
\cite[\swcref{Definitions 5.6 and 5.10}{\cref{swc-ex:spec_iterated_s_dot,swc-ex:spec_w_bullet}}]
{clmpz-specWaldCat}
for more details.

\begin{lem}[{\cite[\swcref{Definition 5.10}{\cref{swc-ex:spec_iterated_s_dot}}]{clmpz-specWaldCat}}]
	\label{s_dot_multisimplicial}
  For every $n \geq 0$, there is a functor from spectral Waldhausen categories to multisimplicial Waldhausen categories
  \[w_\bullet S^{(n)}_\bullet\colon \mathbf{SpWaldCat} \arr \Fun(\Delta^{n+1}, \mathbf{SpWaldCat})\]
  which on base categories takes $\uncat{\cC}$ to $w_\bullet S^{(n)}_\bullet \uncat{\cC}$, as defined by Waldhausen \cite{1126}.
\end{lem}

For the next lemma, recall that a {\bf Dwyer--Kan equivalence} is a Dwyer--Kan embedding of spectral categories that induces an equivalence of ordinary categories after applying $\pi_0$ to the mapping spectra.

\begin{lem}[{\cite
	[\swcref{Lemma 5.7}{\cref{swc-w_bullet_invariance}}]
	{clmpz-specWaldCat}}]\label{w_bullet_invariance}
	The iterated degeneracy map
  \[w_0S_{k_1,\ldots,k_n}\spcat{\cC} \arr w_kS_{k_1,\ldots,k_n}\spcat{\cC}\]
  is a Dwyer--Kan equivalence of spectral categories.  In particular,
	the spectral categories $w_k\spcat{\cC}$ are all canonically Dwyer--Kan
	equivalent to $\spcat{\cC}$.
\end{lem}

\begin{defn}
We define the $K$-theory of a spectral Waldhausen category $\cC$ to be the
$K$-theory of the base category $\uncat\cC$.  In other words, the $n$th level of the $K$-theory spectrum is obtained from the spectral category $w_\bullet S^{(n)}_\bullet\spcat{\cC}$ by restricting to objects and taking the geometric realization:
\[
K(\cC)_{n} = \bigl| \ob w_{\bullet} S^{(n)}_{\bullet} \spcat{\cC} \bigr|.
\]
\end{defn}

Note that, as usual, the $K$-theory spectrum is a symmetric spectrum with the symmetric groups permuting the  $S_{\bullet}$ terms.

\section{Bimodules over spectral categories and their traces}
\label{sec:spectral_bimodules}

In this section, we define the bicategory of spectral categories and bimodules, which is the relevant generalization of the bicategory of rings and bimodules from \S\ref{sec:bicat_trace}.
This bicategory can be equipped with a shadow via $\THH$, and we use the notion of Morita equivalence from \cref{lem:base_change_euler} to
construct examples of equivalences on $\THH$.

\subsection{Spectral bimodules}\label{subsec:spectral_bimodules}
\begin{defn} \label{def:bimodules}
  If $\mc{C}$ and $\mc{D}$ are spectral categories, a
  \textbf{$(\mc{C}, \mc{D})$-bimodule} is a spectral functor
  \[\cX\colon  \mc{C}^{\text{op}} \sma \mc{D} \to \Sp\]
  to the spectral category of orthogonal spectra. More explicitly, a bimodule $\cX$ consists of an orthogonal spectrum $\cX(c,d)$ for every ordered pair $(c,d) \in \ob\mc{C} \times \ob\mc{D}$, along with a {\bf left action} by $\cC$ and a {\bf right action} by $\cD$
  \[ \mc{C}(a,c) \sma \cX(c,d) \arr \cX(a,d), \qquad \cX(c,d) \sma
    \mc{D}(d,e) \arr \cX(c,e) \] satisfying evident unit and associativity
  conditions.  A {\bf morphism} of
  $(\mc{C}, \mc{D})$-bimodules $\cX \to \cY$ is a collection of maps of orthogonal spectra
  $\cX(c, d) \to \cY(c, d)$ commuting with the $\mc{C}$ and $\mc{D}$
  actions.
  A {\bf pointwise equivalence} of bimodules is a morphism which induces weak
  equivalences of spectra $\cX(c,d) \stackrel\sim\arr \cY(c,d)$ for all objects
  $c\in \mc C$ and $d\in\mc D$.
  We denote the category of spectral $(\mc C,\mc D)$-bimodules by $\Mod_{(\mc C,\mc D)}$.
\end{defn}

\begin{example}
If $A$ and $B$ are ring spectra, then a $(B, A)$-bimodule is the same thing as a bimodule over the associated one-object spectral categories.
\end{example}

\begin{example}
  Let $\cC$ be a spectral category.  Then $\cC$ gives rise to a
  $(\cC,\cC)$-bimodule defined by the enrichment functor
  $\Hom\colon \cC^\op \sma \cC \to \Sp$.  By an abuse of notation we denote this
  bimodule by $\cC$.
\end{example}

We can further generalize this example by allowing different sources for the domains
and codomains of the mapping spectra.

\begin{defn} \label{def:base_change} Recall from \cref{def:twisting} that a pre-twisting of spectral categories
  $\lMr{F}{\cA\slash\cC\reflectbox{\slash}\cB}{G}$ is a pair of spectral functors $F\colon \cA \to \cC$ and $G\colon \cB \to \cC$. Given a pre-twisting, we define an $(\mc{A}, \mc{B})$-bimodule
  (which, by an abuse of notation, we denote by $\tensor[_F]{\mc{C}}{_G}$) by
	\[
	\tensor[_F]{\mc{C}}{_G} (a, b) \coloneqq \mc{C}(F(a), G(b)).
	\]
We have the following special cases:
\begin{itemize}
	\item the $(\mc{A},\mc{C})$-bimodule $\tensor[_F]{\mc{C}}{_\id}$,
          abbreviated by $\tensor[_F]{\mc{C}}{}$ or $\tensor[_{\mc A}]{\mc C}{}$
          (when $F$ is clear from context).
	\item the $(\mc{C},\mc{B})$-bimodule $\tensor[_\id]{\mc{C}}{_G}$,
          abbreviated by $\tensor{\mc{C}}{_G}$ or $\tensor[]{\mc C}{_{\mc B}}$ (when $G$
          is clear from context).
	\item the $(\mc{C},\mc{C})$-bimodule $\tensor[_\id]{\mc{C}}{_\id}$,
          abbreviated by $\mc{C}$.  Note that this agrees with the use of
          $\cC$ as a bimodule above.
        \end{itemize}
      \end{defn}

\begin{example}\label{ex:bimodule_smash_functor}
  Let $A$ and $B$ be ring spectra and let $M$ be a cofibrant
  $(B,A)$-bimodule. As discussed in
	\cite[\swcref{Definition 3.8}{\cref{swc-modules_convention}}]
	{clmpz-specWaldCat}, $M$ induces a
  spectral functor on the categories of modules
  ${M \sma_A -}\colon \tensor[^A]{\Mod}{}  \to \tensor[^B]{\Mod}{},$ and thus defines a $(\tensor[^B]{\Mod}{},\tensor[^A]{\Mod}{})$-bimodule
  $(\tensor[^B]{\Mod}{})_{M \sma_A -}$. We will often abbreviate this to
  $\Mod_M$ if the rings are understood. By
	\cite[\swcref{Lemma 3.15}{\cref{swc-commuting_bimodules}}]
	{clmpz-specWaldCat}, there is a natural equivalence of $(B,A)$-bimodules $M \to \Mod_M$ given by the map $M \to \tensor[^B]{\Mod}{}(B,M)$ adjoint to the $B$-action on $M$.  Note that the target is implicitly restricted to a $(B, A)$-bimodule along the canonical
  inclusions of spectral categories $A \to \tensor[^A]{\Mod}{}$ and
  $B \to \tensor[^B]{\Mod}{}$.
\end{example}

The bar construction provides a canonical model for the derived coend of bimodules over spectral categories; this is the appropriate analog in our context of the derived tensor product of bimodules over rings.

\begin{defn}
  Let $\cX$ be a $(\mc C, \mc D)$-bimodule and let $\cY$ be a
  $(\mc D, \mc E)$-bimodule.  Define the {\bf two-sided categorical bar
    construction}
  $B(\cX, \mc D, \cY)$ to be the $(\mc C, \mc E)$-bimodule whose value at
  $(c,e)$ is the geometric realization of the simplicial spectrum
  $B_\bullet (\cX, \mc D, \cY)(c,e)$ given by
  \[ B_n (\cX, \mc{D}, \cY)(c,e) = \bigvee_{d_0, \dots, d_n}
    \cX(c, d_0) \sma \mc{D}(d_0, d_1) \sma \cdots \sma \mc{d}(d_{n-1}, d_n)
    \sma \cY(d_n, e). \]
    As is usual for bar constructions, the iterated $\cD$-action maps define canonical pointwise equivalences
  \begin{equation}\label{eq:units}
    B(\mc D,\mc D,\cY)(d,e) \stackrel\sim\arr {\cY}(d,e) \qquad \text{and} \qquad
  B(\cX,\mc D,\mc D)(c,d) \stackrel\sim\arr {\cX}(c,d).
\end{equation}

  When $\cX$ is a $(\mc C,\mc C)$-bimodule we define the {\bf
    topological Hochschild homology} or {\bf cyclic bar construction} $\THH(\mc C; \cX)$ to be the realization
  of the simplicial spectrum $B^{\mathrm{cy}}_\bullet(\mc C; \cX)$ given by
  \[
    B^{\mathrm{cy}}_n (\mc{C}; \cX) \coloneqq \bigvee_{c_0, \dots, c_n} \mc{C}(c_0, c_1) \sma \mc{C}(c_1, c_2) \sma \cdots \sma \mc{C}(c_{n-1}, c_n) \sma \cX(c_n, c_0).
  \]
  When $\mc C$ is pointwise cofibrant, the definition is equivalent to all
  other definitions of $\THH$ in the literature,
  e.g.
  \cite{bokstedt_thh,blumberg_mandell_unpublished,dundas_goodwillie_mccarthy,nikolaus_scholze}.
\end{defn}

Note that $\THH(\mc C; \cX)$ is functorial in both $\cC$ and $\cX$.  Directly from the definition we get the following
observation:
\begin{lem}
  A morphism of twistings $\lMr{F}{\cA/\cC}{G} \to \lMr{F'}{\cA'/\cC'}{G'}$ induces a
  morphism
  \[\THH(\cA; \lMr{F}{\cC}{G}) \to \THH(\cA'; \lMr{F'}{\cC'}{G'}).\]
\end{lem}

We equip the category $\Mod_{(\mc C,\mc D)}$ of $(\mc C,\mc D)$-bimodules with the model structure discussed in {\cite[2.4-2.8]{blumberg_mandell_published}}, in which the weak equivalences are the pointwise equivalences.

Bimodules over spectral categories form a bicategory with shadow, and this structure is the foundation of all our work in this paper.  It echoes the structure of the bicategory of rings and bimodules from
Example~\ref{ex:bicat_rings}. There are several previous constructions in the literature, for instance \cite[22.11]{shulman_hocolim}, \cite[4.13]{lm}, \cite[2.17]{cp},  \cite[7.4.2]{malkiewich_parametrized}.

\begin{defn} \label{main_homotopy_bicategory}
  Let $\mathbf{Bimod(SpCat)}$ be the bicategory with
  \begin{itemize}
  \item[] {\bf 0-cells} the pointwise cofibrant spectral categories $\mc C$,
  \item[] {\bf 1- and 2-cells} the objects and morphisms in the homotopy categories $\operatorname{Ho}\left(\operatorname{Mod}_{(\mc{C}, \mc{D})}\right)$,
  \item[] and {\bf horizontal composition of 1-cells} $\cX$, a
    $(\cC,\cD)$-bimodule, and $\cY$, a $(\cD, \cE)$-bimodule, given by the bar
    construction
    \[ \cX \odot \cY \coloneqq B(\cX, \mc{D}, \cY). \]
  \end{itemize}
  The bimodules $\mc{D} = \tensor[_\id]{\mc{D}}{_\id}$ are the
  units for the horizontal composition, with unit isomorphisms given by \eqref{eq:units}.
  The associativity of the horizontal composition follows from a comparison of
  bisimplicial spectra.
  We equip the bicategory $\mathbf{Bimod(SpCat)}$ with a shadow using topological Hochschild homology:
    \[
      \sh{\cX} \coloneqq  \THH(\mc{C}; \cX).
    \]
\end{defn}

The horizontal composition of 1-cells is compatible with the
bimodule structures given by functors:
\begin{lem}\label{base_change_composition}
  For any twisting $\lMr{F}{\cC}{G}$ there is a canonical isomorphism
  of 1-cells
  \[\lMr{F}{\cC}{G} \simeq \lMr{F}{\cC}{} \odot \lMr{}{\cC}{G}.\]
  For composable spectral functors $\mc{A} \to \mc{B} \to \mc{C}$, there are
  canonical isomorphisms of 1-cells
  \begin{equation*} \tensor[_{\mc A}]{\mc
      B}{} \odot \tensor[_{\mc B}]{\mc C}{} \simeq \tensor[_{\mc A}]{\mc
            C}{} \qquad \tensor{\mc C}{_{\mc B}} \odot \tensor{\mc B}{_{\mc A}} \simeq
          \tensor{\mc C}{_{\mc A}}. \end{equation*}
\end{lem}

Our examples of 1-cells also give simple ways to construct dual pairs:
\begin{prop}[{\cite[Appendix]{ponto_thesis}}, {\cite[Lem.~7.6]{ps:indexed}}] \label{prop:dual_pair}
    Let $F\colon \mc A \to \mc C$ be a spectral functor.  Then there is a dual
  pair $(\tensor[_F]{\mc C}{},\tensor{\mc C}{_F})$ whose coevaluation and
  evaluation maps are induced by the composites
  \[\eta \colon \cA(a,b) \xto{\ \ F\ \ } \cC(Fa,Fb) \overset{\sim}{\arr} B(\cC,\cC,\cC)(Fa,Fb) =
    (\lMr{F}{\cC}{} \odot \lMr{}{\cC}{F})(a, b)\]
  and
  \[\epsilon \colon (\lMr{}{\cC}{F} \odot \lMr{F}{\cC}{})(c,d) = B(\lMr{}{\cC}{F},\mc A,\lMr{F}{\cC}{})(c,d) \xto{\ \ F\ \ } B(\cC,\cC,\cC)(c,d) \xto{\ \sim\ } \cC(c,d).\]
\end{prop}

\subsection{Bicategorical traces and $\THH$}
In this subsection we show how familiar maps on $\THH$ can be described as
traces of endomorphisms of bimodules.   The primary results are well-known, but our explicit use of the shadow structure on $\THH$ simplifies and clarifies previous proofs, e.g. in the work of Blumberg and Mandell  \cite{blumberg_mandell_published}.

 \begin{prop}[{\cite[5.8]{cp}}]\label{lem:base_change_euler}
   Let $F\colon\cA \to \cC$ be a spectral functor.
   \begin{enumerate}
    \item Given a $(\mc C,\mc C)$-bimodule $\cX$, write
        $\tensor[_F]{\cX}{_F} \coloneqq \lMr{F}{\cC}{} \odot \cX \odot
        \lMr{}{\cC}{F}$.  Then the trace of the map
  \[
  \lMr{F}{\cX}{F} \odot \lMr{F}{\cC}{} = \lMr{F}{\cC}{} \odot \cX \odot \lMr{}{\cC}{F} \odot \lMr{F}{\cC}{} \xarr{\id \odot \epsilon} \lMr{F}{\cC}{} \odot \cX,
  \]
taken with respect to the dual pair $(\tensor[_F]{\mc C}{},\tensor{\mc C}{_F})$,
is the map
\[ \THH(F)\colon \THH(\mc A;\tensor[_F]{\cX}{_F})\arr \THH(\mc C;\cX) \]
induced by $F$ on the cyclic bar construction.
\item The Euler
characteristic $\chi(\lMr{}{\cC}{F})$ of the left dualizable 1-cell $\lMr{}{\cC}{F}$ (resp. the right dualizable 1-cell
$\tensor[_F]{\mc C}{}$) is the induced
map
\[ \THH(F) \colon \THH(\mc A) \arr \THH(\mc C). \]
\end{enumerate}
\end{prop}

\begin{defn}
	A spectral functor $F\colon \cA \to \cC$ is a \textbf{Morita equivalence} if the dual pair
	$(\lMr{F}{\cC}{}, \tensor{\mc C}{_F})$ is a Morita equivalence, in the sense of \cref{morita_equivalence}.
\end{defn}

\begin{lem}[{\cite[5.12]{blumberg_mandell_published}}] If $F$ is a Dwyer--Kan embedding and surjective up to thick closure, then $F$ is a Morita equivalence.
\end{lem}

The condition of being surjective up to thick closure means that the representable functors $\cC(c,-)$ can be obtained from the representable functors $\cC(Fa,-)$ for $a \in \ob\cA$ by cofiber sequences and retracts, and similarly on the other side $\cC(-,c)$.

\cref{lem:base_change_euler} implies that when $F\colon\cA \to \cC$ is a Morita
equivalence, the induced map $\THH(F)$ is an equivalence.
The next result follows immediately:
\begin{thm}[See {\cite[5.9]{cp} and
    \cite[5.12]{blumberg_mandell_published}}] \label{cor:base_change_morita_isos}\label{morita_invariance_thh}
  Let $F\colon \mc A\to \mc C$ be a map of pointwise cofibrant spectral
  categories and $\cX$ be a $(\mc{C},\mc{C})$-bimodule. If $F$ is a
  Dwyer--Kan embedding and surjective up to thick closure, then the induced maps of spectra
  \[ \THH(F) \colon \THH(\mc{A})\arr \THH(\mc{C}), \qquad \THH(F) \colon  \THH(\mc A;\tensor[_F]{\cX}{_F})\arr \THH(\mc C;\cX) \]
  are equivalences.
\end{thm}

We can use the theorem to show that $\THH(A)$ is equivalent to $\THH(\tensor[^A]{\Perf}{})$.

\begin{example}\label{ex:classic_morita_equiv}
  If $A$ is a ring spectrum, then the inclusion of spectral categories
  $A \arr \tensor[^A]{\Perf}{}$ is a Dwyer--Kan embedding, and surjective up to thick closure essentially by the definition of $\tensor[^A]{\Perf}{}$.
  Thus
  \[ \THH(A) \stackrel\sim\arr \THH(\tensor[^A]{\Perf}{}). \]
  Furthermore, if $M$ is an $(A,A)$-bimodule, then along the map $A \to \tensor[^A]{\Perf}{}$ we have an equivalence of bimodules $M \to \tensor[^A]{\Mod}{_M}$ from \cref{ex:bimodule_smash_functor}, and hence an equivalence
  \begin{equation*}
    \THH(A;M) \stackrel\sim\arr \THH(\tensor[^A]{\Perf}{};\tensor[^A]{\Mod}{_M}).
  \end{equation*}
\end{example}

\begin{example}\label{ex:w_bullet_THH}
  In \cref{w_bullet_invariance} we saw that the iterated degeneracy maps
  $w_0\cC \to w_k\cC$ for a spectral Waldhausen
  category $\cC$ are Dwyer--Kan equivalences, and thus induce equivalences
  on $\THH$.
\end{example}

 \begin{example}
   Let $\cC$ and $\cD$ be spectral categories with chosen zero object,
   let $\cC \times \cD$ be their product in spectral categories, and let
   $\cC \vee \cD \subset \cC \times \cD$ be the full subcategory spanned by the
   pairs in which at least one coordinate is the zero object. Then the inclusion
   $\cC \vee \cD \arr \cC \times \cD$ is both a Dwyer--Kan embedding and
   surjective up to thick closure. Therefore we have an equivalence
 	\[
  \THH(\cC) \vee \THH(\cD) \stackrel\cong\arr \THH(\cC \vee \cD) \stackrel\sim\arr \THH(\cC \times \cD).
  \]
 	This gives a short proof of the nontrivial fact that $\THH$ preserves finite products.
 \end{example}

We can further strengthen \cref{lem:base_change_euler} if $\cX$ is of the form $\lMr{F}{\cD}{G}$.
\begin{prop} \label{base_change_beck_chevalley}
  Given a morphism of twistings
  $(I,J)\colon \lMr{F}{\cC/\cD}{G} \to \lMr{F'}{\cC'/\cD'}{G'}$
  there are induced maps\footnote{These are examples of \textbf{Beck--Chevalley maps}.}
  \begin{equation*}
    \lMr{F}{\cD}{G} \odot \lMr{I}{\cC'}{} \to
    \lMr{I}{\cC'}{} \odot \lMr{F'}{\cD'}{G'} \qqand
    \lMr{}{\cC'}{I} \odot \lMr{F}{\cD}{G}   \to
     \lMr{F'}{\cD'}{G'}  \odot\lMr{}{\cC'}{I}
  \end{equation*}
  whose traces are both the induced map
  \begin{equation*} \THH(I;J) \colon
    \THH(\cC;\tensor[_F]{\cD}{_G}) \arr
    \THH(\cC';\tensor[_{F'}]{{\cD'}}{_{G'}}) \end{equation*}
\end{prop}

\begin{proof}
  Write $\widehat F \coloneqq JF$ and $\widehat G \coloneqq JG$, and observe that $\widehat F = F'I$ and $\widehat G = G'I$ by the commutativity of the diagram defining a morphism of twistings (\cref{def:twisting}).  We prove the proposition for
  the first map; the second is proved analogously.

  The desired map is the composite
  \begin{align*}
    \lMr{F}{\cD}{} \odot \lMr{}{\cD}{G} \odot \lMr{\cC}{\cC'}{} & \xarr{\id\odot
      \eta \odot \id} \lMr{F}{\cD}{} \odot \lMr{\cD}{\cD'}{} \odot
                                                                  \lMr{}{\cD'}{\cD} \odot \lMr{}{\cD}{G} \odot \lMr{\cC}{\cC'}{} \\
    & \oarr{\sim} \lMr{\widehat F}{\cD'}{\widehat G}
     \oarr{\sim} \lMr{\cC}{\cC'}{} \odot
      \lMr{F'}{\cD'}{} \odot \lMr{}{\cD'}{G'} \odot \lMr{}{\cC'}{\cC} \odot
      \lMr{\cC}{\cC'}{}
     \xarr{\id\odot \epsilon} \lMr{\cC}{\cC'}{} \odot \lMr{F'}{\cD'}{} \odot \lMr{}{\cD'}{G'},
  \end{align*}
  where the two isomorphisms of 1-cells are obtained using \cref{base_change_composition}.
  This map is now in a form where its trace agrees with the right-hand side of the
  equation in \cite[Proposition 7.1]{ps:bicat}, with $M = \lMr{\cC}{\cC'}{}$.
  Applying the proposition and simplifying implies that the
  trace of this map is
\begin{align*}
\sh{\lMr{F}{\cD}{G} }
	&\xarr{\ \sh{\id\odot \eta \odot \id}\ }
\sh{\lMr{\widehat F}{\cD'}{\widehat G}}
\xto{\ \tr(\id \odot \epsilon)\ } \sh{ \lMr{F'}{\cD'}{G'}}.
\end{align*}
Applying the definition of the shadow $\sh{-}$ on the bicategory of spectral categories and using \cref{lem:base_change_euler} with $\cX = \tensor[_{F'}]{{\cD'}}{_{G'}}$, the composite is
 \begin{equation*}
	\THH(\cC;\tensor[_F]{\cD}{_G}) \xarr{\THH(\id_{\cC}; J)}
        \THH(\cC;\tensor[_{\widehat F}]{{\cD'}}{_{\widehat G}}) \xarr{\THH(I)} \THH(\cC';\tensor[_{F'}]{{\cD'}}{_{G'}}),
\end{equation*}
where the first map applies $J$ to mapping spectra in $\cD$ and the second is the map induced by $I$.
\end{proof}

Putting \cref{morita_invariance_thh} and \cref{base_change_beck_chevalley} together gives the
following:

\begin{cor}\label{equivalence_on_twisted_thh}
  Let $(I,J)\colon \lMr{F}{\cC/\cD}{G} \to \lMr{F'}{\cC'/\cD'}{G'}$ be a morphism of
  twistings where $I$ and $J$ are Dwyer--Kan embeddings and $I$ is surjective up to thick closure.  Then the induced map
  \begin{equation*} \THH(I;J) \colon \THH(\cC;\tensor[_F]{\cD}{_G}) \arr
    \THH(\cC';\tensor[_{F'}]{{\cD'}}{_{G'}}) \end{equation*} is an equivalence.
\end{cor}

\section{The additivity theorem for $\THH$, revisited}\label{sec:add}

\subsection{Additivity without coefficients}

In this section we prove:
\begin{thm}\label{cor:small_add}
  There is an equivalence of spectra
  \[ (\iota_j)_{j = 1}^{k} \colon \ \bigvee_{j=1}^k \THH(\spcat{\cC}) \stackrel\sim\arr \THH(S_k\spcat{\cC}). \]
\end{thm}
This is similar in spirit to existing additivity results, such as \cite[1.6.20]{dundas_mccarthy} and \cite[IV.2.5.8]{dundas_goodwillie_mccarthy} which use a category of upper-triangular matrices,
\cite[Thm.~10.8]{blumberg_mandell_published} which proves a version for DG-categories,
and \cite[3.1.1]{blumberg_mandell_unpublished} which proves additivity for $\textup{WTHH}(\cC) \coloneqq \THH(S_\bullet \cC)$, in other words after one copy of $S_\bullet$ has been applied.

Our approach to \cref{cor:small_add} is fundamentally an adaptation of a technique from \cite[\S 7]{blumberg_mandell_published}, made more conceptual by the machinery of shadows.

\begin{defn}
  Let $\cC$ be a spectral Waldhausen category.  By
	\cite[\swcref{Theorem 4.1}{\cref{swc-thm:moore_end}}]
	{clmpz-specWaldCat}
  there is a canonical equivalence $S_1\cC \simeq \cC$.  Let $s_{k-1}\colon S_{k-1}\cC \to S_k \cC$ and
  $d_k\colon S_k \cC \to S_{k-1}\cC$ be the last degeneracy and face functors,
  respectively. Let $\pi_k\colon S_k \cC \to S_1\cC \stackrel\sim\to \cC$ be the
  induced by $d_0^{k-1}$ and let $\iota_k\colon \cC \stackrel\sim\to S_1 \cC \to S_k\cC$
  be induced by $s_0^{k-1}$.  More generally, for $1 \leq j \leq k$ write $\iota_j\colon \cC \to S_k\cC$
  for the functor induced by $s_0^{j-1}s_1^{k-j}$; these are the functors inducing the equivalence in \cref{cor:small_add}.
\end{defn}

The next proposition is the main ingredient for the proof of the additivity
theorem. As the proof is technical, we postpone it until
\S\ref{sec:important_dual_pair}.

\begin{prop}\label{sdot_adjunctions}
  The coevaluation map of the dual pair
  $(\lMr{s_{k-1}}{(S_k\cC)}{},\bcl{}{s_{k-1}}{(S_k\cC)})$
  and the evaluation map of the dual pair
  $(\lMr{\pi_k}{\cC}{},\bcl{}{\pi_k}{\cC})$ are pointwise
  equivalences of bimodules.  The other evaluation map and coevaluation map
  induce a homotopy cofiber sequence of $(S_k\spcat\cC,S_k\spcat\cC)$-bimodules
  \[ \lMr{}{(S_k\cC)}{s_{k-1}} \odot \lMr{s_{k-1}}{(S_k\cC)}{} \xarr{\ \epsilon \ } S_k\cC \xarr{\ \eta \ }
    \lMr{\pi_k}{\cC}{} \odot \lMr{}{\cC}{\pi_k}.\]
\end{prop}

Theorem~\ref{cor:small_add} follows by induction from the following proposition:

\begin{prop}\label{prop:small_add}
The spectral functors $s_{k-1}$ and $\iota_k$ induce an equivalence
\[ \THH(S_{k-1}\spcat{\cC}) \vee \THH(\spcat{\cC}) \stackrel\sim\arr \THH(S_k\spcat{\cC}). \]
\end{prop}

\cref{fig:pic_shadows_add} is the version of \cref{fig:shadow_single} for \cref{cor:small_add}.

\begin{figure}[h]
     \centering
	\hspace{1in}
     \begin{subfigure}[b]{0.25\textwidth}
         \centering
	\tdplotsetmaincoords{90}{70}
         \begin{tikzpicture}[tdplot_main_coords]

\def\ra{1}
\def\aa{10}
\def\a1{(\aa+60)}
\def\ab{(\a1+50)}
\def\ac{(\ab+70)}
\def\ad{(\ac+60)}
\def\ae{(\ad+60)}
\def\af{(\ae+60)}
\draw [ domain=\ac:\af ,smooth,variable =\x, gray] plot ({(\ra+.5)*sin \x},.5,{(\ra+.5)* cos \x}); 
\draw [ domain=\ac:\af,smooth,variable =\x, gray] plot ({(\ra+.75)*sin \x},.25,{(\ra+.75)* cos \x}); 
\draw [ domain=\ac:\af ,smooth,variable =\x, gray] plot ({(\ra+.25)*sin \x},.75,{(\ra+.25)* cos \x}); 
\draw [ domain=\ac:\af,smooth,variable =\x, gray] plot ({(\ra+1)*sin \x},0,{(\ra+1)* cos \x}); 
\draw [ domain=\ac:\af ,smooth,variable =\x, gray] plot ({\ra*sin \x},1,{\ra* cos \x}); 

\filldraw[draw=coa, very thick,fill=coa, opacity=.9] ({\ra*sin \aa},0,{\ra* cos \aa})--({(\ra+1)*sin \aa},0,{(\ra+1)* cos \aa})--({\ra*sin \aa},1,{\ra* cos \aa})--({\ra*sin \aa},0,{\ra* cos \aa});

\node[circle,draw=black, fill=coc, inner sep=0pt,minimum size=3pt] at ({(\ra+1)*sin \aa},0,{(\ra+1)* cos \aa}) {};
\node[circle,draw=black, fill=coc, inner sep=0pt,minimum size=3pt] at ({(\ra+.75)*sin \aa},.25,{(\ra+.75)* cos \aa}) {};
\node[circle,draw=black, fill=coc, inner sep=0pt,minimum size=3pt] at ({(\ra+.5)*sin \aa},.5,{(\ra+.5)* cos \aa}) {};
\node[circle,draw=black, fill=coc, inner sep=0pt,minimum size=3pt] at ({(\ra+.25)*sin \aa},.75,{(\ra+.25)* cos \aa}) {};
\node[circle,draw=black, fill=coc, inner sep=0pt,minimum size=3pt] at ({(\ra+0)*sin \aa},1,{(\ra+0)* cos \aa}) {};

\draw [ domain=\aa:\af ,smooth,variable =\x, cod, very thick] plot ({\ra*sin \x},0,{\ra* cos \x}); 
\draw [ domain=\aa:\ac ,smooth,variable =\x, gray] plot ({(\ra+.5)*sin \x},.5,{(\ra+.5)* cos \x}); 
\draw [ domain=\aa:\ac,smooth,variable =\x, gray] plot ({(\ra+.75)*sin \x},.25,{(\ra+.75)* cos \x}); 
\draw [ domain=\aa:\ac ,smooth,variable =\x, gray] plot ({(\ra+.25)*sin \x},.75,{(\ra+.25)* cos \x}); 
\draw [ domain=\aa:\ac,smooth,variable =\x, gray] plot ({(\ra+1)*sin \x},0,{(\ra+1)* cos \x}); 
\draw [ domain=\aa:\ac ,smooth,variable =\x, gray] plot ({\ra*sin \x},1,{\ra* cos \x});

\node[circle,draw=black, fill=cob, inner sep=0pt,minimum size=3pt] at ({\ra*sin \aa},0,{\ra* cos \aa}) {};
\end{tikzpicture}
         \caption{$\THH(S_k\spcat{\cC})$}
         \label{fig:shadow_flag}
     \end{subfigure}
     \hfill
     \begin{subfigure}[b]{0.25\textwidth}
         \centering
	\tdplotsetmaincoords{90}{70}
         \begin{tikzpicture}[tdplot_main_coords]

\def\ra{1}
\def\rb{2}
\def\aa{10}
\def\a1{(\aa+60)}
\def\ab{(\a1+50)}
\def\ac{(\ab+70)}
\def\ad{(\ac+60)}
\def\ae{(\ad+60)}
\def\af{(\ae+60)}
\draw [ domain=\ac:\af ,smooth,variable =\x, gray] plot ({(\ra+.5)*sin \x},.5*\rb,{(\ra+.5)* cos \x}); 
\draw [ domain=\ac:\af,smooth,variable =\x, gray] plot ({(\ra+.75)*sin \x},.25*\rb,{(\ra+.75)* cos \x}); 
\draw [ domain=\ac:\af ,smooth,variable =\x, gray] plot ({(\ra+.25)*sin \x},.75*\rb,{(\ra+.25)* cos \x}); 
\draw [ domain=\ac:\af,smooth,variable =\x, gray] plot ({(\ra+1)*sin \x},0*\rb,{(\ra+1)* cos \x}); 
\draw [ domain=\ac:\af ,smooth,variable =\x, gray] plot ({\ra*sin \x},1*\rb,{\ra* cos \x}); 
%
%
%
%
%
%

\node[circle,draw=black, fill=coc, inner sep=0pt,minimum size=3pt] at ({(\ra+1)*sin \aa},0*\rb,{(\ra+1)* cos \aa}) {};
\node[circle,draw=black, fill=coc, inner sep=0pt,minimum size=3pt] at ({(\ra+.75)*sin \aa},.25*\rb,{(\ra+.75)* cos \aa}) {};
\node[circle,draw=black, fill=coc, inner sep=0pt,minimum size=3pt] at ({(\ra+.5)*sin \aa},.5*\rb,{(\ra+.5)* cos \aa}) {};
\node[circle,draw=black, fill=coc, inner sep=0pt,minimum size=3pt] at ({(\ra+.25)*sin \aa},.75*\rb,{(\ra+.25)* cos \aa}) {};
\node[circle,draw=black, fill=coc, inner sep=0pt,minimum size=3pt] at ({(\ra+0)*sin \aa},1*\rb,{(\ra+0)* cos \aa}) {};
\draw [ domain=\aa:\ac ,smooth,variable =\x, gray] plot ({(\ra+.5)*sin \x},.5*\rb,{(\ra+.5)* cos \x}); 
\draw [ domain=\aa:\ac,smooth,variable =\x, gray] plot ({(\ra+.75)*sin \x},.25*\rb,{(\ra+.75)* cos \x}); 
\draw [ domain=\aa:\ac ,smooth,variable =\x, gray] plot ({(\ra+.25)*sin \x},.75*\rb,{(\ra+.25)* cos \x}); 
\draw [ domain=\aa:\ac,smooth,variable =\x, gray] plot ({(\ra+1)*sin \x},0*\rb,{(\ra+1)* cos \x}); 
\draw [ domain=\aa:\ac ,smooth,variable =\x, gray] plot ({\ra*sin \x},1*\rb,{\ra* cos \x});

\end{tikzpicture}
         \caption{$ \bigvee_{i=1}^k \THH(\spcat{\cC}) $}
         \label{fig:shadow_wo_flag}
     \end{subfigure}
	\hspace{1in}
        \caption{Graphical representations of \cref{cor:small_add}}
        \label{fig:pic_shadows_add}
      \end{figure}
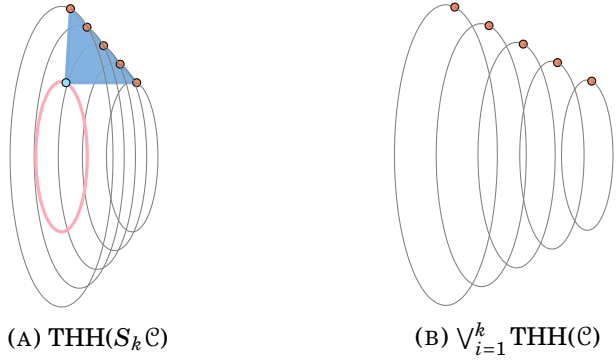

\begin{proof}
  By \cref{lem:base_change_euler}, the induced map $\THH(s_{k-1})$ is
  the Euler characteristic of $\lMr{}{(S_k\cC)}{s_{k-1}}$, computed using the
  dual pair $(\lMr{s_{k-1}}{(S_{k}\cC)}{},\lMr{}{(S_k\cC)}{s_{k-1}})$.  By
  \cref{defn:euler_characteristic}, the Euler characteristic is the following
  composite:
 \begin{equation*}
   \chi(\lMr{}{(S_k\cC)}{s_{k-1}}) \colon \sh{S_{k-1}\cC} \arr \sh{\lMr{s_{k-1}}{(S_{k}\cC)}{} \odot
     \lMr{}{(S_k\cC)}{s_{k-1}}} \simeq \sh{ \lMr{}{(S_k\cC)}{s_{k-1}}\odot
     \lMr{s_{k-1}}{(S_{k}\cC)}{}} \stackrel\epsilon\arr \sh{S_k\cC}.
 \end{equation*}
 The first map is an equivalence by the first statement in \cref{sdot_adjunctions}.
Rewriting in terms of $\THH$, it follows that the induced map $\THH(s_{k-1})$ is the composite
  \begin{equation}  \label{eq:left-half}
    \THH(S_{k-1}\cC) \stackrel{\sim}\arr \THH(S_k\cC; \lMr{}{(S_k\cC)}{s_{k-1}}\odot
      \lMr{s_{k-1}}{(S_{k}\cC)}{})
    \oarr{\ \epsilon \ } \THH(S_k\cC).
  \end{equation}

  Similarly, the induced map $\THH(\pi_k)$ is the composite
  \begin{equation} \label{eq:right-half}
  \chi(\lMr{}{\cC}{\pi_k}) \colon \THH(S_k\cC)  \xarr{\ \eta \ }
    \THH(S_k\cC; \lMr{\pi_k}{\cC}{} \odot \lMr{}{\cC}{\pi_k}) \stackrel{\sim}\arr
    \THH(\cC),
  \end{equation}
  where the equivalence is induced by the evaluation map of the dual pair $(\lMr{\pi_k}{\cC}{},\bcl{}{\pi_k}{\cC})$.

  Composing the two sequences in \eqref{eq:left-half} and \eqref{eq:right-half}, the two middle maps arise by applying $\THH(S_k\cC;-)$ to the
  cofiber sequence from \cref{sdot_adjunctions}.  Since $\THH$ preserves cofiber
  sequences, this produces a cofiber sequence
  \[\THH(S_{k-1}\cC) \stackrel{s_{k-1}}\arr \THH(S_k\cC) \stackrel{\pi_k}\arr
    \THH(\cC).\] The second map has a section, induced by $\iota_k$, so the
  cofiber sequence
  splits.
\end{proof}

\subsection{Additivity with coefficients}\label{sec:thh_add_coeff}

Next we generalize \cref{cor:small_add} to allow for twisted
coefficients.  For ease of
future reference we state the theorem in its multisimplicial form.

\begin{rmk}
	From the properties of the $S_\bullet$-construction, any twisting
	$\lMr{L}{\cC/\cD}{R}$ of spectral Waldhausen categories induces a twisting
	$\lMr{S_\bullet L}{(S_\bullet\cC/S_\bullet\cD)}{S_\bullet R}$.  For ease of
	reading, in such cases we drop the $S_\bullet$-notation from the subscripts
	and simply write $\lMr{L}{(S_\bullet\cC/S_\bullet\cD)}{R}$.
\end{rmk}

\begin{thm}\label{thm:THH_additivity}
  Given a twisting  $\lMr{L}{\cC/\cD}{R}$ of spectral Waldhausen categories
  there is an equivalence of spectra
	\[
	\bigvee_{\substack{1 \leq i_j \leq k_j \\ 1 \leq j \leq n}} \THH(\spcat{\cC};\tensor[_L]{\spcat{\cD}}{_R})
	\stackrel\sim\arr \THH(w_{k_0}S^{(n)}_{k_1, \dotsc, k_n} \spcat{\cC};\tensor[_L]{(w_{k_0}S^{(n)}_{k_1, \dotsc, k_n} \spcat{\cD})}{_R}).
	\]
\end{thm}

The theorem follows from the $w_{\bullet}$-invariance of $\THH$ (see \cref{ex:w_bullet_THH}), and an inductive argument based on the following generalization of
\cref{prop:small_add}:

\begin{prop}\label{prop:small_add_2}
  Let $\lMr{L}{\cC/\cD}{R}$ be a twisting of spectral Waldhausen categories.
The spectral functors $s_{k-1}$ and $\iota_k$ induce an equivalence
\[ \THH(S_{k-1}\spcat{\cC};\tensor[_L]{(S_{k-1}\spcat\cD)}{_R}) \vee \THH(\spcat{\cC};\tensor[_L]{\spcat\cD}{_R}) \stackrel\sim\arr \THH(S_k\spcat{\cC};\tensor[_L]{(S_k\spcat\cD)}{_R}). \]
\end{prop}

The proof of this proposition is largely analogous to the proof of
\cref{prop:small_add}; the main difficulty that the twisting adds is that the
construction of the equivalences in \eqref{eq:left-half} and
\eqref{eq:right-half} requires an extra step.

The functors $s_{k - 1}$, $d_{k}$, $\pi_k$, and $\iota_{j}$ are functors of spectral
Waldhausen categories.  By definition, $d_ks_{k-1} = \id$ and $\pi_k\iota_k = \id$.  These
identities define the unit and counit, respectively, of adjunctions
\[
\xymatrix{
S_{k - 1}\uncat\cC \ar@<.5ex>[r]^-{s_{k - 1}} & S_{k}\uncat \cC \ar@<.5ex>[l]^-{d_{k}}
}
\qquad
\xymatrix{
S_{k}\uncat\cC \ar@<.5ex>[r]^-{\pi_{k}} & \uncat \cC \ar@<.5ex>[l]^-{\iota_{k}}
}
\]
on the associated base categories. These adjunctions do \emph{not} extend to
spectrally enriched adjunctions between our models for the spectral categories
$S_{k}\cC$, because the Moore end is not 2-functorial (\cite[\swcref{Section 4.5}{\cref{swc-2-functoriality}}]{clmpz-specWaldCat}).  They do, however, still satisfy a condition analogous to an
adjunction:

\begin{prop} \label{prop:spectAdj}
  The spectral functors $d_{k}$ and $\pi_{k}$ induce equivalences of spectra
  \begin{align*}
  S_{k}\cC(s_{k - 1}a, b) & \oarr{\sim} S_{k - 1}\cC(d_{k}s_{k - 1}a, d_{k}b) = S_{k - 1}\cC(a, d_{k}b) \\
  S_{k}\cC(a, \iota_{k}b) & \oarr{\sim} \cC(\pi_{k}a, \pi_{k}\iota_{k}b) = \cC(\pi_{k}a, b).
  \end{align*}
\end{prop}

We postpone the proof of the proposition to \S\ref{sec:important_dual_pair}, and now prove \cref{prop:small_add_2}.

\begin{proof}[Proof of \cref{prop:small_add_2}]
  We fill in the details that differ from the proof of \cref{prop:small_add}.
  The spectral functor $s_{k-1}$ defines a morphism of twistings
  \[(s_{k-1},s_{k-1})\colon\lMr{L}{(S_{k-1}\cD)}{R} \arr \lMr{L}{(S_{k}\cD)}{R}.\]
  By \cref{base_change_beck_chevalley} there is an associated map
  \[f\colon \lMr{L}{(S_{k-1}\cD)}{R} \odot \lMr{s_{k-1}}{(S_k\cC)}{} \arr
    \lMr{s_{k-1}}{(S_k\cC)}{} \odot \lMr{L}{(S_k\cD)}{R}\]
  whose trace is the induced map
  \[\THH(s_{k-1})\colon \THH(S_{k-1}\cC;\lMr{L}{(S_{k-1}\cD)}{R}) \arr \THH(S_k\cC;
    \lMr{L}{(S_k\cD)}{R}).\] By definition, the trace of $f$ is the composite
  \begin{align*}\sh{\lMr{L}{(S_{k-1}\cD)}{R}} &\xarr{\ \eta \ }
                                                \sh{\lMr{L}{(S_{k-1}\cD)}{R}
                                                \odot \lMr{s_{k-1}}{(S_k\cC)}{}
                                                \odot \lMr{}{(S_{k}\cC)}{s_{k-1}}} \\
    &\xarr{\ f \ }
    \sh{\lMr{s_{k-1}}{(S_k\cC)}{} \odot \lMr{L}{(S_k\cD)}{R} \odot
                                                \lMr{}{(S_{k}\cC)}{s_{k-1}} } \\
    &\cong     \sh{ \lMr{}{(S_k\cC)}{s_{k-1}}  \odot\lMr{s_{k-1}}{(S_{k}\cC)}{} \odot
      \lMr{L}{(S_k\cD)}{R}} \xarr{\ \epsilon \ }
      \sh{\lMr{L}{(S_k\cD)}{R}}.
  \end{align*}
  As in the proof of
  \cref{prop:small_add}, we will prove that, prior to the application of
  the evaluation map, the composite is an equivalence.  It suffices to
  prove that $f$ is a pointwise equivalence of $(S_{k-1}\cC,S_k\cC)$-bimodules,
  which follows from the commutative diagram
  \[
  \begin{tikzcd}
\lMr{L}{(S_{k - 1}\cD)}{} \odot \lMr{}{(S_{k - 1}\cD)}{R} \odot \lMr{s_{k - 1}}{(S_{k}\cC)}{} \ar[r, "f"] \ar[dd, "\simeq"', "\id \odot \id \odot d_k"] &
\lMr{s_{k - 1}}{(S_{k}\cC)}{} \odot \lMr{L}{(S_{k}\cD)}{} \odot \lMr{}{(S_{k}\cD)}{R} \ar[d, "\simeq"] \\
 & \lMr{L}{(S_{k - 1}\cD)}{} \odot \lMr{s_{k - 1}}{(S_{k}\cD)}{} \odot \lMr{}{(S_{k}\cD)}{R} \ar[d, "\simeq", "\id \odot d_{k} \odot \id"'] \\
\lMr{L}{(S_{k - 1}\cD)}{} \odot \lMr{}{(S_{k - 1}\cD)}{R} \odot \lMr{}{(S_{k - 1}\cC)}{d_k} \ar[r, "\simeq"] &
\lMr{L}{(S_{k - 1}\cD)}{} \odot \lMr{}{(S_{k}\cD)}{d_{k}} \odot \lMr{}{(S_{k}\cD)}{R},
  \end{tikzcd}
  \]
  where the unlabeled equivalences are instances of
  \cref{base_change_composition} and the equivalences induced by $d_{k}$ are
  from \cref{prop:spectAdj}.

  Similarly, the map $\pi_k$ induces a morphism of twistings
  \[(\pi_k,\pi_k)\colon \lMr{L}{(S_k\cD)}{R} \to \lMr{L}{\cD}{R}.\] The same
  logic as above reduces the proof to showing that the map
  \[g\colon \lMr{L}{(S_k\cD)}{R} \odot \lMr{\pi_k}{\cC}{} \arr
    \lMr{\pi_k}{\cC}{} \odot \lMr{L}{\cD}{R} \]
    from \cref{base_change_beck_chevalley} is a pointwise
  equivalence of bimodules, which follows in the same manner from the commutative diagram
  \[
\begin{tikzcd}
\lMr{L}{(S_{k}\cD)}{} \odot \lMr{}{(S_{k}\cD)}{R} \odot \lMr{}{(S_{k}\cC)}{\iota_{k}} \ar[dd, "\simeq"', "\id \odot \id \odot \pi_k"] \ar[r, "\simeq"] &
\lMr{L}{(S_{k}\cD)}{} \odot \lMr{}{(S_{k}\cD)}{\iota_{k}} \odot \lMr{}{\cD}{R} \ar[d, "\id \odot \pi_k \odot \id"', "\simeq"] \\
 & \lMr{L}{(S_{k}\cD)}{} \odot \lMr{\pi_{k}}{\cD}{} \odot \lMr{}{\cD}{R} \ar[d, "\simeq"] \\
\lMr{L}{(S_{k}\cD)}{} \odot \lMr{}{(S_{k}\cD)}{R} \odot \lMr{\pi_{k}}{\cC}{} \ar[r, "g"] &
\lMr{\pi_{k}}{\cC}{} \odot \lMr{L}{\cD}{} \odot \lMr{}{\cD}{R}.
\end{tikzcd}
  \]
    The rest of the proof proceeds as in the untwisted case.
\end{proof}

\begin{rmk}
  These results can be generalized to the case when $\cD$ is a pointed
  spectral category.  In this case, $w_k\cD$ is defined to include all maps and  $w_0\spcat\cD \to w_k\spcat\cD$ is still a Dwyer--Kan
  embedding. The definition of $S_k\spcat\cD$ from \S\ref{subsec:Sdot}
  is modified to be the subcategory of
  $\spcat{\Fun([k] \times [k],\cD)}$ on diagrams sending each $(i \geq j)$ to the
  zero object $* \in \uncat\cD$.

  This doesn't affect any of the proofs because we only ever consider diagrams
  in the image of the functors $L$ and $R$, and so it is enough to control the
  behavior of cofibrations and pushouts in $S_k\spcat\cC$.
\end{rmk}

\subsection{The technical proofs}\label{sec:important_dual_pair}

In this subsection we prove \cref{prop:spectAdj,sdot_adjunctions}.

\begin{proof}[Proof of \cref{prop:spectAdj}]
  The proof requires explicit properties of the construction of the mapping
  spectra in $S_{k}\cC$ from
	\cite[\swcref{\S 4}{\S\ref{swc-ex:fun_cat}}]
	{clmpz-specWaldCat}.
  The key fact is that the mapping spectrum $S_{k}\cC(a, b)$ is equivalent, via
  the canonical restriction maps, to the homotopy limit of the zig-zag diagram
  of spectra built out of the composition maps between $\cC(a(i, j), b(i, j))$
  for $0 \leq i, j \leq k$
	\cite[\swcref{Theorem 4.1{\it ii}}{\cref{swc-thm:moore_end}\ref{swc-thm:moore_end_zz}}]
	{clmpz-specWaldCat}.
  Under these equivalences, the last face functor
  $d_{k} \colon S_{k}\cC \to S_{k - 1}\cC$ agrees with the map to the homotopy
  limit of the subdiagram where $0 \leq i, j \leq k - 1$.  Replacing the domain
  $a$ with a diagram $s_{k - 1}a$ in the image of the last degeneracy functor,
  the canonical map $s_{k - 1}a(k - 1, j) \to s_{k - 1}a(k, j)$ is the identity
  map, and thus the induced maps of spectra
\[
\cC(s_{k - 1}a(k, j), b(k, j)) \arr \cC(s_{k - 1}a(k - 1, j), b(k, j))
\]
are identity maps for any $j$.  Similarly, restricting to the bottom row gives
\[
\cC(s_{k - 1}a(i, k), b(i, k)) = \ast
\]
for any $i$.  It follows that the map of homotopy limits from the diagram with $0 \leq i, j \leq k$ to the diagram with $0 \leq i, j \leq k - 1$ is an equivalence of spectra.  Therefore, the last face functor $d_{k}$ induces an equivalence of mapping spectra
\[
S_{k}\cC(s_{k - 1}a, b) \oarr{\simeq} S_{k - 1}\cC(d_{k}s_{k - 1}a, d_{k}b),
\]
as claimed.  A similar argument shows that the functor $\pi_{k}$ induces an equivalence of mapping spectra
\[
S_{k}\cC(a, \iota_{k}b) \oarr{\sim} \cC(\pi_{k}a, \pi_{k}\iota_{k}b).
\]
\end{proof}

For ease of notation, we require an extra definition.

\begin{defn}
  A pointwise map of $(\cC,\cD)$-bimodules, denoted
  $\xymatrix{\cX \ar@{.>}[r] & \cY}$, is a map of spectra $\cX(c,d) \to \cY(c,d)$ for
  each $c\in \ob \cC$ and $d\in \ob \cD$.  These are not required to satisfy any
  coherence with the $\cC$ and $\cD$ actions.  When a pointwise map of
  $(\cC,\cD)$-bimodules is compatible with the $\cC$-action, we denote it by
  $\xymatrix{\cX \ar@{(.>}[r] & \cY}$.
\end{defn}

We write $\epsilon_0 \colon s_{k - 1}d_{k} \arr
\id$ for the counit of the adjunction $(s_{k-1}\dashv d_k)$ of base categories.  Composing with the morphism $\epsilon_0
\colon s_{k - 1}d_{k}b \arr b$ in $S_{k}\cC_0$ defines a pointwise map of
bimodules
\[\xymatrix{\lMr{}{(S_k\cC)}{s_{k-1}d_k} \ar@{(.>}[r] & S_k\cC},\]
compatible with the left action.  Similarly, composing with the unit $\eta_0 \colon b \arr \iota_{k} \pi_{k}$
for the adjunction $(\pi_k \dashv \iota_k)$ of base categories defines a pointwise map of bimodules
\[\xymatrix{\lMr{}{(S_k\cC)}{} \ar@{(.>}[r] & \lMr{}{(S_k\cC)}{\iota_{k} \pi_{k} } }\]
that is also compatible with the left action.

\begin{lem} \label{lem:annoying_diagrams}
The pointwise equivalences from \cref{prop:spectAdj} fit into commutative
diagrams of pointwise morphisms of bimodules
\[
  \xymatrix{\lMr{}{(S_{k}\cC)}{s_{k - 1}} \odot \lMr{s_{k - 1}}{(S_{k}\cC)}{} \ar[d]^\simeq \ar[r]^-\epsilon & S_{k}\cC \\
\lMr{}{(S_{k}\cC)}{s_{k - 1}} \odot \lMr{}{(S_{k - 1}\cC)}{d_{k}} \ar[r]^-\simeq & \lMr{}{(S_{k}\cC)}{s_{k - 1}d_{k}} \ar@{(.>}[u]_{\epsilon_0}
} \qqand
\xymatrix{S_{k}\cC \ar[r]^\eta \ar@{(.>}[d]^{\eta_0} & \lMr{\pi_{k}}{\cC}{} \odot \lMr{}{\cC}{\pi_{k}} \\
\lMr{}{(S_{k}\cC)}{\iota_{k}\pi_{k}} \ar[r]^\simeq & \lMr{}{(S_{k}\cC)}{\iota_{k}} \odot \lMr{}{\cC}{\pi_{k}} \ar[u]^\simeq
}
\]
relating the evaluation (resp. coevaluation) map for the dual pairs with the
counit (resp. unit) of the adjunction on base categories.\end{lem}
\begin{proof}
  We prove the lemma for the first diagram; the second follows analogously.
  Recall that the spectral category $S_k\cC$ is defined as a full subcategory of the functor category $\spcat{\Fun([k]^2,\cC)}$. Define a twisting of spectral categories
  $\lMr{L}{S_k\cC/\widetilde{S_k\cC}}{R}$ where
\[\widetilde{S_k\cC} \subseteq \spcat{\Fun([1] \times [k]^2,\cC)}\] is the full subcategory of diagrams that at each $i \in [1]$ satisfy the conditions for $S_k\cC$. The spectral functor $L\colon S_k\cC \to \widetilde{S_k\cC}$ arises from the collapse $[1] \to {*}$, and the spectral functor $R\colon S_k\cC \to \widetilde{S_k\cC}$ arises from the map of posets $[1] \times [k]^2 \to [k]^2$ that on $1 \in [1]$ is the identity of $[k]^2$ and on $0 \in [1]$ applies $i \mapsto \max(i,k-1)$ to each coordinate of $[k]^2$ (this is the map of totally ordered sets inducing $s_{k-1}d_k$). Let $r_0,r_1\colon \widetilde{S_k\cC} \rightrightarrows S_k\cC$ denote the spectral functors that restrict to $0 \in [1]$ and $1 \in [1]$, respectively. We form the following diagram of $(S_k\cC,S_k\cC)$-bimodules
  \[ \xymatrix{
	  S_k\cC_{s_{k-1}} \odot \lMr{s_{k-1}}{S_k\cC}{} \ar[d]^-{(1,1,d_k)} \ar@/^2em/[rr]^-{=} & \ar[l]_-\sim^-{(r_0,1,1)} \widetilde{S_k\cC}_{s_{k-1}} \odot \lMr{s_{k-1}}{S_k\cC}{} \ar[d]^-{(1,s_{k-1},1)} \ar[r]^-\sim_-{(r_1,1,1)} &
	  S_k\cC_{s_{k-1}} \odot \lMr{s_{k-1}}{S_k\cC}{} \ar[d]^-{(1,s_{k-1},1)} \\
	  S_k\cC_{s_{k-1}} \odot S_{k-1}\cC_{d_k} \ar[d]^-\sim & \ar[l]_-\sim^-{(r_0,d_k,d_k)} \widetilde{S_k\cC} \odot S_k\cC \ar[d]^-\sim \ar[r]^-\sim_-{(r_1,1,1)} &
	  S_k\cC \odot S_k\cC \ar[d]^-\sim \\
	  (S_k\cC)_{s_{k-1}d_k} \ar@/_2em/@{(.>}[rr]_{\epsilon_0} & \ar[l]_-\sim^-{r_0} \widetilde{S_k\cC} \ar[r]^-\sim_-{r_1} &
	  S_k\cC.
  } \]
  The four rectangular regions commute and all the solid arrows are maps of $(S_k\cC,S_k\cC)$-bimodules, as is verified by checking that various maps of twistings, arising from maps of posets, agree with each other. The top region also commutes easily. The region at the very bottom commutes in the category of pointwise maps of bimodules, again using the description of $\Fun([1],\cC)$ as a homotopy limit of a zig-zag -- see
	\cite[\swcref{\S4}{\S\ref{swc-ex:fun_cat}}]
	{clmpz-specWaldCat} for more details. The outside maps are the desired pointwise maps of bimodules, finishing the proof.
\end{proof}

We are now ready to prove \cref{sdot_adjunctions}.

\begin{proof}
The first claim in the proposition follows from the commutativity of the diagrams
\[
\begin{tikzcd}
S_{k - 1}\cC \ar[r, "\eta"] \ar[d, equal] & \lMr{s_{k - 1}}{(S_{k}\cC)}{} \odot \lMr{}{(S_{k}\cC)}{s_{k - 1}} \ar[d, "d_{k} \odot \id", "\simeq"'] \\
\lMr{}{(S_{k - 1})}{d_{k}s_{k - 1}} \ar[r, "\simeq"] & \lMr{}{(S_{k - 1}\cC)}{d_{k}} \odot \lMr{}{(S_{k}\cC)}{s_{k - 1}}
\end{tikzcd}
\quad \text{and} \quad
\begin{tikzcd}
\lMr{}{\cC}{\pi_{k}} \odot \lMr{\pi_{k}}{\cC}{}  \ar[r, "\epsilon"] & \cC \ar[d, equal] \\
\lMr{}{\cC}{\pi_{k}} \odot \lMr{}{(S_{k}\cC)}{\iota_{k}} \ar[u, "\simeq"', "\id \odot \pi_{k}"] \ar[r, "\simeq"] & \lMr{}{\cC}{\pi_{k}\iota_{k}}
\end{tikzcd}
\]
which are formally analogous to those in \cref{lem:annoying_diagrams} (but easier to check because $\eta_0 = \id$ for $(s_{k-1}\dashv d_k)$ and $\epsilon_0 = \id$ for $(\pi_k \dashv \iota_k)$).

  To check the cofiber sequence statement note that the two diagrams in
  \cref{lem:annoying_diagrams} give, for each $x,y\in \ob S_k\cC$, a natural weak
  equivalence betwen the sequence of interest and the sequence
  \[S_k\cC(x,s_{k-1}d_ky) \to S_k\cC(x,y) \to S_k\cC(x,\iota_k\pi_k y).\]
  Thus to show that the given sequence of bimodules is a homotopy cofiber
  sequence it suffices to show that this is a homotopy cofiber sequence of
  spectra for each $x,y$.

  The counit and unit of the adjunctions $(s_{k-1}\dashv d_k)$ and
  $(\pi_k \dashv \iota_k)$ of base categories fit into a pushout square of functors
  $S_k \uncat\cC \to S_k\uncat\cC$
  \begin{equation*} \xymatrix{
      (s_{k-1} d_k) \ar[d] \ar[r]^-{\epsilon_0} & \id \ar[d]^-{\eta_0} \\
      {*} \ar[r] & (\iota_k \pi_k)  } \end{equation*}
  whose horizontal arrows are cofibrations.
  Since $S_k\spcat\cC$ is a spectral Waldhausen category, there is an induced homotopy pushout
  square of spectra
  \begin{equation*}
    \xymatrix{
      S_{k}\cC(x,s_{k-1}d_k y) \ar[d] \ar[r] & S_{k}\cC(x,y) \ar[d] \\
      S_{k}\cC(x,\ast) \ar[r] & S_{k}\cC(x,\iota_k\pi_k y).
    }
  \end{equation*}
  Since the lower-left corner is contractible, the other three terms form a
  homotopy cofiber sequence, as desired.
\end{proof}

\section{The Dennis trace}\label{sec:dennis_trace}

In this section we construct the Dennis trace map
$K(\End(\spcat\cC)) \to \THH(\spcat\cC)$ out of endomorphism $K$-theory for a spectrally enriched Waldhausen category $\spcat\cC$, as well
as a twisted Dennis trace which allows bimodule coefficients.
This material serves as the
scaffolding for the construction of the trace map to $\TR$ in
\S\ref{sec:equiv_dennis_trace}--\ref{sec:trace_to_TR}.
We conclude the section with a concrete description of the effect of the Dennis trace on $\pi_0$ in terms of bicategorical traces (\cref{lem:pi_0_dennis_trace}).

\subsection{Endomorphism categories}

We begin by defining endomorphism categories.

\begin{defn}\label{def:category_of_endomorphisms}
  For any Waldhausen category $\uncat{\cC}$, let $\End(\uncat\cC)$ be the
  Waldhausen category of functors $\Fun(\bbN,\uncat\cC)$, where $\bbN$ is
  considered as a category with one object and morphism set $\bbN$.  More
  concretely, the objects of $\End(\uncat{\cC})$ are endomorphisms
  $f\colon a \arr a$ in $\uncat{\cC}$, and the morphisms are commuting squares of the
  form
  \[ \xymatrix @R=1.5em{
      a \ar[r]^-f \ar[d]_-i & a \ar[d]^-{i} \\
      b \ar[r]_-g & b. } \]
  We define the morphism to be a cofibration or
  weak equivalence if $i$ is a cofibration or weak equivalence, respectively.
  We also define exact functors
 \[
 \begin{tikzcd}
 \uncat{\cC} \arrow[r, bend left=50, "\iota_0"{below}]
 \arrow[r, bend right=50, "\iota_1"]
 &\End(\uncat{\cC}) \arrow[l, ""]
 \end{tikzcd}
\]
where $\End(\uncat{\cC}) \arr \uncat{\cC}$  forgets the endomorphism $f$.  The inclusions $\iota_0,\iota_1\colon \uncat{\cC} \arr \End(\uncat{\cC})$ equip each object $a$ with either the zero endomorphism or the identity endomorphism.
\end{defn}

\begin{example}
  If $A$ is a ring spectrum and $\spcat\cC = \Perf = \tensor[^A]{\Perf}{}$ is the
  spectral Waldhausen category of perfect $A$-modules from \cref{ex:Perf_as_spectral_wald_cat}, then $K(\uncat\cC)$ is the usual algebraic $K$-theory spectrum $K(A)$ of $A$, and the $K$-theory of $\End(\uncat{\cC})$ is the $K$-theory of endomorphisms $K(\End(A))$.
\end{example}

It is also possible to extend the definition of endomorphism $K$-theory to twistings.  First we recall our main example of a twisting.

\begin{example}\label{ex:twisting_mod}
	Let $A$ be a ring spectrum. Recall from \cref{modules_convention,ex:perfect_modules,ex:bimodule_smash_functor} the spectral categories $
\tensor[]{\Perf}{}$ of perfect left $A$-modules and $
\tensor[]{\Mod}{}$ of all left $A$-modules. Let
	\[ L,R\colon \tensor[]{\Perf}{}  \rightrightarrows \tensor[]{\Mod}{} \]
	denote, respectively, the inclusion and the functor $M \sma_A -$ for a
        cofibrant $(A,A)$-bimodule $M$. This defines a twisting that we denote by $\tw{\tensor[]{\Perf}{}}{}{M}{\tensor[]{\Mod}{}}$.
\end{example}

\begin{defn}
  Given a twisting $\tw{\cC}{L}{R}{\cD}$ of spectral Waldhausen categories, the
  \textbf{twisted endomorphism category} $\End\left(\tw{\cC}{L}{R}{\cD}\right)$
  is the category where
  \begin{itemize}
  \item the objects are pairs $(a,f)$ of $a \in \ob \uncat{\cC}$ and a morphism $f\colon L(a) \rto R(a)$ in $\uncat{\cD}$, and
  \item a morphism $(a,f) \rto (b, g)$ is a morphism $i\colon a \rto b$
    in $\uncat{\cC}$ such that the diagram
    \[
      \begin{tikzcd}
        L(a) \ar{r}{f} \ar{d}{L(i)} & R(a) \ar{d}{R(i)}\\
        L(b) \ar{r}{g} & R(b)
      \end{tikzcd}\]
    commutes.
  \end{itemize}
Note that this definition only uses the base categories $\uncat\cC$ and $\uncat\cD$, and not the spectral enrichment.
\end{defn}

\begin{example}\hfill
\begin{enumerate}
\item
  When $\cD = \cC$ and $L = R = \id_\cC$ we get the usual endomorphism category.

\item
  The twisted endomorphism category for $\tw{\tensor[]{\Perf}{}}{}{M}{\tensor[]{\Mod}{}} = \tw{\tensor[^A]{\Perf}{}}{}{M}{\tensor[^A]{\Mod}{}}$  has as objects $A$-module maps $P \to M \sma_A P$ with $P$ a perfect $A$-module. Following \cite{LM12}, the $K$-theory of its base Waldhausen category is denoted by
  \[ K(A; M) \coloneqq K\End(\tw{\tensor[^A]{\Perf}{}}{}{M}{\tensor[^A]{\Mod}{}}). \]
\end{enumerate}
\end{example}

\subsection{Bispectra}

Before defining the Dennis trace, we introduce some formal structure that arises naturally when analyzing $\THH$.

\begin{defn}\label{def:bispectra}
  A \textbf{bispectrum} is a symmetric spectrum object in orthogonal
  spectra \cite{mandell_may_shipley_schwede,hovey:spec_sym_spec,clmpz-specWaldCat}.
\end{defn}

In order to construct bispectra, we need a technical tool which formalizes the way that the iterated $S_{\bullet}$-constructions $|S^{(n)}_{\bullet,\dotsm,\bullet}\uncat\cC|$ assemble into a symmetric spectrum. Let $\mc I$ be the skeleton of the category of finite sets and injections spanned by the objects
\[\underline{n} = \{1,\ldots,n\}\] for $n \geq 0.$
Let $\mathbf\Delta^{\op \times n}$ be the $n$-fold product of the opposite of the category $\mathbf\Delta$ of nonempty totally ordered finite sets
\[[k] = \{0 < \dotsm < k\}.\]
For each morphism $f\colon \underline{m} \arr \underline{n}$ in $\mc I$, there is an induced functor $f_*\colon \mathbf\Delta^{\op \times m} \arr \mathbf\Delta^{\op \times n}$ taking $([k_1],\dotsc,[k_m])$ to the $n$-tuple whose value at $f(i)$ is $[k_i]$ and whose value outside the image of $f$ is always $[1]$.
In particular, when $m = n$ there is an action of the symmetric group $\Sigma_n$ on $\mathbf\Delta^{\op \times n}$. This rule defines a strict diagram of categories indexed by $\mc I$, and we write $\mc I \int \mathbf\Delta^{\op\times -}$ for its Grothendieck construction.  Thus, the objects of the category $\mc I \int \mathbf\Delta^{\op\times -}$ are tuples
\[(\underline{m}; k_1,\dotsc,k_m),\] where $m, k_i \geq 0$, and a morphism
\[(\underline{m}; k_1,\dotsc,k_m) \to (\underline{n}; l_1,\dotsc, l_n)\] consists of an injection $f\colon \underline{m} \arr \underline{n}$ and a morphism $(\phi_i) \colon f_*([k_1],\dotsc,[k_m]) \to ([l_1],\dotsc,[l_n])$ in $\mathbf\Delta^{\op\times n}$.

\begin{defn}\label{sigma-delta-diagram}
  Given a pointed category $\mc M$, a
  \textbf{$\Sigma_{\mathbf\Delta}$-diagram} in $\mc{M}$ is a functor
  \[ X_{(\bullet;\bullet,\dotsc,\bullet)} \colon \mc I \textstyle\int \mathbf\Delta^{\op\times -} \arr \mc{M} \]
  with the following two properties:
  \begin{itemize}
  \item $X_{(\underline{m};k_1,\dotsc,k_m)} \cong \ast$ any time $k_i = 0$ for at least one $i$, and
  \item the morphisms $(\underline{m};k_1,\dotsc,k_m) \arr (\underline{n};f_*(k_1,\dotsc,k_m))$ with every $\phi_i = \id$ induce isomorphisms
    \[ X_{(\underline{m};k_1,\dotsc,k_m)} \cong X_{(\underline{n};f_*(k_1,\dotsc,k_m))}. \]
  \end{itemize}
\end{defn}

The symmetric group action on $\mathbf\Delta^{\op \times n}$ defines an action of $\Sigma_n$ on the geometric realization of the multisimplicial object $\abs{X_{(\underline{n}, \bullet, \dotsc, \bullet)}}$, and this construction extends to a functor from $\Sigma_\Delta$-diagrams in a pointed simplicial model category $\cM$ to symmetric spectrum objects in $\cM$
	\cite[\swcref{Lemma 6.3}{\cref{swc-realization_is_symmetric_spectrum}}]
	{clmpz-specWaldCat}.
For further discussion of $\Sigma_\Delta$-diagrams, see
	\cite[\swcref{\S6}{\S\ref{swc-sec:properness}}]
	{clmpz-specWaldCat}.

\subsection{Definition of the Dennis trace}\label{subsec:dennis_trace}

Let $\spcat\cC$ be a spectral Waldhausen category and let $\cX$ be a $(\cC,\cC)$-bimodule.  The key observation for the construction of the Dennis trace is that the inclusion of $0$-simplices in the cyclic bar construction defines a canonical map
  \begin{equation} \label{eq:k_end_trace_0}
    \bigvee_{c_0\in \ob \uncat{\cC}} \cX(c_0,c_0) \to \THH(\cC;\cX).
  \end{equation}
When $\cX = \cC$, each object $f \colon c_0 \to c_0$ of $\End(\uncat{\cC})$ defines a map of spectra $\bbS \arr \spcat{\cC}(c_0,c_0)$, and so composing with \eqref{eq:k_end_trace_0} gives a map
\begin{equation}\label{eq:k_end_trace_1}
\Sigma^{\infty} \ob \End(\uncat{\cC}) = \bigvee_{\substack{f\colon c_0 \to c_0,\\ c_0 \neq *}} \bbS \arr \bigvee_{c_0 \in \ob \uncat{\cC}} \spcat{\cC}(c_0,c_0) \arr \THH(\spcat{\cC})
\end{equation}
where $f$ runs over the objects of $\End(\uncat{\cC})$.  See
\cref{fig:k_end_trace_1} for a picture of this map.
Applying \eqref{eq:k_end_trace_1} to the spectral Waldhausen category $w_{k_0}S^{(n)}_{k_1,\ldots,k_n} \spcat{\cC}$
for each value of $n$ and $k_0, \dotsc, k_n$ defines a map of orthogonal spectra
\[
\Sigma^{\infty} \ob \End(w_{k_0}S^{(n)}_{k_1,\ldots,k_n} \uncat{\cC}) \arr \THH(w_{k_0}S^{(n)}_{k_1,\ldots,k_n} \spcat{\cC}).
\]
Appending the splitting from \cref{thm:THH_additivity} gives a zig-zag of orthogonal spectra
\begin{equation}\label{eq:levelwise_trace_zigzag}
\Sigma^{\infty} \ob \End(w_{k_0}S^{(n)}_{k_1,\ldots,k_n} \uncat{\cC}) \arr \THH(w_{k_0}S^{(n)}_{k_1, \dotsc, k_n} \spcat{\cC}) \overset{\simeq}{\longleftarrow} \bigvee_{\substack{i_1,\dotsc, i_n \\ 1 \leq i_j \leq k_j}} \THH(\spcat{\cC}).
\end{equation}
The number of summands on the right is the same as the number of nonzero points
in the set $S^1_{k_1} \sma \dotsm \sma S^1_{k_n}$, where $S^1_\bullet$ is the
simplicial circle $\Delta[1]/\partial \Delta[1]$. Therefore these wedge sums
form an $(n + 1)$-fold multisimplicial spectrum that is constant in the $k_0$
direction.

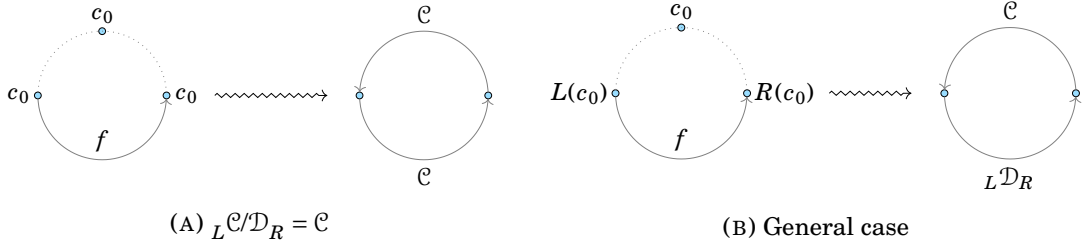
\begin{figure}
\begin{subfigure}{.46\textwidth}
\resizebox{\textwidth}{!}{\begin{tikzpicture}
\def\shrink{3}

%
%

\node[circle,draw=black, fill=cob, inner sep=0pt,minimum size=3pt] at ({-5+cos 180},{-sin 180}) {};

\node[circle,draw=black, fill=cob, inner sep=0pt,minimum size=3pt] at ({-5+cos 0},{-sin 0}) {};
\node[circle,draw=black, fill=cob, inner sep=0pt,minimum size=3pt] at ({-5+cos 90},{sin 90}) {};
\draw [dotted, domain=(180+\shrink):(360-\shrink) ,smooth,variable =\x, gray] plot ({-5+cos \x},{-sin \x}); 
\draw [->, domain=(180-\shrink):(0+\shrink) ,smooth,variable =\x, gray] plot ({-5+cos \x},{-sin \x}); 

\node[above] at ({-5+cos 270},{-sin 270}) {$c_0$};
\node[right] at ({-5+cos 0},{sin 0}) {$c_0$};
\node[left] at ({-5+cos 180},{-sin 180}) {$c_0$};

\node[above] at ({-5+cos 270},{sin 270}) {$f$};

\node[circle,draw=black, fill=cob, inner sep=0pt,minimum size=3pt] at ({cos 180},{-sin 180}) {};

\node[circle,draw=black, fill=cob, inner sep=0pt,minimum size=3pt] at ({cos 0},{-sin 0}) {};
\draw [ <-,domain=(180+\shrink):(360-\shrink) ,smooth,variable =\x, gray] plot ({cos \x},{-sin \x}); 
\draw [->, domain=(180-\shrink):(0+\shrink) ,smooth,variable =\x, gray] plot ({cos \x},{-sin \x}); 

\node[below] at ({cos 270},{sin 270}) {$\mathcal{C}$};

\node[above] at ({cos 270},{-sin 270}) {$\mathcal{C}$};


\draw[->, decorate, decoration={
    zigzag,
    segment length=4,
    amplitude=.9,post=lineto,
    post length=2pt
}] (-3.25,0)--(-1.5,0);
\end{tikzpicture}}
\caption{
  $\lMr{L}{\cC/\cD}{R} = \cC$}\label{fig:k_end_trace_1}
\end{subfigure}
\hfill
\begin{subfigure}{.5\textwidth}
\resizebox{\textwidth}{!}{\begin{tikzpicture}
\def\shrink{3}

%
%

\node[circle,draw=black, fill=cob, inner sep=0pt,minimum size=3pt] at ({-5+cos 180},{-sin 180}) {};

\node[circle,draw=black, fill=cob, inner sep=0pt,minimum size=3pt] at ({-5+cos 0},{-sin 0}) {};
\node[circle,draw=black, fill=cob, inner sep=0pt,minimum size=3pt] at ({-5+cos 90},{sin 90}) {};
\draw [dotted, domain=(180+\shrink):(360-\shrink) ,smooth,variable =\x, gray] plot ({-5+cos \x},{-sin \x}); 
\draw [->, domain=(180-\shrink):(0+\shrink) ,smooth,variable =\x, gray] plot ({-5+cos \x},{-sin \x}); 

\node[above] at ({-5+cos 270},{-sin 270}) {$c_0$};
\node[right] at ({-5+cos 0},{sin 0}) {$R(c_0)$};
\node[left] at ({-5+cos 180},{-sin 180}) {$L(c_0)$};

\node[above] at ({-5+cos 270},{sin 270}) {$f$};

\node[circle,draw=black, fill=cob, inner sep=0pt,minimum size=3pt] at ({cos 180},{-sin 180}) {};

\node[circle,draw=black, fill=cob, inner sep=0pt,minimum size=3pt] at ({cos 0},{-sin 0}) {};
\draw [ <-,domain=(180+\shrink):(360-\shrink) ,smooth,variable =\x, gray] plot ({cos \x},{-sin \x}); 
\draw [->, domain=(180-\shrink):(0+\shrink) ,smooth,variable =\x, gray] plot ({cos \x},{-sin \x}); 

\node[below] at ({cos 270},{sin 270}) {$_L\mathcal{D}_R$};

\node[above] at ({cos 270},{-sin 270}) {$\mathcal{C}$};


\draw[->, decorate, decoration={
    zigzag,
    segment length=4,
    amplitude=.9,post=lineto,
    post length=2pt
}] (-2.75,0)--(-1.5,0);
\end{tikzpicture}}
\caption{
General case}\label{fig:k_end_trace_twisted_1}
\end{subfigure}
\caption{Graphical representation of \eqref{eq:k_end_trace_1}. An endomorphism is ``closed up'' via a bar construction.}
\end{figure}

The construction of the zig-zag \eqref{eq:levelwise_trace_zigzag} works identically for a bimodule arising from a twisting
$\lMr{L}{\cC/\cD}{R}$ of spectral Waldhausen categories.
 In this case, the map \eqref{eq:k_end_trace_0} with $\cX = \lMr{L}{(w_\bullet S_\bullet^{(n)} \cD)}{R}$ defines a zig-zag of multisimplicial orthogonal spectra
\begin{equation}\label{eq:levelwise_trace_zigzag_twisted}
\begin{aligned}
\Sigma^{\infty} \ob \End(\lMr{L}{(w_{\bullet}S^{(n)}_{\bullet}
  \spcat{\cC}/w_{\bullet}S^{(n)}_{\bullet}
  \spcat{\cD})}{R}) &\arr \THH(w_{\bullet}S^{(n)}_{\bullet}
\spcat{\cC};\lMr{L}{(w_\bullet S_\bullet^{(n)} \cD)}{R}) \\
&\overset{\simeq}{\longleftarrow} (S^1_\bullet)^{\sma n} \sma \THH(\spcat{\cC};\lMr{L}{\cD}{R}).
\end{aligned}
\end{equation}

The following lemma follows directly from the definitions and \cref{thm:THH_additivity}:

\begin{lem}\label{zigzag_respects_simplicial_str}
The maps in the zig-zag \eqref{eq:levelwise_trace_zigzag_twisted} of multisimplicial orthogonal spectra commute with the $\Sigma_n$-actions and
the identifications that remove a simplicial direction when its index is equal
to 1.  In other words the given maps produce a zig-zag of $\Sigma_\Delta$-diagrams of
simplicial orthogonal spectra.
\end{lem}

Another way of saying this is that
$(S^1_\bullet)^{\sma n} \sma \THH(\spcat{\cC})$ is the free
$\Sigma_\Delta$-diagram on the spectrum $\THH(\spcat\cC)$ at level $(0;)$, and
$\bigvee \iota_{i_1,\ldots,i_k}$ agrees with the map that arises from the
free-forgetful adjunction.

Taking the geometric realization of these multisimplicial orthogonal spectra
gives a zig-zag of bispectra.  At level $n$ in the symmetric spectrum
direction, the zig-zag of bispectra is
\begin{equation}\label{eq:levelwise_trace_after_realization}
\begin{aligned}
  \abs{\Sigma^{\infty}
    \ob \End(\lMr{L}{(w_{\bullet}S^{(n)}_{\bullet}
  \spcat{\cC}/w_{\bullet}S^{(n)}_{\bullet}
      \spcat{\cD})}{R})} &\arr
  \abs{\THH(w_{\bullet}S^{(n)}_{\bullet} \spcat{\cC};
    \lMr{L}{(w_\bullet S^{(n)}_{\bullet}\cD)}{R})} \\
    &\overset{\simeq}{\longleftarrow} \Sigma^n \THH(\spcat{\cC};\lMr{L}{\cD}{R}).
\end{aligned}
\end{equation}
There is a canonical identification of sets
\[
\ob \End(\lMr{L}{(w_{\bullet}S^{(n)}_{\bullet}
  \spcat{\cC}/w_{\bullet}S^{(n)}_{\bullet}\spcat{\cD})}{R}) = \ob w_{\bullet}S^{(n)}_{\bullet}\End(\lMr{L}{\spcat{\cC}/\spcat{\cD}}{R}).
\]
This identifies the bispectrum
on the left of \eqref{eq:levelwise_trace_after_realization} with the orthogonal
suspension spectrum of the symmetric spectrum $K(\End(\lMr{L}{\spcat{\cC}/\spcat{\cD}}{R}))$.  On the
other hand, the spectrum on the right of
\eqref{eq:levelwise_trace_after_realization} is the symmetric suspension
spectrum of the orthogonal spectrum $\THH(\spcat{\cC};\lMr{L}{\cD}{R})$.

Applying the (left-derived) prolongation functor from
	\cite[\swcref{Proposition A.7}{\cref{swc-quillen-adjoints}}]
	{clmpz-specWaldCat}
 to these bispectra, we get a zig-zag of orthogonal spectra
\begin{equation}\label{eq:trc_1_step_back}
\bbP K(\End(\lMr{L}{\cC/\cD}{R})) \arr \bbP\abs{\THH(w_{\bullet}S^{*}_{\bullet}
\spcat{\cC}; \lMr{L}{(w_{\bullet}S^{*}_{\bullet} \spcat{\cD})}{R})} \overset{\simeq}{\longleftarrow} \THH(\cC;\lMr{L}{\cD}{R}).
\end{equation}
The first term in \eqref{eq:trc_1_step_back} is the prolongation of $K(\End(\lMr{L}{\cC/\cD}{R}))$ from symmetric to orthogonal spectra.

\begin{defn} \label{twisted_dennis_trace} The \textbf{Dennis trace map}
  associated to a twisting $\lMr{L}{\cC/\cD}{R}$ is obtained by choosing an
  inverse to the wrong-way map in \eqref{eq:trc_1_step_back}, defining a map
\begin{equation}\label{eq:trc_ors_k_thh}
	\trc \colon \bbP K(\End(\lMr{L}{\cC/\cD}{R})) \arr \THH(\spcat{\cC};\lMr{L}{\spcat{\cD}}{R})
\end{equation}
in the homotopy category of orthogonal spectra,
or, equivalently, in the homotopy category of symmetric spectra
\begin{equation}\label{eq:trc_ss_k_thh}
  \trc \colon K(\End(\lMr{L}{\cC/\cD}{R})) \arr \bbU\THH(\spcat{\cC};\lMr{L}{\cD}{R})
\end{equation}
to the underlying symmetric spectrum of $\THH$.
\end{defn}

\begin{rmk}
	It is unnecessary to check that the multisimplicial objects above are Reedy cofibrant, because the realization can be automatically left-derived without losing the symmetric spectrum structure (\cite[\swcref{Theorem 6.4}{\cref{swc-properness-thm}}]{clmpz-specWaldCat}). See \cite[\swcref{Remark 7.14}{\cref{swc-rmk:properness_for_zigzag}}]{clmpz-specWaldCat} for additional discussion. The upshot is that this construction of the Dennis trace is insensitive to the choice of model for $\THH$.
      \end{rmk}

\begin{obs}
  Consider the case when $\cC = \cD$ and $L = R = \id$.
The Dennis trace map on $K(\End(\uncat\cC))$ extends the Dennis trace on $K(\uncat{\cC})$, as constructed in e.g. \cite{blumberg_mandell_unpublished,dundas_goodwillie_mccarthy}, in the sense that the following diagram commutes:
\[
\xymatrix{
  K(\uncat{\cC}) \ar[d]_-{\trc} \ar[r]^-{\iota_1} & K(\End(\uncat{\cC})) \ar[dl]^-{\trc} \\
  \THH(\spcat{\cC}).
} \]
\end{obs}

To see this, observe that other authors construct the Dennis trace map in the same way that we did for $K(\End(\uncat\cC))$, except using the map \eqref{eq:k_end_trace_0} with $\cX(c_0, c_0) = \bbS$ instead of \eqref{eq:k_end_trace_1}. The two maps agree along the inclusion of the identity morphisms, and tracing through the construction above this becomes the map of $K$-theory spectra $\iota_1 \colon K(\uncat{\cC}) \arr K(\End(\uncat{\cC}))$ induced by the exact functor $\iota_1$ from beneath \cref{def:category_of_endomorphisms}.

\begin{rmk}\label{cyclic_ktheory}
The same argument shows that the diagram
\[
\xymatrix{
	K(\uncat{\cC}) \ar[d]_-{0} \ar[r]^-{\iota_0} & K(\End(\uncat{\cC})) \ar[dl]^-{\trc} \\
	\THH(\spcat{\cC})
} \]
commutes.
As a result, the Dennis trace factors through the mapping cone of $\iota_0$, often called {\bf cyclic $K$-theory} (or the reduced $K$-theory of endomorphisms):
\[ K^\cyc(\uncat\cC) = \widetilde K(\End(\uncat\cC)) = K(\End(\uncat\cC)) \ / \ \iota_0 K(\uncat\cC). \]
\end{rmk}

If we return to the case of general $\cD$, the inclusion of identity endomorphisms $\iota_1$ becomes meaningless, but the inclusion of zero endomorphisms $\iota_0$ is still
 defined, and there is a commutative diagram
	\[
	\xymatrix{
		K(\uncat{\cC}) \ar[d]_-{0} \ar[r]^-{\iota_0} & K\End\left(\tw{\cC}{L}{R}{\cD}\right) \ar[dl]^-{\trc} \\
		\THH\left(\spcat{\cC};\tensor[_L]{\spcat\cD}{_R}\right).
	} \]
Thus the Dennis trace descends to a map out of cyclic $K$-theory,
\[
\widetilde K\End\left(\tw{\cC}{L}{R}{\cD}\right) = K\End\left(\tw{\cC}{L}{R}{\cD}\right) \ / \ \iota_0 K(\uncat\cC) \arr \THH\left(\spcat{\cC};\tensor[_L]{\spcat\cD}{_R}\right).
\]
In particular, for a ring spectrum $A$ and bimodule $M$ the Dennis trace defines a map
\[ K^{\cyc}(A; M) = \widetilde K(A;M) \xarr{\trc} \THH(\tensor[^A]{\Perf}{};\tensor[^A]{\Mod}{_M}) \overset\sim\longleftarrow \THH(A;M). \]

\begin{rmk}\label{comparison_to_LM12}
	We sketch an argument that this agrees with the trace defined in \cite[9.2]{LM12}, see also \cite[V]{dundas_goodwillie_mccarthy}. The idea is to use a version of \cite{dmpsw} with bimodule coefficients to turn our smash products into B\" okstedt smash products. Then the relevant bispectrum is an $\Omega$-spectrum in the $\THH$ direction, hence the prolongation is equivalent to the functor that restricts to the symmetric spectrum direction. After these manipulations, the inclusion of endomorphisms map \eqref{eq:k_end_trace_1} is the same as the one in \cite[9.1]{LM12}.
\end{rmk}

\begin{example}
	Taking $\spcat\cC = \tensor[^A]{\Perf}{}$ for a ring spectrum $A$, we get maps
	\[ \xymatrix @R=1em {
		K(A) \ar@{=}[d] & \widetilde K\End(A) \ar@{=}[d] & \ar[d]^-\sim \THH(A) \\
		K(\tensor[^A]{\Perf}{}) \ar[r]^-{\iota_1} & \widetilde K\End(\tensor[^A]{\Perf}{}) \ar[r]^-{\trc} & \THH(\tensor[^A]{\Perf}{})
	} \]
whose composite agrees with the Dennis trace map $K(A) \to \THH(A)$ studied previously, e.g. \cite{dundas_mccarthy,dundas_goodwillie_mccarthy,madsen_survey}.
\end{example}

\begin{prop}\label{lem:pi_0_dennis_trace}
Let $f \colon L(a) \to R(a)$ be an object of the twisted endomorphism category $\End\left(\tw{\cC}{L}{R}{\cD}\right)$.  The image of $[f] \in K_0\End\left(\tw{\cC}{L}{R}{\cD}\right)$ under the Dennis trace
is the homotopy class of the composite
\[
\bbS \xarr{[f]} \spcat{\cD}(L(a), R(a)) \arr \bigvee_{c_0 \in \cC} \spcat{\cD}(L(c_0), R(c_0)) \xarr{\text{$0$-skeleton}} \THH(\spcat\cC; \tensor[_L]{\spcat\cD}{_R})
\]
which includes $f$ as a $0$-simplex in the cyclic bar construction.

\end{prop}
\begin{proof}
Consider the following commutative diagram, where the top row maps to the bottom
row by mapping symmetric spectrum level 0 and simplicial level 0 into the
zig-zag of bispectra:
\[
	\xymatrix @R=1.5em{
		\Sigma^{\infty} \ob \End(w_0 \lMr{L}{\cC/\cD}{R}) \ar[d] \ar[r] & \THH(w_0\cC; \lMr{L}{w_0\cD}{R}) \ar[d] & \ar[l]_-\cong \Sigma^0 \THH(\cC;\lMr{L}{\spcat{\cD}}{R}) \ar[d]^-\cong \\
		\bbP K(\End(\lMr{L}{\cC/\cD}{R})) \ar[r] & \bbP|\THH(w_\bullet
                S^{\ast}_{\bullet,\ldots,\bullet} \spcat{\cC};\lMr{L}{(w_\bullet
                  S^*_{\bullet,\ldots,\bullet} \cD)}{R})| & \ar[l]_-\cong \THH(\spcat{\cC};\lMr{L}{\cD}{R})
	}
\]
The vertical map on the left is surjective on $\pi_0$ by the standard presentation for $K_0$ of a Waldhausen category. By the construction of the Dennis trace, the first horizontal map in the top row is the inclusion of endomorphisms map, so the conclusion follows.
\end{proof}

\begin{example}\label{ex:trace_is_trace}
	If $A$ is a ring spectrum, each perfect left $A$-module $P$ defines a class $[P] \in K_0(A)$. By \cref{lem:pi_0_dennis_trace}, its image in $\pi_0 \THH(\tensor[^A]{\Perf}{})$ is the inclusion of the 0-simplex corresponding to the identity map of $P$ into the cyclic bar construction. By \cite[\S 7]{cp},
	the image of this class under the Morita equivalence
\[\pi_0 \THH(\tensor[^A]{\Perf}{}) \cong \pi_0\THH(A)\]
 is the Euler characteristic of $P$. More generally, each endomorphism $f\colon P \arr P$ defines a class $[f] \in K_0 \End(A)$ whose image in $\pi_0\THH(A)$ is the trace of $f$, by \cref{lem:pi_0_dennis_trace} and \cite[7.11]{cp}.
\end{example}

\begin{example}\label{ex:twisted_trace_is_twisted_trace}
	Every twisted endomorphism $f\colon P \to M \sma_A P$ of a perfect left $A$-module $P$ defines a class $[f] \in K_0(A;M)$ whose image in $\pi_0\THH(\tensor[^A]{\Perf}{};\tensor[^A]{\Mod}{_M})$ is the inclusion of the 0-simplex corresponding to $f$ in the cyclic bar construction. By \cite[7.4]{cp}, applied as in the proof of \cite[7.11]{cp}, its image under the Morita equivalence
\[\pi_0 \THH(\tensor[^A]{\Perf}{};\tensor[^A]{\Mod}{_M}) \cong \pi_0\THH(A;M)\]
 is the bicategorical trace
	\[
	\tr(f) \colon \bbS = \sh{\bbS} \arr \sh{M} = \THH(A;M).
	\]
\end{example}

\section{The equivariant Dennis trace}\label{sec:equiv_dennis_trace}

In order to build a trace to topological restriction homology ($\TR$), we
construct $C_r$-equivariant refinements of $\THH$ and the $K$-theory of
endomorphisms whose fixed points for varying $r$ are linked together. This
requires some generalizations of the constructions and theorems of the previous
three sections.  For a short review of equivariant spectra and the notation we
are using see \S\ref{sec:equivariant_review}.  We then define the equivariant Dennis trace and give a description of its effect on $\pi_0$ in terms of the bicategorical trace
\[\tr(f_{r} \circ \dotsm \circ f_{1})\] of a composite twisted endomorphism (\cref{prop:main_identify_pi_0_TR}).

\subsection{$r$-fold endomorphisms and $\THH^{(r)}$}

\begin{defn} \label{def:twisted_endo} Given a spectral Waldhausen category
  $\spcat\cC$, let $\uncat\cC^{\times r}$ denote the $r$-fold product of the
  base category, with Waldhausen structure determined coordinate-wise.  Let
  $\rho$ denote the rotation functor
	\[
	\rho \colon (c_1,\ldots,c_r) \longmapsto (c_2,\ldots,c_r,c_1).
	\]
\end{defn}

\noindent
Given a twisting $_L\cC/\cD_R$ of $\spcat{\cC}$, let
$R^{r}_{\rho} = R^r \circ \rho = \rho \circ R^r \colon \uncat\cC^{\times r} \rto \uncat\cD^{\times r}$
denote the exact functor
\[ R^{r}_{\rho}(c_1,\ldots,c_r) = (R(c_2), \ldots,R(c_{r}), R(c_1)). \]

\begin{defn} \label{def:rfoldTwisted}
	Define the \textbf{$r$-fold twisted endomorphism category of $\lMr{L}{\cC/\cD}{R}$} by
	\[\End^{(r)}\left(\tw{\cC}{L}{R}{\cD}\right)
          \coloneqq \End\left(\tensor[_{L^r}]{ \left(\cC^{\times r}/\cD^{\times
                  r}\right)}{_{R^{r}_{\rho}}}\right).\] There is an exact
        functor
        \[\Delta_r\colon\End(\lMr{L}{\cC/\cD}{R}) \to \End^{(r)}(\lMr{L}{\cC/\cD}{R})\]
        taking $(a, f)$ to $((a,\ldots,a), (f,\ldots,f))$; we refer to this as
        the \textbf{duplication functor}.
\end{defn}

An object of this category consists of an $r$-tuple of objects $(a_1, \ldots, a_r)$ in $\uncat\cC$ and an $r$-tuple of morphisms in $\uncat\cD$
	\[ \bigl( f_1 \colon L(a_1) \arr R(a_2), \; \ldots, \; f_{r - 1} \colon L(a_{r-1}) \arr R(a_r), \; f_{r} \colon L(a_r) \arr R(a_1) \bigr). \]
	See \cref{fig:obj_twisted}.
	A morphism $(a_i, f_i) \arr (b_i, g_i)$
	is an $r$-tuple of morphisms $(t_i \colon a_i \arr b_i)$ such that
	\[
	R(t_{i + 1}) \circ f_i = g_i \circ L(t_i) \qquad \text{(indices taken mod $r$).}
	\]
	See \cref{fig:mor_twisted}.

\begin{figure}[h]
       \begin{subfigure}[b]{.9\textwidth}
\xymatrix @R=3em @!C=0em{
		& L(a_1) \ar[rr]^-{f_1} && R(a_2) &&&\cdots&&& L(a_{r-1}) \ar[rr]^-{f_{r-1}} && R(a_r) && L(a_r) \ar[rr]^-{f_r} && R(a_1) & \\
		a_1 \ar@{~>}[ur] &&&& \ar@{~>}[ul] a_2 \ar@{~>}[ur] &&&& \ar@{~>}[ul] a_{r-1} \ar@{~>}[ur] &&&& \ar@{~>}[ul] a_r \ar@{~>}[ur] &&&& \ar@{~>}[ul] a_1
	}
	\caption{Objects of $\End^{(r)}\left(\tw{\cC}{L}{R}{\cD}\right)$}\label{fig:obj_twisted}
	\end{subfigure}
       \begin{subfigure}[b]{.8\textwidth}
\xymatrix @R=1.5em@C=2em{
		a_1 \ar[ddd]^-{t_1} \ar@{~>}[dr] &&  & \ar@{~>}[dl] a_2 \ar[ddd]^-{t_2} &
		\cdots &
		a_r \ar[ddd]^-{t_r} \ar@{~>}[dr] &  & & \ar@{~>}[dl] a_1 \ar[ddd]^-{t_1}
		\\
		& L(a_1) \ar[d]^-{L(t_1)} \ar[r]^-{f_1} & R(a_2) \ar[d]^-{R(t_2)} &  &
		\cdots &
		& L(a_r) \ar[d]^-{L(t_r)} \ar[r]^-{f_r} & R(a_1) \ar[d]^-{R(t_1)} &
		\\
		& L(b_1) \ar[r]_-{g_1} & R(b_2) &  &
		\cdots &
		& L(b_r) \ar[r]_-{g_r} & R(b_1) &
		\\
		b_1 \ar@{~>}[ru] && & \ar@{~>}[lu] b_2 &
		\cdots &
		b_r \ar@{~>}[ru] &  &  & \ar@{~>}[ul] b_1.
	}
	\caption{Morphisms of $\End^{(r)}\left(\tw{\cC}{L}{R}{\cD}\right)$}\label{fig:mor_twisted}
	\end{subfigure}
	\caption{}
\end{figure}

\begin{defn}\label{ex:thh_n}

	Suppose $\spcat\cC$ is a pointwise cofibrant spectral category and $\spcat\cX$ is a pointwise cofibrant $(\spcat\cC,\spcat\cC)$ bimodule.
	For each $r \geq 1$, let $\spcat\cC^{\sma r}$ be the spectral category with object set $(\ob\spcat\cC)^{\times r}$ and mapping spectra
	\[ \spcat\cC^{\sma r}((a_1,\ldots,a_r),(b_1,\ldots,b_r)) = \bigwedge_{i=1}^r \spcat\cC(a_i,b_i). \]
	As in \cref{def:twisted_endo} we let
	\[ \rho\colon \spcat{\cC}^{\sma r} \to \spcat\cC^{\sma r}, \qquad \rho(c_1,\ldots,c_r) = (c_2,\ldots,c_r,c_1) \]
	be the spectral functor that rotates the smash product factors.

	Similarly, let $\spcat\cX^{\sma r}$ denote the $(\spcat\cC^{\sma r},\spcat\cC^{\sma r})$-bimodule whose value on each pair of $r$-tuples is the evident $r$-fold smash product. Twisting on the right by $\rho$ gives another $(\spcat\cC^{\sma r},\spcat\cC^{\sma r})$-bimodule $\spcat\cX^{\sma r}_\rho$. We define \textbf{$r$-fold topological Hochschild homology} by the formula
	\[ \THH^{(r)}(\spcat\cC;\spcat\cX) \coloneqq \THH(\spcat\cC^{\sma r};\spcat\cX^{\sma r}_\rho). \]
        Rotating the coordinates of $\spcat\cX^{\sma r}_\rho$ gives an isomorphism of
bimodules $\rho \colon \spcat\cX^{\sma r}_\rho \arr \spcat\cX^{\sma r}_\rho$
over the map of spectral categories $\rho \colon \spcat\cC^{\sma r} \arr
\spcat\cC^{\sma r}$. In other words, we get an action of the cyclic group $C_r$
on the pair consisting of the spectral category $\spcat\cC^{\sma r}$ and the
bimodule $\spcat\cX^{\sma r}_\rho$. This makes $r$-fold $\THH$ into an
orthogonal $C_r$-spectrum.
\end{defn}

\begin{prop}\label{thh_n_derived}
	When $\spcat\cC$ and $\spcat\cX$ are pointwise cofibrant, $\THH^{(r)}$
        is cofibrant and the Hill--Hopkins--Ravenel norm diagonal of
        \cref{norm_diagonal} induces isomorphisms of orthogonal $C_s$-spectra
	\[ \Phi^{C_r} \THH^{(rs)}(\spcat\cC;\spcat\cX) \cong \THH^{(s)}(\spcat\cC;\spcat\cX) \]
	for all $r,s \geq 1$.
\end{prop}

\begin{proof}
	For the cofibrancy statement it suffices to check that the latching maps of the simplicial spectrum defining $\THH^{(r)}$ are cofibrations of orthogonal $C_r$-spectra.
These latching maps are cofibrations  since the $r$-fold smash power turns cofibrations into equivariant cofibrations \cite[4.11]{malkiewich_thh_dx}.

	For the second part,
	we commute $\Phi^{C_r}$ with the realization and apply the norm diagonal $D_r$ at every simplicial level:

\noindent
 \begin{small}
\begin{align*}
	\Phi^{C_r} &\left| [k] \mapsto \bigvee_{(a^0_i),\dotsc,(a^k_i)} \spcat\cC^{\sma rs}((a^0_i),(a^1_i)) \sma \dotsm \sma \spcat\cC^{\sma rs}((a^{k-1}_i),(a^k_i)) \sma \spcat\cX^{\sma rs}_\rho((a^k_i),(a^0_i)) \right| \\
	\cong &\left| [k] \mapsto \bigvee_{(b^0_i),\dotsc,(b^k_i)} \Phi^{C_r}\left(\spcat\cC^{\sma rs}((b^0_i)^{\times r},(b^1_i)^{\times r}) \sma \dotsm \sma \spcat\cC^{\sma rs}((b^{k-1}_i)^{\times r},(b^k_i)^{\times r}) \sma \spcat\cX^{\sma rs}_\rho((b^k_i)^{\times r},(b^0_i)^{\times r})\right) \right| \\
	\cong &\left| [k] \mapsto \bigvee_{(b^0_i),\dotsc,(b^k_i)} \spcat\cC^{\sma s}((b^0_i),(b^1_i)) \sma \dotsm \sma \spcat\cC^{\sma s}((b^{k-1}_i),(b^k_i)) \sma \spcat\cX^{\sma s}_\rho((b^k_i),(b^0_i)) \right|. \\
	\end{align*}
\end{small}

\noindent Here $(a^j_i) = (a^j_1,a^j_2,\dotsc,a^j_{rs})$ ranges over $rs$-tuples of objects of $\spcat\cC$, $(b^j_i) = (b^j_1,b^j_2,\dotsc,b^j_s)$ ranges over $s$-tuples of objects of $\spcat\cC$, and
\[(b^j_i)^{\times r} = (b^j_1, \dotsc, b^j_s, \dotsc, b^j_1, \dotsc, b^j_s)\]
 is the $rs$-tuple obtained by duplicating an $s$-tuple $r$ times. It remains to check that the diagonal respects the faces and degeneracy maps, and the $C_s$-action. For degeneracies and every face but the first, this follows by naturality of the diagonal. For the first face this follows from \cite[3.26]{malkiewich_thh_dx}, and for the $C_s$-action it follows from \cite[3.27]{malkiewich_thh_dx}. In other words, both are consequences of the rigidity theorem \cite[1.2]{malkiewich_thh_dx}.
\end{proof}

\begin{rmk}\label{rmk:unwinding_subdivision}
	When $\cX = \cC$, the spectrum $\THH^{(r)}(\cC;\mc{C}) $ is isomorphic to $\THH(\cC)$ using $r$-fold subdivision, and this is the cyclotomic structure constructed in \cite[4.6]{malkiewich_thh_dx}. It is equivalent to B\" okstedt's cyclotomic structure on $\THH(\cC)$ by the main result of \cite{dmpsw}.
\end{rmk}

The next result formalizes the observation that if we unwind the $r$-fold
topological Hochschild homology spectrum we get the construction illustrated in
\cref{fig:thhr}.
\begin{figure}[h]
	\begin{tikzpicture}
\draw [ domain=90:270 ,smooth,variable =\x, gray] plot ({cos \x},{sin \x}); 
\draw [ domain=270:450 ,smooth,variable =\x, gray] plot ({cos \x},{sin \x}); 

\def\aa{0}
\def\ab{(\aa+60)}
\def\ac{(\ab+60)}
\def\ad{(\ac+60)}
\def\ae{(\ad+60)}
\def\af{(\ae+60)}

\def\ba{30}
\def\bb{(\ba+60)}
\def\bc{(\bb+60)}
\def\bd{(\bc+60)}
\def\be{(\bd+60)}
\def\bf{(\be+60)}

\def\r{1.3}
\def\ra{1.2}

\node at ({\ra* cos \aa},{\ra* sin \aa}) {$\spcat\cC$};
\node at ({\ra* cos \ab},{\ra* sin \ab}) {$\spcat\cC$};
\node at ({\ra* cos \ac},{\ra* sin \ac}) {$\spcat\cC$};
\node at ({\ra* cos \ad},{\ra* sin \ad}) {$\spcat\cC$};
\node at ({\ra* cos \ae},{\ra* sin \ae}) {$\spcat\cC$};
\node at ({\ra* cos \af},{\ra* sin \af}) {$\spcat\cC$};

\node at ({\r* cos \ba},{\r* sin \ba}) {$\spcat\cX$};
\node at ({\r* cos \bb},{\r* sin \bb}) {$\spcat\cX$};
\node at ({\r* cos \bc},{\r* sin \bc}) {$\spcat\cX$};
\node at ({\r* cos \bd},{\r* sin \bd}) {$\spcat\cX$};
\node at ({\r* cos \be},{\r* sin \be}) {$\spcat\cX$};
\node at ({\r* cos \bf},{\r* sin \bf}) {$\spcat\cX$};

\node[circle,draw=black, fill=coc, inner sep=0pt,minimum size=3pt] at ({cos \ba},{sin \ba}) {};
\node[circle,draw=black, fill=coc, inner sep=0pt,minimum size=3pt] at ({cos \bb},{sin \bb}) {};
\node[circle,draw=black, fill=coc, inner sep=0pt,minimum size=3pt] at ({cos \bc},{sin \bc}) {};
\node[circle,draw=black, fill=coc, inner sep=0pt,minimum size=3pt] at ({cos \bd},{sin \bd}) {};
\node[circle,draw=black, fill=coc, inner sep=0pt,minimum size=3pt] at ({cos \be},{sin \be}) {};
\node[circle,draw=black, fill=coc, inner sep=0pt,minimum size=3pt] at ({cos \bf},{sin \bf}) {};
\end{tikzpicture}
	\caption{$\THH^{(r)}(\spcat\cC;\spcat\cX) $}\label{fig:thhr}
\end{figure}

\begin{prop}\label{unwinding}
	After forgetting the $C_r$-action, there is a natural equivalence
	\[ \THH^{(r)}(\spcat\cC;\spcat\cX) \simeq \THH(\spcat\cC;\spcat\cX \odot \dotsm \odot \spcat\cX) \]
	with $r$ copies of $\spcat\cX$ on the right.
\end{prop}

\begin{proof}
	The spectrum $\THH(\spcat\cC;\spcat\cX^{\odot r})$ is built using $r$ bar constructions, the final one being the cyclic bar construction, so it is the realization of an $r$-fold multisimplicial spectrum. Taking the diagonal of this multisimplicial spectrum, we identify the resulting simplicial spectrum with the one for $\THH^{(r)}(\spcat\cC;\spcat\cX)$ by regrouping copies of $\spcat\cC$ and $\spcat\cX$. Note that the use of the cyclic action $\rho$ in the definition of $\THH^{(r)}$ is essential to this argument.
      \end{proof}

\begin{rmk}\label{thh_n_morita_invariance}
	As a consequence of \cref{equivalence_on_twisted_thh,thh_n_derived,unwinding,equivs_measured_on_geometric_fp}, the construction $\THH^{(r)}(\cC;\tensor[_L]{\cD}{_R})$ sends Morita equivalences in the $\cC$ variable and Dwyer--Kan embeddings in the $\cD$ variable to equivalences of $C_r$-spectra. Consequently, the map
	\[ \THH^{(r)}(A;M) \stackrel\sim\arr \THH^{(r)}( \tensor[^A]{\Perf}{} ; \tensor[^A]{\Mod}{_M} ) \]
	induced by the Morita equivalence $A \arr \tensor[^A]{\Perf}{}$ (\cref{ex:classic_morita_equiv}) is an equivalence of $C_r$-spectra.
      \end{rmk}

\begin{thm}[Additivity of $\THH^{(r)}$]\label{thm:THH_n_additivity}
	Let $\spcat{\cC}$ be a spectral Waldhausen category. Then the maps $\iota_{i_1,\ldots,i_n}$ defined above \cref{thm:THH_additivity} induce an equivalence of $C_r$-spectra
	\[
	\bigvee_{\substack{i_1,\dotsc, i_n \\ 1 \leq i_j \leq k_j}} \THH^{(r)}(\spcat{\cC})
	\stackrel\sim\arr \THH^{(r)}(w_{k_0}S^{(n)}_{k_1, \dotsc, k_n} \spcat{\cC}).
	\]
	For any twisting $L,R\colon \spcat\cC \rightrightarrows \spcat\cD$, the maps $\iota_{i_1,\ldots,i_n}$ also induce an equivalence of $C_r$-spectra
	\[
	\bigvee_{\substack{i_1,\dotsc, i_n \\ 1 \leq i_j \leq k_j}} \THH^{(r)}(\spcat{\cC};\tensor[_L]{\spcat{\cD}}{_R})
	\stackrel\sim\arr \THH^{(r)}(w_{k_0}S^{(n)}_{k_1, \dotsc, k_n} \spcat{\cC};\tensor[_L]{(w_{k_0}S^{(n)}_{k_1, \dotsc, k_n} \spcat{\cD})}{_R}).
	\]
\end{thm}

\begin{figure}[h]
     \centering
	\hspace{1in}
     \begin{subfigure}[b]{0.25\textwidth}
         \centering
	\tdplotsetmaincoords{90}{70}
         \begin{tikzpicture}[tdplot_main_coords]

\def\ra{1.5}
\def\aa{10}
\def\a1{(\aa+60)}
\def\ab{(\a1+50)}
\def\ac{(\ab+70)}
\def\ad{(\ac+60)}
\def\ae{(\ad+60)}
\def\af{(\ae+60)}
\draw [ domain=\ac:\af ,smooth,variable =\x, gray] plot ({(\ra+.5)*sin \x},.5,{(\ra+.5)* cos \x}); 
\draw [ domain=\ac:\af,smooth,variable =\x, gray] plot ({(\ra+.75)*sin \x},.25,{(\ra+.75)* cos \x}); 
\draw [ domain=\ac:\af ,smooth,variable =\x, gray] plot ({(\ra+.25)*sin \x},.75,{(\ra+.25)* cos \x}); 
\draw [ domain=\ac:\af,smooth,variable =\x, gray] plot ({(\ra+1)*sin \x},0,{(\ra+1)* cos \x}); 
\draw [ domain=\ac:\af ,smooth,variable =\x, gray] plot ({\ra*sin \x},1,{\ra* cos \x}); 

\filldraw[draw=coa, very thick,fill=coa, opacity=.9] ({\ra*sin \aa},0,{\ra* cos \aa})--({(\ra+1)*sin \aa},0,{(\ra+1)* cos \aa})--({\ra*sin \aa},1,{\ra* cos \aa})--({\ra*sin \aa},0,{\ra* cos \aa});

\filldraw[draw=coa, very thick,fill=coa, opacity=.9] ({\ra*sin \a1},0,{\ra* cos \a1})--({(\ra+1)*sin \a1},0,{(\ra+1)* cos \a1})--({\ra*sin \a1},1,{\ra* cos \a1})--({\ra*sin \a1},0,{\ra* cos \a1});

\filldraw[draw=coa, very thick,fill=coa, opacity=.9] ({\ra*sin \ab},0,{\ra* cos \ab})--({(\ra+1)*sin \ab},0,{(\ra+1)* cos \ab})--({\ra*sin \ab},1,{\ra* cos \ab})--({\ra*sin \ab},0,{\ra* cos \ab});

\filldraw[draw=coa, very thick,fill=coa, opacity=.9] ({\ra*sin \ac},0,{\ra* cos \ac})--({(\ra+1)*sin \ac},0,{(\ra+1)* cos \ac})--({\ra*sin \ac},1,{\ra* cos \ac})--({\ra*sin \ac},0,{\ra* cos \ac});

\filldraw[draw=coa, very thick,fill=coa, opacity=.9] ({\ra*sin \ad},0,{\ra* cos \ad})--({(\ra+1)*sin \ad},0,{(\ra+1)* cos \ad})--({\ra*sin \ad},1,{\ra* cos \ad})--({\ra*sin \ad},0,{\ra* cos \ad});

\filldraw[draw=coa, very thick,fill=coa, opacity=.9] ({\ra*sin \ae},0,{\ra* cos \ae})--({(\ra+1)*sin \ae},0,{(\ra+1)* cos \ae})--({\ra*sin \ae},1,{\ra* cos \ae})--({\ra*sin \ae},0,{\ra* cos \ae});

\node[circle,draw=black, fill=coc, inner sep=0pt,minimum size=3pt] at ({(\ra+1)*sin \aa},0,{(\ra+1)* cos \aa}) {};
\node[circle,draw=black, fill=coc, inner sep=0pt,minimum size=3pt] at ({(\ra+.75)*sin \aa},.25,{(\ra+.75)* cos \aa}) {};
\node[circle,draw=black, fill=coc, inner sep=0pt,minimum size=3pt] at ({(\ra+.5)*sin \aa},.5,{(\ra+.5)* cos \aa}) {};
\node[circle,draw=black, fill=coc, inner sep=0pt,minimum size=3pt] at ({(\ra+.25)*sin \aa},.75,{(\ra+.25)* cos \aa}) {};
\node[circle,draw=black, fill=coc, inner sep=0pt,minimum size=3pt] at ({(\ra+0)*sin \aa},1,{(\ra+0)* cos \aa}) {};

\node[circle,draw=black, fill=coc, inner sep=0pt,minimum size=3pt] at ({(\ra+1)*sin \a1},0,{(\ra+1)* cos \a1}) {};
\node[circle,draw=black, fill=coc, inner sep=0pt,minimum size=3pt] at ({(\ra+.75)*sin \a1},.25,{(\ra+.75)* cos \a1}) {};
\node[circle,draw=black, fill=coc, inner sep=0pt,minimum size=3pt] at ({(\ra+.5)*sin \a1},.5,{(\ra+.5)* cos \a1}) {};
\node[circle,draw=black, fill=coc, inner sep=0pt,minimum size=3pt] at ({(\ra+.25)*sin \a1},.75,{(\ra+.25)* cos \a1}) {};
\node[circle,draw=black, fill=coc, inner sep=0pt,minimum size=3pt] at ({(\ra+0)*sin \a1},1,{(\ra+0)* cos \a1}) {};

\node[circle,draw=black, fill=coc, inner sep=0pt,minimum size=3pt] at ({(\ra+1)*sin \ab},0,{(\ra+1)* cos \ab}) {};
\node[circle,draw=black, fill=coc, inner sep=0pt,minimum size=3pt] at ({(\ra+.75)*sin \ab},.25,{(\ra+.75)* cos \ab}) {};
\node[circle,draw=black, fill=coc, inner sep=0pt,minimum size=3pt] at ({(\ra+.5)*sin \ab},.5,{(\ra+.5)* cos \ab}) {};
\node[circle,draw=black, fill=coc, inner sep=0pt,minimum size=3pt] at ({(\ra+.25)*sin \ab},.75,{(\ra+.25)* cos \ab}) {};
\node[circle,draw=black, fill=coc, inner sep=0pt,minimum size=3pt] at ({(\ra+0)*sin \ab},1,{(\ra+0)* cos \ab}) {};

\node[circle,draw=black, fill=coc, inner sep=0pt,minimum size=3pt] at ({(\ra+1)*sin \ac},0,{(\ra+1)* cos \ac}) {};
\node[circle,draw=black, fill=coc, inner sep=0pt,minimum size=3pt] at ({(\ra+.75)*sin \ac},.25,{(\ra+.75)* cos \ac}) {};
\node[circle,draw=black, fill=coc, inner sep=0pt,minimum size=3pt] at ({(\ra+.5)*sin \ac},.5,{(\ra+.5)* cos \ac}) {};
\node[circle,draw=black, fill=coc, inner sep=0pt,minimum size=3pt] at ({(\ra+.25)*sin \ac},.75,{(\ra+.25)* cos \ac}) {};
\node[circle,draw=black, fill=coc, inner sep=0pt,minimum size=3pt] at ({(\ra+0)*sin \ac},1,{(\ra+0)* cos \ac}) {};

\node[circle,draw=black, fill=coc, inner sep=0pt,minimum size=3pt] at ({(\ra+1)*sin \ad},0,{(\ra+1)* cos \ad}) {};
\node[circle,draw=black, fill=coc, inner sep=0pt,minimum size=3pt] at ({(\ra+.75)*sin \ad},.25,{(\ra+.75)* cos \ad}) {};
\node[circle,draw=black, fill=coc, inner sep=0pt,minimum size=3pt] at ({(\ra+.5)*sin \ad},.5,{(\ra+.5)* cos \ad}) {};
\node[circle,draw=black, fill=coc, inner sep=0pt,minimum size=3pt] at ({(\ra+.25)*sin \ad},.75,{(\ra+.25)* cos \ad}) {};
\node[circle,draw=black, fill=coc, inner sep=0pt,minimum size=3pt] at ({(\ra+0)*sin \ad},1,{(\ra+0)* cos \ad}) {};

\node[circle,draw=black, fill=coc, inner sep=0pt,minimum size=3pt] at ({(\ra+1)*sin \ae},0,{(\ra+1)* cos \ae}) {};
\node[circle,draw=black, fill=coc, inner sep=0pt,minimum size=3pt] at ({(\ra+.75)*sin \ae},.25,{(\ra+.75)* cos \ae}) {};
\node[circle,draw=black, fill=coc, inner sep=0pt,minimum size=3pt] at ({(\ra+.5)*sin \ae},.5,{(\ra+.5)* cos \ae}) {};
\node[circle,draw=black, fill=coc, inner sep=0pt,minimum size=3pt] at ({(\ra+.25)*sin \ae},.75,{(\ra+.25)* cos \ae}) {};
\node[circle,draw=black, fill=coc, inner sep=0pt,minimum size=3pt] at ({(\ra+0)*sin \ae},1,{(\ra+0)* cos \ae}) {};

\draw [ domain=\aa:\af ,smooth,variable =\x, cod, very thick] plot ({\ra*sin \x},0,{\ra* cos \x}); 
\draw [ domain=\aa:\ac ,smooth,variable =\x, gray] plot ({(\ra+.5)*sin \x},.5,{(\ra+.5)* cos \x}); 
\draw [ domain=\aa:\ac,smooth,variable =\x, gray] plot ({(\ra+.75)*sin \x},.25,{(\ra+.75)* cos \x}); 
\draw [ domain=\aa:\ac ,smooth,variable =\x, gray] plot ({(\ra+.25)*sin \x},.75,{(\ra+.25)* cos \x}); 
\draw [ domain=\aa:\ac,smooth,variable =\x, gray] plot ({(\ra+1)*sin \x},0,{(\ra+1)* cos \x}); 
\draw [ domain=\aa:\ac ,smooth,variable =\x, gray] plot ({\ra*sin \x},1,{\ra* cos \x});

\node[circle,draw=black, fill=cob, inner sep=0pt,minimum size=3pt] at ({\ra*sin \aa},0,{\ra* cos \aa}) {};
\node[circle,draw=black, fill=cob, inner sep=0pt,minimum size=3pt] at ({\ra*sin \a1},0,{\ra* cos \a1}) {};
\node[circle,draw=black, fill=cob, inner sep=0pt,minimum size=3pt] at ({\ra*sin \ab},0,{\ra* cos \ab}) {};
\node[circle,draw=black, fill=cob, inner sep=0pt,minimum size=3pt] at ({\ra*sin \ac},0,{\ra* cos \ac}) {};
\node[circle,draw=black, fill=cob, inner sep=0pt,minimum size=3pt] at ({\ra*sin \ad},0,{\ra* cos \ad}) {};
\node[circle,draw=black, fill=cob, inner sep=0pt,minimum size=3pt] at ({\ra*sin \ae},0,{\ra* cos \ae}) {};
\end{tikzpicture}
         \caption{}
         \label{fig:shadow_flag_mult}
     \end{subfigure}
     \hfill
     \begin{subfigure}[b]{0.25\textwidth}
         \centering
	\tdplotsetmaincoords{90}{70}
         \begin{tikzpicture}[tdplot_main_coords]

\def\ra{1.5}
\def\rb{2}
\def\aa{10}
\def\a1{(\aa+60)}
\def\ab{(\a1+50)}
\def\ac{(\ab+70)}
\def\ad{(\ac+60)}
\def\ae{(\ad+60)}
\def\af{(\ae+60)}
\draw [ domain=\ac:\af ,smooth,variable =\x, gray] plot ({(\ra+.5)*sin \x},.5*\rb,{(\ra+.5)* cos \x}); 
\draw [ domain=\ac:\af,smooth,variable =\x, gray] plot ({(\ra+.75)*sin \x},.25*\rb,{(\ra+.75)* cos \x}); 
\draw [ domain=\ac:\af ,smooth,variable =\x, gray] plot ({(\ra+.25)*sin \x},.75*\rb,{(\ra+.25)* cos \x}); 
\draw [ domain=\ac:\af,smooth,variable =\x, gray] plot ({(\ra+1)*sin \x},0*\rb,{(\ra+1)* cos \x}); 
\draw [ domain=\ac:\af ,smooth,variable =\x, gray] plot ({\ra*sin \x},1*\rb,{\ra* cos \x}); 
%
%
%
%
%
%

\node[circle,draw=black, fill=coc, inner sep=0pt,minimum size=3pt] at ({(\ra+1)*sin \aa},0*\rb,{(\ra+1)* cos \aa}) {};
\node[circle,draw=black, fill=coc, inner sep=0pt,minimum size=3pt] at ({(\ra+.75)*sin \aa},.25*\rb,{(\ra+.75)* cos \aa}) {};
\node[circle,draw=black, fill=coc, inner sep=0pt,minimum size=3pt] at ({(\ra+.5)*sin \aa},.5*\rb,{(\ra+.5)* cos \aa}) {};
\node[circle,draw=black, fill=coc, inner sep=0pt,minimum size=3pt] at ({(\ra+.25)*sin \aa},.75*\rb,{(\ra+.25)* cos \aa}) {};
\node[circle,draw=black, fill=coc, inner sep=0pt,minimum size=3pt] at ({(\ra+0)*sin \aa},1*\rb,{(\ra+0)* cos \aa}) {};

\node[circle,draw=black, fill=coc, inner sep=0pt,minimum size=3pt] at ({(\ra+1)*sin \a1},0*\rb,{(\ra+1)* cos \a1}) {};
\node[circle,draw=black, fill=coc, inner sep=0pt,minimum size=3pt] at ({(\ra+.75)*sin \a1},.25*\rb,{(\ra+.75)* cos \a1}) {};
\node[circle,draw=black, fill=coc, inner sep=0pt,minimum size=3pt] at ({(\ra+.5)*sin \a1},.5*\rb,{(\ra+.5)* cos \a1}) {};
\node[circle,draw=black, fill=coc, inner sep=0pt,minimum size=3pt] at ({(\ra+.25)*sin \a1},.75*\rb,{(\ra+.25)* cos \a1}) {};
\node[circle,draw=black, fill=coc, inner sep=0pt,minimum size=3pt] at ({(\ra+0)*sin \a1},1*\rb,{(\ra+0)* cos \a1}) {};

\node[circle,draw=black, fill=coc, inner sep=0pt,minimum size=3pt] at ({(\ra+1)*sin \ab},0*\rb,{(\ra+1)* cos \ab}) {};
\node[circle,draw=black, fill=coc, inner sep=0pt,minimum size=3pt] at ({(\ra+.75)*sin \ab},.25*\rb,{(\ra+.75)* cos \ab}) {};
\node[circle,draw=black, fill=coc, inner sep=0pt,minimum size=3pt] at ({(\ra+.5)*sin \ab},.5*\rb,{(\ra+.5)* cos \ab}) {};
\node[circle,draw=black, fill=coc, inner sep=0pt,minimum size=3pt] at ({(\ra+.25)*sin \ab},.75*\rb,{(\ra+.25)* cos \ab}) {};
\node[circle,draw=black, fill=coc, inner sep=0pt,minimum size=3pt] at ({(\ra+0)*sin \ab},1*\rb,{(\ra+0)* cos \ab}) {};

\node[circle,draw=black, fill=coc, inner sep=0pt,minimum size=3pt] at ({(\ra+1)*sin \ac},0*\rb,{(\ra+1)* cos \ac}) {};
\node[circle,draw=black, fill=coc, inner sep=0pt,minimum size=3pt] at ({(\ra+.75)*sin \ac},.25*\rb,{(\ra+.75)* cos \ac}) {};
\node[circle,draw=black, fill=coc, inner sep=0pt,minimum size=3pt] at ({(\ra+.5)*sin \ac},.5*\rb,{(\ra+.5)* cos \ac}) {};
\node[circle,draw=black, fill=coc, inner sep=0pt,minimum size=3pt] at ({(\ra+.25)*sin \ac},.75*\rb,{(\ra+.25)* cos \ac}) {};
\node[circle,draw=black, fill=coc, inner sep=0pt,minimum size=3pt] at ({(\ra+0)*sin \ac},1*\rb,{(\ra+0)* cos \ac}) {};

\node[circle,draw=black, fill=coc, inner sep=0pt,minimum size=3pt] at ({(\ra+1)*sin \ad},0*\rb,{(\ra+1)* cos \ad}) {};
\node[circle,draw=black, fill=coc, inner sep=0pt,minimum size=3pt] at ({(\ra+.75)*sin \ad},.25*\rb,{(\ra+.75)* cos \ad}) {};
\node[circle,draw=black, fill=coc, inner sep=0pt,minimum size=3pt] at ({(\ra+.5)*sin \ad},.5*\rb,{(\ra+.5)* cos \ad}) {};
\node[circle,draw=black, fill=coc, inner sep=0pt,minimum size=3pt] at ({(\ra+.25)*sin \ad},.75*\rb,{(\ra+.25)* cos \ad}) {};
\node[circle,draw=black, fill=coc, inner sep=0pt,minimum size=3pt] at ({(\ra+0)*sin \ad},1*\rb,{(\ra+0)* cos \ad}) {};

\node[circle,draw=black, fill=coc, inner sep=0pt,minimum size=3pt] at ({(\ra+1)*sin \ae},0*\rb,{(\ra+1)* cos \ae}) {};
\node[circle,draw=black, fill=coc, inner sep=0pt,minimum size=3pt] at ({(\ra+.75)*sin \ae},.25*\rb,{(\ra+.75)* cos \ae}) {};
\node[circle,draw=black, fill=coc, inner sep=0pt,minimum size=3pt] at ({(\ra+.5)*sin \ae},.5*\rb,{(\ra+.5)* cos \ae}) {};
\node[circle,draw=black, fill=coc, inner sep=0pt,minimum size=3pt] at ({(\ra+.25)*sin \ae},.75*\rb,{(\ra+.25)* cos \ae}) {};
\node[circle,draw=black, fill=coc, inner sep=0pt,minimum size=3pt] at ({(\ra+0)*sin \ae},1*\rb,{(\ra+0)* cos \ae}) {};

\draw [ domain=\aa:\ac ,smooth,variable =\x, gray] plot ({(\ra+.5)*sin \x},.5*\rb,{(\ra+.5)* cos \x}); 
\draw [ domain=\aa:\ac,smooth,variable =\x, gray] plot ({(\ra+.75)*sin \x},.25*\rb,{(\ra+.75)* cos \x}); 
\draw [ domain=\aa:\ac ,smooth,variable =\x, gray] plot ({(\ra+.25)*sin \x},.75*\rb,{(\ra+.25)* cos \x}); 
\draw [ domain=\aa:\ac,smooth,variable =\x, gray] plot ({(\ra+1)*sin \x},0*\rb,{(\ra+1)* cos \x}); 
\draw [ domain=\aa:\ac ,smooth,variable =\x, gray] plot ({\ra*sin \x},1*\rb,{\ra* cos \x});

\end{tikzpicture}
         \caption{}
         \label{fig:shadow_flag_mult_2}
     \end{subfigure}
	\hspace{1in}
        \caption{Graphical representations of \cref{thm:THH_n_additivity}}
        \label{fig:pic_shadows_add_mult}
\end{figure}

\begin{proof}
	By \cref{thh_n_derived,unwinding,equivs_measured_on_geometric_fp}, it suffices to check that the corresponding non-equivariant map
	\[
	\bigvee_{\substack{i_1,\dotsc, i_n \\ 1 \leq i_j \leq k_j}} \THH(\spcat{\cC};\tensor[_L]{\spcat{\cD}}{_R}^{\odot r})
	\stackrel\sim\arr \THH(w_{k_0}S^{(n)}_{k_1, \dotsc, k_n} \spcat{\cC};\tensor[_L]{(w_{k_0}S^{(n)}_{k_1, \dotsc, k_n} \spcat{\cD})}{_R}^{\odot r}).
	\]
	is an equivalence. This is proven by the same induction as in
        \cref{thm:THH_additivity}.  The only difference is that in the inductive
        step (proof of \cref{prop:small_add_2}) we apply the map $f$ a total of $r$ times instead of just once.
\end{proof}

\subsection{Definition of the equivariant Dennis trace}

We next define the equivariant refinement of the Dennis trace.

\begin{rmk}
	If $\spcat\cC$ is a spectral Waldhausen category, then $(\spcat{\cC}^{\sma r},\uncat\cC^{\times r})$ fails to be a spectral Waldhausen category because it does not satisfy the pushout axiom. Therefore the equivariant Dennis trace is \emph{not} the twisted Dennis trace applied to $\spcat\cC^{\sma r}$. Instead, the smash powers have to occur on the outside of the $S_\bullet$ construction.
\end{rmk}

\begin{defn}\label{equivariant_dennis_trace}
Let $\tw{\cC}{L}{R}{\cD}$ be a twisting of a spectral Waldhausen category $\spcat{\cC}$. For each $r \geq 1$ the inclusion of the $0$-simplices into the cyclic bar construction defines a $C_r$-equivariant map
\begin{equation}\label{include_endos_equivariant}
\Sigma^{\infty} \ob \End^{(r)}\left(\tensor[_{L}]{ \left(\cC/\cD\right)}{_{R}}\right)
\arr \bigvee_{c_1,\ldots,c_r \in  \ob \spcat{\cC}} \bigwedge_{i=1}^r \spcat{\cD}(L (c_i), R(c_{i+1}))\arr \THH^{(r)}(\spcat\cC;\tensor[_L]{\spcat\cD}{_R}).
\end{equation}
\begin{figure}
	\begin{tikzpicture}
\def\shrink{3}
\def\step{30}
\def\rad{1.75}

\node[circle,draw=black, fill=cob, inner sep=0pt,minimum size=3pt] at ({-5+\rad*cos 0},{\rad*sin 0}) {};
\node[circle,draw=black, fill=cob, inner sep=0pt,minimum size=3pt] at ({-5+\rad*cos \step},{\rad*sin \step}) {};
\node[circle,draw=black, fill=cob, inner sep=0pt,minimum size=3pt] at ({-5+\rad*cos (2*\step)},{\rad*sin (2*\step)}) {};
\node[circle,draw=black, fill=cob, inner sep=0pt,minimum size=3pt] at ({-5+\rad*cos (3*\step)},{\rad*sin (3*\step)}) {};
\node[circle,draw=black, fill=cob, inner sep=0pt,minimum size=3pt] at ({-5+\rad*cos (4*\step)},{\rad*sin (4*\step)}) {};
\node[circle,draw=black, fill=cob, inner sep=0pt,minimum size=3pt] at ({-5+\rad*cos (5*\step)},{\rad*sin (5*\step)}) {};
\node[circle,draw=black, fill=cob, inner sep=0pt,minimum size=3pt] at ({-5+\rad*cos (6*\step)},{\rad*sin (6*\step)}) {};
\node[circle,draw=black, fill=cob, inner sep=0pt,minimum size=3pt] at ({-5+\rad*cos (7*\step)},{\rad*sin (7*\step)}) {};
\node[circle,draw=black, fill=cob, inner sep=0pt,minimum size=3pt] at ({-5+\rad*cos (8*\step)},{\rad*sin (8*\step)}) {};
\node[circle,draw=black, fill=cob, inner sep=0pt,minimum size=3pt] at ({-5+\rad*cos (9*\step)},{\rad*sin (9*\step)}) {};
\draw [dotted, domain=(0+\shrink):(\step-\shrink) ,smooth,variable =\x, gray] plot ({-5+\rad*cos \x},{\rad*sin \x}); 
\draw [->, domain=(\step+\shrink):(2*\step-\shrink) ,smooth,variable =\x, gray] plot ({-5+\rad*cos \x},{\rad*sin \x}) ; 
\draw [dotted, domain=(2*\step+\shrink):(3*\step-\shrink) ,smooth,variable =\x, gray] plot ({-5+\rad*cos \x},{\rad*sin \x}); 
\draw [dotted, domain=(3*\step+\shrink):(4*\step-\shrink) ,smooth,variable =\x, gray] plot ({-5+\rad*cos \x},{\rad*sin \x}); 
\draw [->, domain=(4*\step+\shrink):(5*\step-\shrink) ,smooth,variable =\x, gray] plot ({-5+\rad*cos \x},{\rad*sin \x}) ; 
\draw [dotted, domain=(5*\step+\shrink):(6*\step-\shrink) ,smooth,variable =\x, gray] plot ({-5+\rad*cos \x},{\rad*sin \x}); 
\draw [dotted, domain=(6*\step+\shrink):(7*\step-\shrink) ,smooth,variable =\x, gray] plot ({-5+\rad*cos \x},{\rad*sin \x}); 
\draw [->, domain=(7*\step+\shrink):(8*\step-\shrink) ,smooth,variable =\x, gray] plot ({-5+\rad*cos \x},{\rad*sin \x}) ; 
\draw [dotted, domain=(8*\step+\shrink):(9*\step-\shrink) ,smooth,variable =\x, gray] plot ({-5+\rad*cos \x},{\rad*sin \x}); 
\draw [dashed, domain=(9*\step+\shrink):(360-\shrink) ,smooth,variable =\x, gray] plot ({-5+\rad*cos \x},{\rad*sin \x}); 

\node[right] at ({-5+\rad*cos 0},{\rad*sin 0}) {$c_r$};
\node[right] at ({-5+\rad*cos \step},{\rad*sin \step}) {$L(c_r)$};
\node[below left] at ({-5+\rad*cos (1.5*\step)},{\rad*sin (1.5*\step)}) {$f_r$};
\node[above ] at ({-5+\rad*cos (2*\step)},{\rad*sin (2*\step)}) {$R(c_1)$};
\node[above] at ({-5+\rad*cos (3*\step)},{\rad*sin (3*\step)}) {$c_1$};
\node[ above left] at ({-5+\rad*cos (4*\step)},{\rad*sin (4*\step)}) {$L(c_1)$};
\node[ below right] at ({-5+\rad*cos (4.5*\step)},{\rad*sin (4.5*\step)}) {$f_1$};
\node[left] at ({-5+\rad*cos (5*\step)},{\rad*sin (5*\step)}) {$R(c_2)$};
\node[left] at ({-5+\rad*cos (6*\step)},{\rad*sin (6*\step)}) {$c_2$};

\node[left] at ({-5+\rad*cos (7*\step)},{\rad*sin (7*\step)}) {$L(c_2)$};
\node[above right] at ({-5+\rad*cos (7.5*\step)},{\rad*sin (7.5*\step)}) {$f_2$};
\node[below left] at ({-5+\rad*cos (8*\step)},{\rad*sin (8*\step)}) {$R(c_3)$};
\node[below] at ({-5+\rad*cos (9*\step)},{\rad*sin (9*\step)}) {$c_3$};

\def\dist{1}

\node[circle,draw=black, fill=cob, inner sep=0pt,minimum size=3pt] at ({\dist+\rad*cos 0},{\rad*sin 0}) {};
\node[circle,draw=black, fill=cob, inner sep=0pt,minimum size=3pt] at ({\dist+\rad*cos \step},{\rad*sin \step}) {};
\node[circle,draw=black, fill=cob, inner sep=0pt,minimum size=3pt] at ({\dist+\rad*cos (2*\step)},{\rad*sin (2*\step)}) {};
\node[circle,draw=black, fill=cob, inner sep=0pt,minimum size=3pt] at ({\dist+\rad*cos (4*\step)},{\rad*sin (4*\step)}) {};
\node[circle,draw=black, fill=cob, inner sep=0pt,minimum size=3pt] at ({\dist+\rad*cos (5*\step)},{\rad*sin (5*\step)}) {};
\node[circle,draw=black, fill=cob, inner sep=0pt,minimum size=3pt] at ({\dist+\rad*cos (7*\step)},{\rad*sin (7*\step)}) {};
\node[circle,draw=black, fill=cob, inner sep=0pt,minimum size=3pt] at ({\dist+\rad*cos (8*\step)},{\rad*sin (8*\step)}) {};
\draw [dotted, domain=(0+\shrink):(\step-\shrink) ,smooth,variable =\x, gray] plot ({\dist+\rad*cos \x},{\rad*sin \x}); 
\draw [->, domain=(\step+\shrink):(2*\step-\shrink) ,smooth,variable =\x, gray] plot ({\dist+\rad*cos \x},{\rad*sin \x}) ; 
\draw [->, domain=(2*\step+\shrink):(4*\step-\shrink) ,smooth,variable =\x, gray] plot ({\dist+\rad*cos \x},{\rad*sin \x}); 
\draw [->, domain=(4*\step+\shrink):(5*\step-\shrink) ,smooth,variable =\x, gray] plot ({\dist+\rad*cos \x},{\rad*sin \x}) ; 
\draw [->, domain=(5*\step+\shrink):(7*\step-\shrink) ,smooth,variable =\x, gray] plot ({\dist+\rad*cos \x},{\rad*sin \x}); 
\draw [->, domain=(7*\step+\shrink):(8*\step-\shrink) ,smooth,variable =\x, gray] plot ({\dist+\rad*cos \x},{\rad*sin \x}) ; 
\draw [dotted, domain=(8*\step+\shrink):(9*\step) ,smooth,variable =\x, gray] plot ({\dist+\rad*cos \x},{\rad*sin \x}); 
\draw [dashed, domain=(9*\step):(360-\shrink) ,smooth,variable =\x, gray] plot ({\dist+\rad*cos \x},{\rad*sin \x}); 

\node[right] at ({\dist+\rad*cos 0},{\rad*sin 0}) {$\mathcal{C}$};
\node[above right] at ({\dist+\rad*cos (1.5*\step)},{\rad*sin (1.5*\step)}) {$_L\mathcal{D}_R$};
\node[above] at ({\dist+\rad*cos (3*\step)},{\rad*sin (3*\step)}) {$\mathcal{C}$};
\node[above left] at ({\dist+\rad*cos (4.5*\step)},{\rad*sin (4.5*\step)}) {$_L\mathcal{D}_R$};
\node[left] at ({\dist+\rad*cos (6*\step)},{\rad*sin (6*\step)}) {$\mathcal{C}$};

\node[below left] at ({\dist+\rad*cos (7.5*\step)},{\rad*sin (7.5*\step)}) {$_L\mathcal{D}_R$};
\node[below] at ({\dist+\rad*cos (9*\step)},{\rad*sin (9*\step)}) {$\mathcal{C}$};

\draw[->, decorate, decoration={
    zigzag,
    segment length=4,
    amplitude=.9,post=lineto,
    post length=2pt
}] (-2.5,0)--(-1.25,0);
\end{tikzpicture}
	\caption{Graphical representation of \eqref{include_endos_equivariant}}\label{fig:multi_trace_inclusion}
\end{figure}
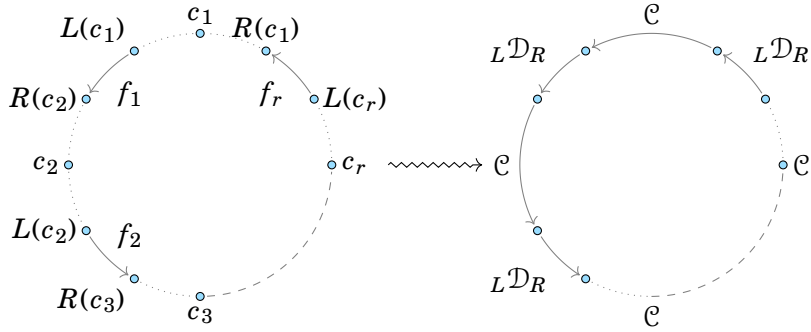
See \cref{fig:multi_trace_inclusion}. This is natural with respect to morphisms of twisted spectral Waldhausen categories, so we apply it to the multisimplicial spectral Waldhausen category $w_{\bullet} S_{\bullet}^{(n)} \spcat{\cC}$ twisted by $w_{\bullet} S_{\bullet}^{(n)} \spcat{\cD}$, and get a zig-zag of multisimplicial orthogonal spectra
\begin{align*}\label{eq:levelwise_trace_zigzag_equivariant}
\Sigma^{\infty} \ob \End^{(r)}\left(\tensor[_L]{(w_{\bullet}S_{\bullet}^{(n)}{{\cC}}/ w_\bullet S_\bullet^{(n)} {\cD})}{_R}\right)
& \arr \THH^{(r)}\left( w_{\bullet}S_{\bullet}^{(n)} \spcat{\cC} ; \tensor[_L]{(w_\bullet S_\bullet^{(n)} {\spcat\cD})}{_R}\right) \\
& \overset{\simeq}{\longleftarrow} (S^1_\bullet)^{\sma n} \sma \THH^{(r)}\left(\spcat{\cC} ; \tensor[_L]{\spcat{\cD}}{_R} \right),
\end{align*}
the second map coming from \cref{thm:THH_n_additivity}. The same argument as in \cref{zigzag_respects_simplicial_str} shows that this is a zig-zag of $\Sigma_\Delta$-diagrams of simplicial orthogonal $C_r$-spectra, hence we can take their geometric realization and get a zig-zag of $C_r$-equivariant bispectra. We describe these in detail in \cite[\swcref{Appendix A}{\cref{swc-sec:model_structures}}]{clmpz-specWaldCat}. Applying left-derived prolongation
	(\cite[\swcref{Proposition A.7}{\cref{swc-quillen-adjoints}}]
	{clmpz-specWaldCat}),
and using the canonical identifications of object sets
\[
\ob w_{\bullet} S_{\bullet}^{(n)} \End^{(r)}\left(\tensor[_L]{{\cC}/{\cD}}{_R}\right) = \ob \End^{(r)} \left(\tensor[_L]{(w_{\bullet}S_{\bullet}^{(n)}\spcat{{\cC}}/w_\bullet S_\bullet^{(n)} {\cD})}{_R}\right),
\]
we get a zig-zag of orthogonal $C_r$-spectra
\begin{equation}\label{eq:equivariant_levelwise_trace_zigzag}
\bbP K\End^{(r)}\left(\tw{\cC}{L}{R}{\cD}\right)
\arr
\left\vert\THH^{(r)} \left( w_{\bullet} S_{\bullet}^\ast \spcat{\cC} ; \tensor[_L]{(w_\bullet S_\bullet^\ast \spcat{\cD})}{_R}\right)\right\vert
\overset{\simeq}{\longleftarrow}
\THH^{(r)} \left(\spcat\cC;\tensor[_L]{\spcat\cD}{_R}\right).
\end{equation}
	The {\bf $r$-fold Dennis trace map}
	\[
          \trc^{(r)} \colon K\End^{(r)}\left(\tw{\cC}{L}{R}{\cD}\right) \arr \THH^{(r)} \left(\spcat\cC;\tensor[_L]{\spcat\cD}{_R}\right)
	\]
	is the map represented by this zig-zag in the stable homotopy category of orthogonal $C_r$-spectra.
\end{defn}

The case $r = 1$ recovers the twisted Dennis trace
\[ \trc^{(1)} = \trc \colon \Sigma^{\infty} K(\End({}_L\cC/\cD_R)) \arr \THH(\spcat\cC;\tensor[_L]{\spcat\cD}{_R}) \]
from \cref{twisted_dennis_trace}.

The following is an extension of \cref{lem:pi_0_dennis_trace}:

\begin{lem}\label{lem:trace_is_inclusion_objects_r_fold}
	Let $(f_i \colon L(a_i) \arr R(a_{i + 1}))$ be an object of the $r$-fold twisted endomorphism category $\End^{(r)}\left(\tw{\cC}{L}{R}{\cD}\right)$, where we take indices modulo $r$.  The image of the class $[f_1, \dotsc, f_r] \in K_0\End^{(r)}\left(\tw{\cC}{L}{R}{\cD}\right)$ under the $r$-fold Dennis trace map
	\[
	\trc^{(r)} \colon \pi_0K\End^{(r)}\left(\tw{\cC}{L}{R}{\cD}\right) \arr \pi_0\THH^{(r)}(\tw{\cC}{L}{R}{\cD})
	\]
	is the composite
	\[
	\bbS \xarr{\bigwedge [f_i]} \bigwedge_{i = 1}^{r} \spcat{\cD}(L(a_i), R(a_{i + 1)})) \arr \bigvee_{(c_1, \dotsc, c_r) \in \cC^{r}} \bigwedge_{i = 1}^{r} \spcat{\cD}(L(c_i), R(c_{i + 1})) \xarr{\text{$0$-skeleton}} \THH^{(r)}(\tw{\cC}{L}{R}{\cD})
	\]
	which includes the homotopy classes of the $f_i$ as $0$-simplices in the cyclic bar construction.
\end{lem}

We can use this result to identify the image under the trace of certain classes in the
$K$-theory of endomorphisms.  Let $A$ be a ring spectrum and let $M$ be an $(A, A)$-bimodule.  Given a collection $P_1, \dotsc, P_r$ of perfect left $A$-modules and $A$-module maps
\[ f_i \colon P_i \arr M \sma_A P_{i + 1}, \]
where the indices are taken modulo $r$, we let $f_{r} \circ \dotsm \circ f_{1}$ denote the composite map
  \[
  P_1 \oarr{f_1} M \sma_A P_{2} \xarr{\id \sma f_{2}} M \sma_A M \sma_A P_{3} \xarr{\id \sma f_{3}} \ \dotsm \ \xarr{\id \sma f_r} M^{\sma_{A} r} \sma_A P_1.
  \]
The bicategorical trace of $f_{r} \circ \dotsm \circ f_{1}$ is a map of spectra
  \[
  \tr(f_{r} \circ \dotsm \circ f_{1}) \colon \bbS \cong \THH(\bbS) \arr \THH(A; M^{\sma_{A} r}).
  \]
On the other hand, the collection $(f_1, \dotsc, f_r)$ is an object of the $r$-fold twisted endomorphism category $\End^{(r)}(\cP/\cM_{- \sma_{A} M})$, and thus determines a class $[f_1, \dotsc, f_r]$ in
\[K_0(\End^{(r)}(\cP/\cM_{M})).\]

\begin{prop}\label{prop:main_identify_pi_0_TR}
	The image of the class $[f_1, \dotsc, f_r] \in K_0(\End^{(r)}(\cP/\cM_{M}))$ under the composite map
	\begin{equation}\label{eq:composite_main_identify_n_distinct}
		\begin{tikzcd}
			K(\End^{(r)}(\cP/\cM_{M})) \ar{r}{\trc^{(r)}} &\THH^{(r)} (\cP; \cM_{M}) &  \ar{l}[swap]{\simeq} \THH^{(r)}(A; M) \ar{r}{\simeq} & \THH(A; M^{\sma_{A}r})
		\end{tikzcd}
	\end{equation}
	is the homotopy class of the trace of the composite
  $[\tr(f_{r} \circ \dotsm \circ f_{1}) ] \in \pi_0 \THH(A; M^{\sma_{A}r})$.
\end{prop}

\begin{proof}
  The homotopy class of $f_{r} \circ \dotsm \circ f_{1}$ is encoded by the map of spectra
  \[
  \bbS \cong \bbS^{\sma r} \xarr{\bigwedge_{i = 1}^{r} [f_i]} \bigwedge_{i = 1}^r \spcat{\cM}(P_{i}, M \sma_A P_{i+1}) \oarr{\circ} \spcat{\cM}(P_{1}, M^{\sma_{A} r} \sma_A P_{1}).
  \]
  This fits into the following commutative diagram of inclusions and composition maps.
\begin{small}
  \begin{equation}\label{eq:pi_0_TR_diag1}
	{\begin{tikzcd}
  \bbS \ar{r}{\bigwedge_{i = 1}^{r} [f_i]} \ar{ddr}[swap]{[f_1 \circ \dotsm \circ f_r]} & \displaystyle\bigwedge_{i = 1}^r \spcat{\cM}(P_{i}, M \sma_A P_{i+1}) \ar{dd}{\circ} \ar{r}{}
  & \displaystyle \bigvee_{(Q_1, \dotsc, Q_r) \in \ob \cP^r} \spcat{\cM}(Q_1, M \sma_A Q_2) \sma \dotsm \sma \spcat{\cM}(Q_r, M \sma_A Q_1)  \ar[equal]{d}
  \\
  & & \displaystyle\bigvee_{Q_1} \bigvee_{(Q_2, \dotsc, Q_r)} \spcat{\cM}(Q_1, M \sma_A Q_2) \sma \dotsm \sma \spcat{\cM}(Q_r, M \sma_A Q_1) \ar{d}{\circ}
  \\
  & \spcat{\cM}(P_1, M^{\sma_{A} r} \sma_A P_1) \ar{r}{}
  & \displaystyle \bigvee_{Q \in \ob \cP} \spcat{\cM}(Q, M^{\sma_{A} r} \sma_A Q)
  \end{tikzcd}}
  \end{equation}
\end{small}
\noindent The right-hand column of the previous diagram is the $0$-skeleton of the left-hand column of the next diagram:
\begin{equation}\label{eq:pi_0_TR_diag2}
  \begin{tikzcd}
  \THH^{(r)}(\spcat{\cP}; \spcat{\cM}_{M}) \ar{d}{\cong} & \ar{l}[swap]{\simeq} \THH^{(r)}(A; M) \ar{d}{\cong} \\
  \THH(\spcat{\cP}; \spcat{\cM}_{M}^{\odot r}) \ar{d}{\circ} & \ar{l}[swap]{\simeq} \THH(A; M^{\odot r}) \ar{d}{\simeq}
  \\
  \THH(\spcat{\cP}; \spcat{\cM}_{M^{\sma_{A} r}}) & \ar{l}[swap]{\simeq} \THH(A; M^{\sma_{A} r}).
  \end{tikzcd}
\end{equation}
Here the horizontal arrows come from the Morita equivalence $A \to \cP$ and the map of bimodules $M \to \cM_M$ from \cref{ex:bimodule_smash_functor}. The upper vertical maps on the left and right are the unwinding equivalence from \cref{unwinding} and the top region commutes by naturality of this equivalence. We use the notation $M \odot -$ to denote the bar construction, to distinguish it from the strict smash product $M \sma_A -$. The lower-right vertical map collapses this bar construction, and the bottom region commutes using (\cite[\swcref{Lemma 3.16}{\cref{swc-composing_bimodules}}]
{clmpz-specWaldCat}).

The lower-left vertical map of \eqref{eq:pi_0_TR_diag2} arises by applying $\THH$ to the morphism of $(\spcat{\cP}, \spcat{\cP})$-bimodules $\spcat{\cM}_{M}^{\odot r} \arr \spcat{\cM}_{M^{\sma_{A} r}}$ defined by iterating the composition operation
\begin{align*}
	& \spcat{\cM}(P, M^{\sma_{A} i} \sma_A -) \odot \spcat{\cM}(-, M \sma_A P') \\
	\oarr{\id \sma -} & \spcat{\cM}(P, M^{\sma_{A} i} \sma_A -) \odot \spcat{\cM}(M^{\sma_{A} i} \sma_A -, M^{\sma_{A} (i+1)} \sma_A P') \\
	\oarr{\circ} & \spcat{\cM}(P, M^{\sma_{A} (i+1)} \sma_{A} P')
\end{align*}
for $i = 1, \dotsc, r - 1$. The commutativity of the bottom region of \eqref{eq:pi_0_TR_diag2} is the fact that when we take $P = P' = A$, the resulting composite map of $(A,A)$-bimodule spectra
\[
  M^{\sma_A r} \arr \spcat{\cM}(A, M \sma_A A)^{\odot r} \oarr{\circ} \spcat{\cM}(A, M^{\sma_{A} r} \sma_{A} A)
\]
is adjoint to the identity of $M^{\sma_A r}$.

  Paste the diagrams \eqref{eq:pi_0_TR_diag1} and \eqref{eq:pi_0_TR_diag2} together by including the $0$-skeleta into the THH terms.  The composite along the top and right of the resulting diagram agrees with the composite in the statement of the proposition, by Lemma \ref{lem:trace_is_inclusion_objects_r_fold}.  The bottom composite is the inclusion of the map $f_{r} \circ \dotsm \circ f_{1}$ into the cyclic bar construction for $\THH(\spcat{\cP}; \spcat{\cM}_{M^{\sma_{A} r}})$, followed by the Morita equivalence to $\THH(A; M^{\sma_{A} r})$.  By \cite[7.4]{cp},
  this is the homotopy class of the bicategorical trace $[\tr(f_{r} \circ \dotsm \circ f_{1}) ]$,
  and thus the proof is complete.
\end{proof}

\section{The trace to topological restriction homology}\label{sec:trace_to_TR}

Now that we have constructed an equivariant refinement of the Dennis trace, we
distill out of it a trace from the $K$-theory of endomorphisms to topological
restriction homology which we call the $\TR$-trace.  We then define an analog of
the ghost coordinates of Witt vectors for $\TR$: the ghost maps $g_{n} \colon
\TR X_{\bullet} \arr X_{n}$.  Finally, we prove that applying the $n$-th ghost map to the $\TR$-trace encodes the trace $\tr(f^{\circ n})$ of the $n$-fold iterate of a self-map (\cref{thm:main_identify_pi_0_TR}).

\subsection{Restriction Systems}\label{subsec:restriction_systems}
We begin by formalizing the sense in which the $r$-fold Dennis traces $\trc^{(r)}$ fit
together as $r$ varies. The situation is a bit subtle because on the $K$-theory
side they are related by the categorical fixed points, and on the $\THH$ side
they are related by the geometric fixed points.

\begin{defn} \label{genuine_restriction_system}
For each $m, n \geq 1$, let $\rsf_{n}$ be a functor from $C_{mn}$-spectra to $C_{m}$-spectra such that for all $r,s \geq 1$ we have natural transformations
  \[ \rsf_{rs} \to \rsf_s \circ \rsf_r. \]
An \textbf{$\rsf_*$-pre-restriction system} consists of the following data:
  \begin{itemize}
  \item A sequence of spectra $\{X_n\}_{n=1}^\infty$ together with an action of
    $C_n$ on $X_n$ for all $n$.
  \item A $C_s$-equivariant map $c_r\colon \rsf_r(X_{rs}) \to X_s$ for all $r,s\geq 1$
  making the square
  \[ \xymatrix @C=3em{
      \rsf_{rs}(X_{rst}) \ar[d] \ar[r]^-{c_{rs}} & X_t \\
      \rsf_s(\rsf_r(X_{rst})) \ar[r]_-{\rsf_s(c_r)} & \rsf_s(X_{st}) \ar[u]_-{c_s}
    } \]
  commute for all $r,s,t\geq 1$.
  \end{itemize}

  When $\rsf_n$ is the categorical fixed points functor $(-)^{C_n}$ and each $c_r$ is an isomorphism, this is called a
  \textbf{naive restriction system}.  When $\rsf_n$ is the geometric fixed points functor $\Phi^{C_n}$,
  and each $c_r$ induces a weak
  equivalence out of the derived geometric fixed points, this is a \textbf{genuine restriction system}.

  A morphism of restriction systems consists of equivariant maps
  $X_n \arr Y_n$ commuting with the maps $c_r$.
\end{defn}

\begin{rmk}
  Our definition of restriction systems recalls the structures defining a
  $p$-cyclotomic spectrum; however, instead of working with the pro-group
  $\bbZ_p$ we are instead working with the pro-group $\widehat\bbZ$.
  Restriction systems should therefore be closely related to pro-spectra; see,
  for example, \cite{fausk}.
\end{rmk}

\begin{example}\label{ex:naive_restriction_systems}
  The definition of naive restriction system works equally well for the
  category of spaces, so let us consider that case first.
  Let $X$ be a space equipped with a self-map $f \colon X \to X$.  There is a naive restriction system
  whose $n$th term is the twisted free loop space of the Fuller construction
  $\mc{L}^{\Psi^n(f)} X^n$, consisting of all $n$-tuples of points
  $x_1, \dotsc, x_n$ and paths from $f(x_i)$ to $x_{i+1}$ (indices modulo $n$)
  \cite{mp1,kw2}. The case where all of the paths are constant gives a naive restriction system where the $n$-th level is the $n$-periodic points $\Fix(f^{\circ n})$.
  In the case where $f = \id$, the system assigns to each $n \geq 1$ the free loop space
  $\mc{L} X = \Map(S^1, X)$ with $C_n$ acting by rotating the loops. The
  structure maps are the $n$-power maps $(\mc{L} X)^{C_n} \cong \mc{L} X$.
\item\label{ex:fuller_maps}
\end{example}

\begin{example}
  If $T$ is a cyclotomic spectrum in the sense of
  \cite{ABGHLM_TC,blumberg_mandell_cyclotomic}, there is a genuine
  restriction system where the $n$th term is $T$ with the generator of $C_n$
  acting by $e^{2\pi i/n} \in S^1$.
\end{example}

\begin{example}\label{ex:gen_restriction_systems}\label{suspend_naive_system_of_spaces}
  We give a couple of examples of ways to go between different kinds of
  restriction system structures. If $X_\bullet$ is a naive restriction system of based spaces, taking
  suspension spectra $\Sigma^\infty X_\bullet$ gives a genuine restriction
  system. This uses the isomorphism
  $\Phi^{C_r} \Sigma^\infty X \cong \Sigma^\infty (X^{C_r})$, and the fact that
  in this case the geometric fixed points agree with the left-derived geometric
  fixed points.

  Now suppose instead that $X_\bullet$ is a genuine restriction system of
  fibrant orthogonal spectra.  For any $G$-spectrum $Y$ there is a canonical map
  $\kappa\colon Y^G \to \Phi^G Y$ from the categorical fixed points to the geometric
  fixed points.  There are therefore maps
  \[\gamma_r\colon X^{C_{rs}}_{rst} \stackrel{\kappa^{C_s}}\arr (\Phi^{C_r}X_{rst})^{C_s}
    \stackrel{c_r^{C_s}}\arr X_{st}^{C_s}.\]
  These maps make $X_\bullet$ into a $(-)^{C_n}$-pre-restriction system.
\end{example}

The key examples for the purposes of the current discussion are the following:

\begin{example}\label{ex:k_restriction_system}
There is a naive restriction system whose $n$-th level is the $K$-theory $K(\End^{(n)}(\tensor[_{L}]{\cC/\cD}{_{R}}))$ of the $n$-fold endomorphism category from \cref{def:twisted_endo}. The structure maps $c_r$ identify the fixed points of such a category with the same category for a smaller value of $n$:
\[ K\left(\End^{(rs)}\left(\tw{\cC}{L}{R}{\cD}\right)\right)^{C_r} \cong K\left(\End^{(rs)}\left(\tw{\cC}{L}{R}{\cD}\right)^{C_r}\right) \cong K\End^{(s)}\left(\tw{\cC}{L}{R}{\cD}\right). \]
In other words, an $rs$-tuple of objects and morphisms that is strictly preserved by the $C_r$-action must be an $s$-tuple of objects and morphisms that are repeated $r$ times.
\end{example}

\begin{example} \label{thh_system}
  When $\cC$ is a pointwise cofibrant spectral category and $\cX$ is a
  pointwise cofibrant bimodule, the isomorphisms from \cref{thh_n_derived}
  make $\THH^{\bullet}(\cC;\cX)$ into a genuine restriction system. The proof is essentially that of \cite[4.6]{malkiewich_thh_dx}, but simpler because there are no subdivisions.
\end{example}

The $r$-fold Dennis trace map from \cref{equivariant_dennis_trace} assembles into a map of restriction systems. Since the Dennis trace is a zig-zag of bispectra, the most natural statement to make is that it defines a zig-zag of restriction systems of bispectra. The details of this notion are in
	\cref{sec:model_str_restriction_systems}---the important thing to know is that the geometric fixed points of a bispectrum are taken at each symmetric spectrum level separately, so that in the symmetric spectrum direction they behave like categorical fixed points, and in the orthogonal direction they behave like geometric fixed points. It is this convention that rectifies the apparent disparity between $K\End^{(r)}$ and $\THH^{(r)}$.

We can suspend the naive restriction system of \cref{ex:k_restriction_system} and get a genuine restriction system of bispectra $\Sigma^{\infty} K\left(\End^{\bullet}\left(\tw{\cC}{L}{R}{\cD}\right)\right)$, as in \cref{suspend_naive_system_of_spaces}. Similarly, we can suspend the genuine restriction system of \cref{thh_system} and get a genuine restriction system of bispectra $\Sigma^{\infty} \THH^{\bullet}\left(\spcat{\cC}; \tensor[_L]{\spcat\cD}{_R} \right)$. See \cref{restriction_system_pushup_1,restriction_system_pushup_2} for additional discussion.

\begin{thm}\label{trace_respects_restriction}
	The $r$-fold Dennis trace for varying $r \geq 1$
	\[
	\trc^{\bullet} \colon \Sigma^{\infty} K\left(\End^{\bullet}\left(\tw{\cC}{L}{R}{\cD}\right)\right) \arr \Sigma^{\infty} \THH^{\bullet}\left(\spcat{\cC}; \tensor[_L]{\spcat\cD}{_R} \right)
	\]
	together define a morphism of genuine restriction systems of bispectra.
\end{thm}

By \cref{prolong_restriction_systems,restriction_system_model_structure},
we can therefore make these restriction systems cofibrant and then prolong them back to orthogonal spectra, giving a map in the homotopy category of genuine restriction systems of orthogonal spectra
	\[
	\bbP\Sigma^{\infty} K\left(\End^{\bullet}\left(\tw{\cC}{L}{R}{\cD}\right)\right) \arr \THH^{\bullet}\left(\spcat{\cC}; \tensor[_L]{\spcat\cD}{_R} \right).
	\]

\begin{proof}
	Apply \cref{thh_system} to the middle term
	\[ \THH^{(r)}\left( w_{\bullet}S_{\bullet}^{(n)} \spcat{\cC} ; \tensor[_L]{(w_\bullet S_\bullet^{(n)} {\spcat\cD})}{_R}\right) \]
	of the zig-zag that defines the $r$-fold Dennis trace \eqref{eq:equivariant_levelwise_trace_zigzag}.
	By the naturality of the construction, this gives a $\Sigma_\Delta$-diagram of restriction systems of simplicial orthogonal $C_r$-spectra. Concretely, such an object consists of orthogonal spectra $X_{k_0,k_1,\ldots,k_n,r}$ with the structure of a $\Sigma_\Delta$-diagram in the first $(n+1)$ indices, and of a genuine restriction system in the last index, that commute with each other. By naturality, the maps of the additivity theorem for $\THH^{(r)}$ (\cref{thm:THH_n_additivity}) give a map of $\Sigma_\Delta$-diagrams of restriction systems of simplicial orthogonal $C_r$-spectra.

	We next check that the inclusion of $K$-theory \eqref{include_endos_equivariant} also gives a map of $\Sigma_\Delta$-diagrams of restriction systems of simplicial orthogonal $C_r$-spectra. We already know it commutes with most of this structure by the argument before \cref{equivariant_dennis_trace}; the only new thing to check is agreement with the maps of the restriction system. By \cref{norm_on_suspension_spectra}, we can rewrite the maps of the restriction system on the left using the HHR norm diagonal, and then the inclusion of endomorphisms commutes with the restriction system structure, simply because the norm diagonal is natural:
	\[ \xymatrix{
		\Phi^{C_r}\Sigma^{\infty} \ob \End^{(rs)}\left(\tensor[_{L}]{ \left(\cC/\cD\right)}{_{R}}\right)
		\ar[r] &
		\Phi^{C_r}\displaystyle\bigvee_{a_1,\ldots,a_{rs} \in  \ob \spcat{\cC}} \displaystyle\bigwedge_{i=1}^{rs} \spcat{\cD}(L (a_i), R(a_{i+1})) \\
		\Sigma^{\infty} \ob \End^{(s)}\left(\tensor[_{L}]{ \left(\cC/\cD\right)}{_{R}}\right) \ar[u]_-{D_r}
		\ar[r] &
		\displaystyle\bigvee_{b_1,\ldots,b_{s} \in  \ob \spcat{\cC}} \displaystyle\bigwedge_{i=1}^{s} \spcat{\cD}(L (b_i), R(b_{i+1})).  \ar[u]_-{D_r}
	} \]

	Each of the desired two maps is now a map of systems of spectra $X_{k_0,k_1,\ldots,k_n,r}$ with both a $\Sigma_\Delta$-action and with restriction maps. We re-interpret these as $\Sigma_\Delta$-diagrams of restriction systems and take their realization to get symmetric spectrum objects in pre-restriction systems. We identify these with pre-restriction systems of bispectra by noting that the geometric fixed point functor for bispectra is (uniquely) isomorphic to the functor that takes the geometric fixed points at each symmetric spectrum level separately
	(\cite[\swcref{Definition A.9}{\cref{swc-gfp_bispectra}}]
	{clmpz-specWaldCat}).

	It remains to check that we actually have restriction systems, in other words that the geometric fixed points agree with the left-derived geometric fixed points. We already know this for $K$-theory and $\THH$ but not for the middle term of our zig-zag. To verify this condition, it is enough if each of the $C_r$-bispectra $Y_{\bullet,\bullet}$ is a cofibrant orthogonal $C_r$-spectrum at each symmetric spectrum level $Y_{n,\bullet}$. In other words, we must show that the realization in the $k_0$ through $k_n$ directions, $|X_{\bullet,\bullet,\ldots,\bullet,r}|$, is a cofibrant orthogonal $C_r$-spectrum.

	To accomplish this we use the model structure on restriction systems of orthogonal $C_r$-spectra from
	\cref{restriction_system_model_structure}
and then take cofibrant replacement in $\Sigma_\Delta$-diagrams of such objects by
	\cite[\swcref{Theorem 6.4}{\cref{swc-properness-thm}}]
	{clmpz-specWaldCat}.
Then for fixed $n$, each system $X_{k_0,k_1,\ldots,k_n,r}$ is Reedy cofibrant, meaning each latching map is a cofibration of restriction systems. Therefore for each fixed value of $r$ the latching map is a cofibration of orthogonal $C_r$-spectra. Therefore the realization is a cofibrant orthogonal $C_r$-spectrum. Therefore, after cofibrant replacement we get a zig-zag of maps of restriction systems of bispectra.

	Note that after this replacement, the backwards map of the zig-zag is an equivalence of bispectra at level 1 of the restriction system. By \cref{equivs_measured_on_geometric_fp}, it is therefore an equivalence of $C_r$-bispectra at level $r$ for every $r \geq 1$. This finishes the proof.
\end{proof}

\subsection{Topological restriction homology}

\begin{defn}\label{tr_in_general}
  For any $\rsf_*$-pre-restriction system, the spectra $\rsf_n(X_n)$ assemble into a
  diagram indexed by the category $\bbI$ with one object for each $n \geq 1$, one
  morphism $n \to m$ when $m |n $, and no other morphisms (as in \cite[\S
  2.5]{madsen_survey}).
  For any $\rsf_*$-pre-restriction system, the spaces $\rsf_n(X_n)$ assemble into a
  diagram indexed by the category $\bbI$.

  Let $X_\bullet$ be a genuine restriction system.  Define maps
  \begin{equation}\label{eq:tr_in_general} \xymatrix {
      X_{rs}^{C_{rs}} \cong (X_{rs}^{C_r})^{C_s} \ar[r]^-{\kappa^{C_s}} & (\Phi^{C_r} X_{rs})^{C_s} \ar[r]^-{(\gamma_r)^{C_s}} & X_s^{C_s}
    } \end{equation}
  by first composing the canonical map $\kappa$ from categorical fixed points to
  geometric fixed points and the structure map $\gamma_r$.   Then apply
  categorical $C_s$-fixed points.  As above, this makes the spectra
  $X_n^{C_n}$ into a diagram indexed by $\bbI$.
 We define the \textbf{underived topological restriction homology} of $X_\bullet$ by
  \[ \TR^{\un}(X_\bullet) = \underset{\bbI}\lim\, X_n^{C_n}. \]

  If in addition each $X_r$ is
  fibrant as an orthogonal $C_r$-spectrum, we define the \textbf{topological restriction homology} of $X_\bullet$ as the homotopy limit
  \[ \TR(X_\bullet) = \underset{\bbI}\holim\, X_n^{C_n}.\]
  To define $\TR$ for an arbitrary genuine restriction system, we first take a
  fibrant replacement in the model structure from
  \cref{restriction_system_model_structure},
and
  then take the homotopy limit as above.  There is a canonical map
  $\TR^{\un} \to \TR$ that takes fibrant replacement and
  passes to the homotopy limit.
\end{defn}

\begin{defn}\label{tr_of_category}
  For any pointwise cofibrant spectral category $\cC$ and pointwise cofibrant
  $(\cC,\cC)$ bimodule $\cX$, we write $\TR(\cC;\cX)$ for the $\TR$ of the
  restriction system $\THH^{\bullet}(\cC; \cX)$ from \cref{thh_system}.
\end{defn}

When $\cC = A$ is a ring spectrum and $\cX = A$, this is the classical
definition of topological restriction homology. On the other hand, if we take a
general $(A,A)$-bimodule spectrum $M$,  this is a spectral version of
Lindenstrauss--McCarthy's $W(A;M)$ \cite{LM12}.

Before defining the $\TR$-trace we need one more technical result.

\begin{lem}[\cref{restriction_system_pushup_1}]\label{underived_tr_of_suspension}
  Let $X_\bullet$ be a naive restriction system of symmetric spectra.  Then
  $\bbP\Sigma^\infty X_\bullet$ is a genuine restriction system of orthogonal spectra and there is an isomorphism
  \[ \TR^{\un}(\mathbb P\Sigma^\infty X_\bullet) \cong \bbP\Sigma^\infty X_1 \cong \bbP X_1. \]
\end{lem}

We are now ready to define the $\TR$-trace.

\begin{defn} \label{tr_trace}
  The \textbf{$\TR$-trace} is the map $\trc\colon K(\End(\lMr{L}{\cC/\cD}{R})) \to
  \TR(\cC; \lMr{L}{\cD}{R})$ defined as the composition
  \begin{align}
    K(\End(\lMr{L}{\cC/\cD}{R})) &\cong \TR^{\un}(\bbP \Sigma^\infty
    K(\End^\bullet(\lMr{L}{\cC/\cD}{R}))) \arr \TR(\bbP \Sigma^\infty
                                   K(\End^\bullet(\lMr{L}{\cC/\cD}{R}))) \nonumber\\
    &\arr \TR(\THH^\bullet(\cC;\lMr{L}{\cD}{R})) = \TR(\cC;\lMr{L}{\cD}{R}). \label{map_of_trs}
  \end{align}
\end{defn}

\begin{example}
	Taking $\spcat\cC = \tensor[^A]{\Perf}{}$ for a ring spectrum $A$, and twisting by an $(A,A)$-bimodule $M$, the $\TR$-trace gives a map
	\begin{equation}\label{eq:TR_trace_ring_example}
	\trc \colon K(A;M)
	= K \End(\tw{\tensor[]{\Perf}{}}{}{M}{\tensor[]{\Mod}{}})
	\arr \TR\left(\spcat{\cP}; \spcat\cM_M\right)
	\overset\sim\longleftarrow
	\TR(A;M).
	\end{equation}
	As before, the $\TR$-trace is identically zero on the zero endomorphisms, so induces a map out of cyclic $K$-theory
	\begin{equation}
	\trc \colon K^\cyc(A;M)
	= \widetilde K(A;M)
	\arr \TR(A;M).
	\end{equation}
\end{example}

\begin{rmk}\label{equivariant_comparison}
	We briefly discuss two compatibility statements with the existing literature. One is that when $M = A$ and we restrict to identity morphisms, we recover the usual trace $K(\tensor[^A]{\Perf}{}) \to \TR(\tensor[^A]{\Perf}{})$ as in e.g. \cite{bokstedt_hsiang_madsen}. Recall that the more common definition arises by using the fact that the inclusion of identity morphisms into $\THH(\cC)$ lands in the categorical fixed points $\THH(\cC)^{S^1}$, and then constructing a natural map $\THH(\cC)^{S^1} \to \TR(\cC)$ out of $r$-fold subdivision and the restriction maps in equivariant stable homotopy theory. However, the $r$-fold subdivision applied to the inclusion of identity morphisms \eqref{eq:k_end_trace_0} agrees with the equivariant inclusion of identity morphisms in \eqref{include_endos_equivariant}, so our construction produces the same map to $\TR$.

	The other statement is that our trace agrees with the trace to $W(A;M)$ as defined by Lindenstrauss and McCarthy for discrete rings \cite{LM12}. The comparison can be sketched just as in \cref{comparison_to_LM12}, only the equivalences are $C_r$-equivariant and we compare the $r$-fold inclusion of endomorphisms map \eqref{include_endos_equivariant} with the one in \cite[9.1]{LM12}.
\end{rmk}

\subsection{The ghost map}

\begin{defn}\label{defn:ghost_map}
Let $X_\bullet$ be a genuine restriction system. The \textbf{ghost map}
\[
g = (g_n) \colon \TR(X_\bullet) \to \prod_{n \geq 1} X_n
\]
 is defined by the composites
\[
\begin{tikzcd}
 g_n \colon \TR(X_\bullet) \ar{r} &  X_n^{C_n} \ar{r}{F^{n}} &  X_n,
\end{tikzcd}
\]
where $F^{n}$ denotes the inclusion of fixed points.
\end{defn}

The ghost map is defined in the same way for underived $\TR$. It is natural with
respect to maps of restriction systems, and with respect to the inclusion of
underived $\TR$ into $\TR$, and this naturality makes it easy to compute.

\begin{prop}\label{dennis_trace_ghost}
	Let $\spcat{\cC}$ be a spectral Waldhausen category and let ${}_L\cC/\cD_R$ be a twisting of $\spcat{\cC}$.  Then the following diagram commutes
	\[
	\xymatrix{
		K \End\left(\tw{\cC}{L}{R}{\cD}\right)  \ar[d]_-{\Delta_{n}} \ar[r]^{\trc}&\TR\left(\spcat{\cC}; \tensor[_L]{\spcat\cD}{_R}\right)\ar[d]^-{g_n}  \\
		K \End^{(n)}\left(\tw{\cC}{L}{R}{\cD}\right)  \ar[r]^{\trc^{(n)}} & \THH^{(n)}\left(\spcat{\cC}; \tensor[_L]{\spcat\cD}{_R}\right),
	}
      \]
      where $\Delta_n$ is the duplication functor from \cref{def:rfoldTwisted}.
	Setting $n = 1$, we conclude that the $\TR$-trace is a lift of the Dennis trace along the first ghost map $g_1$:
	\[
	\xymatrix{
		K \End\left(\tw{\cC}{L}{R}{\cD}\right) \ar[dr]^-{\trc^{(1)}}\ar[r]^{\trc}&\TR\left(\spcat{\cC}; \tensor[_L]{\spcat\cD}{_R}\right)\ar[d]^-{g_1}  \\
		& \THH\left(\spcat{\cC}; \tensor[_L]{\spcat\cD}{_R}\right).
	}
	\]

\end{prop}
\begin{proof}
	By \cref{underived_tr_of_suspension}, the underived $\TR$ of the genuine restriction system
\[\bbP\Sigma^\infty K \End^{\bullet}\left(\tw{\cC}{L}{R}{\cD}\right)\] is the $K$-theory of endomorphisms $K \End\left(\tw{\cC}{L}{R}{\cD}\right)$. Using the naturality of the ghost map, it suffices to prove that under this identification, the $n$th ghost map
	\[
	g_{n} \colon \TR^{\un}\bbP\Sigma^\infty K \End^{\bullet}\left(\tw{\cC}{L}{R}{\cD}\right)  \arr K \End^{(n)}\left(\tw{\cC}{L}{R}{\cD}\right)
	\]
	agrees with the map induced by $\Delta_{n}$.

	The $n$th ghost map for the restriction system applies the inverse of the structural isomorphisms $\gamma_n$ (see \cref{restriction_system_pushup_1}) to get to the $C_n$-fixed points of the $n$th term, then includes the $C_n$-fixed points into the entire $K$-theory spectrum
	\[ K \End^{(1)}\left(\tw{\cC}{L}{R}{\cD}\right)
	\oarr{\gamma_{n}^{-1}} K\left(\End^{(n)}\left(\tw{\cC}{L}{R}{\cD}\right)\right)^{C_n}
	\oarr{F^{n}} K \End^{(n)}\left(\tw{\cC}{L}{R}{\cD}\right). \]
	It follows from the definition of $\gamma_n$ that this composite is the map induced by the duplication functor $\Delta_{n}$.
\end{proof}

Using the ghost map, we can extend the result of \cref{prop:main_identify_pi_0_TR}.

\begin{thm}\label{thm:main_identify_pi_0_TR}
  Let $A$ be a ring spectrum, let $P$ be a perfect $A$-module, and let $M$ be an $(A, A)$-bimodule.  The image of the class $[f] \in K_0 (A; M)$ determined by a twisted endomorphism $f\colon P \to M \sma_A P$ under the composite
 \begin{equation}\label{eq:composite_main_identify}
 \begin{tikzcd} K(A; M) \ar{r}{\trc}
	& \TR(A; M) \ar{r}{g_{n}} &\THH^{(n)} (A; M) \ar{r}{\simeq} & \THH(A; M^{\sma_{A}n})
\end{tikzcd}
\end{equation}
 is the homotopy class of the trace of the iterate $[\tr(f^{\circ n})] \in \pi_0 \THH(A; M^{\sma_{A} n})$.
\end{thm}

\begin{proof}
	By naturality of the ghost map, we may apply $g_n$ before simplifying from $\Perf$ to $A$. This gives the top route in the commutative square of \cref{dennis_trace_ghost} followed by the equivalence $\THH^{(n)}(\spcat{\cP}; \spcat{\cM}_{M}) \simeq \THH(A; M^{\sma_{A} n})$. The composite then agrees with the bottom route of \cref{dennis_trace_ghost} composed with this equivalence, and thus, by \cref{prop:main_identify_pi_0_TR} with each $f_i = f$, takes $[f]$ to $[\tr(f^{\circ n})] \in \pi_0 \THH(A; M^{\sma_{A} n})$.
\end{proof}

\section{Characteristic polynomials, zeta functions, and the Reidemeister trace}\label{sec:char_poly_zeta}

In this section we explain how the characteristic polynomial, the Reidemeister trace, and the Lefschetz zeta function are all encoded by the $\TR$-trace. The first result is that when $A$ is a commutative Eilenberg--Maclane spectrum, the $\TR$-trace
\[
	K_0 \End(A) \to \pi_0 \TR(A) \cong W(A) \cong (1+tA[[t]])^\times
\]
takes each endomorphism $[f] \in K_0 \End(A)$ to its characteristic polynomial (\cref{tr=char}). Here $W(A) \cong (1+tA[[t]])^\times$ is the ring of big Witt vectors of $A$, and the isomorphism $\pi_0 \TR(A) \cong W(A)$ is a result of Hesselholt and Madsen \cite{hesselholt_madsen}, see also \cite{hesselholt_noncommutative,dknp}.
As a result, the $\TR$-trace of this paper is a generalization of the characteristic polynomial map $K_0 \End(A) \to (1+tA[[t]])^\times$ studied by Almkvist and others \cite{almkvist}.

The second result is that when $A = \Sigma^\infty_+ \Omega X$ is a spherical group ring with $X$ finitely dominated and path-connected, each based map $f\colon X \to X$ defines a class
\[
[f] \in K_0( \Sigma^\infty_+ \Omega X ; \Sigma^\infty_+ \Omega^f X )
\]
whose image in $\TR$ records the Reidemeister traces $R(f^n)$ for all $n \geq 1$. Using \cite{mp1}, this implies that the trace to $\TR$ takes $[f]$ to its Fuller trace $R(\Psi^n f)^{C_n}$ for all $n \geq 1$ (\cref{fuller_underneath_k_theory}). On the other hand, if we map forward to $\TR(\bbZ) \cong W(\bbZ)$ then this class becomes the Lefschetz zeta function of $f$ (\cref{lefschetz_zeta_function}).

In each of the above two cases there is a splitting of $\pi_0\TR$ into an infinite product, but the splittings arise for very different reasons. We refer to the splitting of $\TR(\Sigma^\infty_+ \Omega X)$ as {\bf tom Dieck coordinates} and the splitting of $\TR(HA)$ for a commutative ring $A$ as {\bf Witt coordinates}. These should not be confused with each other, nor should they be confused with the image under the ghost map, which we call {\bf ghost coordinates.} When we apply the ring homomorphism $\bbS \to \bbZ$ to move between the above two examples, all three of these coordinate systems come into play (\cref{crazy_isomorphism}). The distinction between them is needed to fully comprehend how the Lefschetz zeta function of a map $f\colon X \to X$ is related to the Fuller trace $R(\Psi^n f)^{C_n}$.

\subsection{The case of commutative Eilenberg--MacLane spectra}\label{sec:eilenberg_maclane}

The case of commutative Eilenberg--MacLane spectra is intimately tied with the Witt vectors.  There is a conceptual reason for this: by \cref{thm:main_identify_pi_0_TR}, the $\TR$-trace records the traces $\tr(f^{\circ n})$ of the iterates of an endomorphism $f$.  The structure of the Witt vectors collates this information into a single class in $\pi_0 \TR(A)$ which corresponds to the characteristic polynomial $\chi_{f}(t) = \det (1 - t f)$.

In this section, $A$ is a discrete commutative ring.  We abuse notation and make the abbreviations $\TR(A) = \TR(HA)$, etc., so that we have no need to explicitly refer to the associated Eilenberg--MacLane spectrum.

We briefly recall the basic facts regarding the ring of (big) Witt vectors $W(A)$ of $A$, referring the reader to \cite{almkvist, grayson_groth, hesselholt_survey, campbell_witt} for more details.  As a set, $W(A) = A^{\bbZ_{+}}$ consists of collections $a = (a_n)$ of ring elements $a_n \in A$ indexed by the positive integers $n \geq 1$.  We equip $W(A)$ with the addition and multiplication uniquely determined by the requirement that the \textbf{ghost coordinates} map
\begin{equation}\label{eq:witt_ghost_coords}
w = (w_n) \colon W(A) \arr A^{\bbZ_{+}}, \qquad w_{n}(a) = \sum_{d \mid n} d a_{d}^{n/d}
\end{equation}
is a natural transformation of functors from commutative rings to commutative rings, where the ring structure on $A^{\bbZ_{+}}$ is defined componentwise.  There is a natural isomorphism of abelian groups
\begin{equation}\label{eq:witt_iso_power_series}
W(A) \cong (1 + tA[[t]])^{\times} \quad \text{defined by} \; \;
a \longmapsto \prod_{n \geq 1} (1 - a_{n}t^n),
\end{equation}
and the group of invertible power series can be given an additional binary operation for which this is an isomorphism of rings.  The relevance of the ring structure on $W(A) \cong (1 + tA[[t]])^{\times}$, for our purposes, is that the characteristic polynomial map
\begin{align*}
K_0(\End(A)) &\oarr{\chi} (1 + tA[[t]])^{\times} \\
[f \colon P \to P] &\longmapsto \det (1 - t f)
\end{align*}
is a ring homomorphism.  In fact:
\begin{thm}[\cite{almkvist}]\label{char_is_inj}
	For any discrete commutative ring $A$, $\chi$ induces an injective ring homomorphism
\[
\widetilde K_0(\End(A)) \coloneqq K_0(\End(A)) \ / \ \iota_1 K_0(A) \oarr{\chi} (1 + tA[[t]])^{\times}
\]
and the image consists of precisely those power series that are quotients of polynomials.
\end{thm}
\begin{cor}\label{preserves_surj}
	A surjective ring homomorphism $A \to B$ induces a surjection
\[\widetilde K_0(\End(A)) \to \widetilde K_0(\End(B)).\]
\end{cor}

In order to state Hesselholt--Madsen's isomorphism $W(A) \cong \pi_0\TR(A)$ in a useful way, we will also need the Witt vectors $W_{\langle n \rangle} (A) = A^{\langle n \rangle}$ indexed on the truncation set
\[\langle n \rangle = \{ d \in \bbZ_+ : d \mid n \}.\]
The set $W_{\langle n \rangle} (A)$ is made into a commutative ring by declaring that the ghost coordinates map $w \colon W_{\langle n \rangle} (A) \arr A^{\langle n \rangle}$, defined as in \eqref{eq:witt_ghost_coords}, is a ring homomorphism.  The restriction maps $R^{n/d} \colon W_{\langle n \rangle} (A) \arr W_{\langle d \rangle} (A)$, which forget the elements indexed by divisors of $n$ that do not divide $d$, are also ring homomorphisms.  We make the identification $W(A) \cong \lim_n W_{\langle n \rangle} (A)$, where the limit is taken over the restriction maps.

We now recall the result of Hesselholt--Madsen, expressed in terms of the restriction system $\THH^{\bullet}(A)$ from \S\ref{subsec:restriction_systems}.

\begin{thm}\cite[Add. 3.3]{hesselholt_madsen}\label{hm_isom}
  There is a natural isomorphism of rings
  \[
  I_n \colon W_{\langle n \rangle} (A) \xrightarrow{\cong} \pi_0 \THH^{(n)}(A)^{C_n}
  \]
  defined by
  \[
  I_n (a) = \sum_{d \mid n} V^{d} (\Delta_{n/d} (a_d)).
  \]
\end{thm}
Here
  \[
  \Delta_{r} \colon A \cong \pi_0 \THH(A) \arr \pi_0\THH^{(r)}(A)^{C_r}
  \]
  is the duplication map induced by the map $a \mapsto a^{\sma r}$ on 0-skeleta,
  and
  \[
  V^{d} \colon \pi_0\THH^{(r)}(A)^{C_r} \arr \pi_0 \THH^{(dr)}(A)^{C_{dr}}
  \]
  is the Verschiebung map on $\THH$, defined in terms of the equivariant transfer for the subgroup $C_r < C_{dr}$.  The isomorphism $I$ respects the ghost coordinate maps on its doman and codomain, in the sense that $g_d \circ I_n = w_d$ for every $d \mid n$, where
  \[
  g_d \colon  \pi_0 \THH^{(n)}(A)^{C_n} \xarr{R^{n/d}} \pi_0 \THH^{(d)}(A)^{C_d} \xarr{F^{d}} \pi_0 \THH^{(d)}(A) \cong A
  \]
  is the $d$-th ghost coordinate map on $\pi_0 \THH^{(n)}(A)^{C_n}$.
\begin{proof}
Since our conventions are different from those of Hesselholt--Madsen, it's worth saying something about the proof.  The main point is that under the equivalence
\[\THH^{(n)}(A) \simeq \THH(A)\]
 of \cref{unwinding} (see also \cref{rmk:unwinding_subdivision}), our definitions of the restriction map $R^{n}$, in terms of the restriction system structure map, and the Frobenius map $F^{n}$, in terms of the inclusion of fixed-points, agree with Displays (1) and (16) of \cite{hesselholt_madsen}.  It follows that on the $n$-th component $\THH^{(n)}(A)^{C_n}$ of the restriction system, our ghost map $g$ (see \cref{defn:ghost_map}) agrees with theirs (denoted by $\overline{w}$).
The statement of the theorem then follows as in \cite[Add. 3.3]{hesselholt_madsen}.
\end{proof}

\begin{lem}\label{ghosts_agree}
There is a natural isomorphism of rings $I\colon W(A) \oarr{\cong} \pi_0\TR(A)$ that respects the ghost coordinate maps, meaning that $g \circ I = w$.
\end{lem}
\begin{proof}
	The map we want is essentially $I = \lim_n I_n$, the limit over the restriction maps $R$ of the isomorphisms $I_n$. Technically, this produces an isomorphism to $\lim_n \pi_0\THH(A)^{C_n}$, but it lifts to an isomorphism to $\pi_0 \lim_n\THH(A)^{C_n}$ because $\lim^1 \pi_1\THH(A)^{C_n} = 0$. This last claim follows from the surjectivity of $\overline{R}$ in the proof of \cite[Prop 3.3]{hesselholt_madsen}, which generalizes from the $p$-typical case by the discussion on \cite[p. 55]{hesselholt_madsen}.
\end{proof}

Let $f\colon P \to P$ be an endomorphism of a finitely generated projective $A$-module, with associated $K$-theory class $[f] \in K_0(\End(A))$. By \cref{thm:main_identify_pi_0_TR}, the image of $[f]$ under the $\TR$-trace and the ghost map is
\[
(\tr(f), \tr(f^{\circ 2}), \tr(f^{\circ 3}), \dots) \in \prod_{n \geq 1} \pi_0 \THH^{(n)}(A;A) \cong \prod_{n \geq 1} A.
\]
As explained in \cite{hesselholt_survey,campbell_witt}, along the isomorphism $W(A) \cong (1+tA[[t]])^\times$ of \eqref{eq:witt_iso_power_series},
the ghost map of $W(A)$ is identified with the negative logarithmic derivative
\[ \xymatrix @C=4em{ (1+tA[[t]])^\times \ar[r]^-{-t\frac{d}{dt} \log} & tA[[t]] \cong \displaystyle\prod_{n \geq 1} A. } \]
\begin{lem}\label{derivative_of_chi}
	The element of $\prod_{n \geq 1} A$ given by the iterated traces $(\tr(f^{\circ n}))$ is the negative logarithmic derivative of the characteristic polynomial
\[ \chi_f(t) = \det(\id - tf) \in (1+tA[[t]])^\times. \]
\end{lem}
\noindent In other words, the following diagram commutes.
\begin{equation}\label{eq:derivative_of_chi} \xymatrix @R=1.5em {
	& \widetilde K_0 (\End (A)) \ar[dl]_-{\trc} \ar[dr]^-\chi & \\
	\pi_0 \TR(A) \ar[d]_-{\prod g_n} && (1+tA[[t]])^\times \ar[d]^-{-t\frac{d}{dt} \log} \\
	\prod_{n=1}^\infty A \ar@{<->}[rr] && tA[[t]]
} \end{equation}
\begin{proof}
	When $A$ is an algebraically closed field, this follows by induction on the eigenvalues. It therefore holds for any integral domain, by passing to the algebraic closure of the fraction field. In particular, it holds for any polynomial ring over $\bbZ$. If $A$ is a general commutative ring, then there is a surjection from a polynomial ring to $A$. Using \cref{preserves_surj}, the statement therefore holds for $A$ as well. See \cite[4.24]{campbell_witt} for a different derivation.
\end{proof}

\begin{thm} \label{tr=char}
	Let $A$ be a commutative ring. Then the following triangle commutes.
	\begin{equation}\label{eq:tr=chr} \xymatrix{
		& \widetilde K_0 (\End (A)) \ar[dl]_-{\trc} \ar[dr]^-\chi & \\
		\pi_0 \TR(A) \ar[r]^-\cong_-{I^{-1}} &  W(A) \ar[r]^-\cong_-{\eqref{eq:witt_iso_power_series}} & (1+tA[[t]])^\times
	} \end{equation}
	In other words, as invariants underneath the $K$-theory of endomorphisms, the trace to $\pi_0 \TR(A)$ is isomorphic to the characteristic polynomial.
\end{thm}

\begin{proof}
Pasting the  diagrams in \cref{eq:derivative_of_chi,eq:tr=chr}  together gives the diagram
	\[ \xymatrix{
		& \widetilde K_0 (\End (A)) \ar[dl]_-{\trc} \ar[dr]^-\chi & \\
		\pi_0 \TR(A) \ar[d]_-{\prod g_n} & \ar[l]_-\cong^-I W(A) \ar[dl]^-{\prod w_n} \ar[r]^-\cong & (1+tA[[t]])^\times \ar[d]^-{-t\frac{d}{dt} \log} \\
		\prod_{n=1}^\infty A \ar@{<->}[rr] && tA[[t]].
	} \]
	The maps along the outside edge commute by the previous discussion, as does the trapezoid and the small triangle on the left (\cref{ghosts_agree}). If $A$ is torsion-free, the two vertical maps are injective and therefore the triangle at the top commutes as well.

	To extend to the case where $A$ is any commutative ring we use a standard trick (see e.g. \cite[p. 6]{grayson_groth}). Pick a surjective ring homomorphism $A' \to A$ with $A'$ torsion-free, and observe that the diagram is natural in ring homomorphisms. This gives a map from the diagram for $A'$ to the diagram for $A$, that is surjective on the terms
\[\pi_0\TR(A) \cong W(A) \cong (1+tA[[t]])^\times\]
 and on $\widetilde K_0 (\End (A))$ by \cref{preserves_surj}. Therefore the desired triangle for $A$ can be deduced from the same triangle for $A'$.
\end{proof}

\begin{cor}
	Let $A$ be a commutative ring. Then the TR-trace $\widetilde K (\End (A)) \to \TR(A)$ is injective on $\pi_0$.
\end{cor}

\subsection{The case of spherical group rings and fixed-point theory}

As in the case of Eilenberg--MacLane spectra, the computation of $\pi_0 \TR(A)$ for $A = \Sigma^{\infty}_{+}G$ a spherical group ring hinges on the interplay between the ghost coordinates and a set of splitting coordinates for $\pi_0 \TR(A)$. However the splitting here arises for a very different reason, namely the tom Dieck splitting from equivariant stable homotopy theory.

\begin{prop}\label{naive_tr}
	If $X_\bullet$ is a naive restriction system of spaces, then the $\TR$ of its suspension genuine restriction system $\Sigma^\infty X_\bullet$ from \cref{suspend_naive_system_of_spaces} has a tom~Dieck splitting
	\[ \TR(\Sigma^\infty X_\bullet) \simeq \prod_{j \geq 1} (\Sigma^\infty X_j)_{hC_j}. \]
\end{prop}

\begin{proof}
	We interpret the naive restriction system as a $\bbZ$-space $X$ with no free orbits, so that $X_n = X^{n\bbZ}$. Under the identification $C_n = \bbZ/n\bbZ$, the generator of the cyclic group $C_n$ acts on $X_n = X^{n\bbZ}$ by $1 \in \bbZ$, and the generator of the subgroup $C_{n/j} < C_{n}$ acts by $j \in \bbZ$ for each positive divisor $j \mid n$.

	The proposition is a consequence of the classical tom~Dieck splitting theorem, which tells us that the derived fixed point spectrum of the suspension spectrum of $X_n$ splits as a wedge of homotopy orbits
	\[
	(\Sigma^{\infty} X_n)^{C_n} = (\Sigma^{\infty} X^{n\bbZ})^{C_n} \simeq \bigvee_{j \mid n} \Sigma^{\infty} (X^{j \bbZ})_{h C_{j}} = \bigvee_{j \mid n} \Sigma^{\infty} (X_j)_{h C_{j}}.
	\]
	Here all of the fixed points labeled by cyclic groups $C_i$ are genuine fixed points, i.e. they are implicitly right-derived. We also need the standard fact that the restriction $(\Sigma^\infty X^{n\bbZ})^{C_n} \to (\Sigma^\infty X^{k\bbZ})^{C_k}$ for $k | n$ corresponds along this splitting to the map that restricts to the summands where $j | k$.

	These statements only apply in the homotopy category, so they do not directly imply anything about the homotopy limit defining $\TR$. However, the map from $j$th summand is described more concretely as a transfer followed by a splitting of the restriction map:
	\[
		\Sigma^{\infty} (X^{j \bbZ})_{h C_{j}}
		\xarr{trf} \left(\Sigma^{\infty} X^{j \bbZ}\right)^{C_{j}}
		\arr \left(\left(\Sigma^{\infty} X^{n\bbZ}\right)^{C_{n/j}}\right)^{C_{j}}
		= (\Sigma^{\infty} X^{n\bbZ})^{C_n}.
	\]
	Therefore, if we pick a representative for the splitting $t_n\colon (\Sigma^{\infty} X_{n})^{C_n} \to (\Sigma^{\infty} X_{n})_{h C_n}$ for each $n \geq 1$, the $j$th term of the splitting in the homotopy category is given by the formula
	\[
	(\Sigma^{\infty} X_{n})^{C_n}
	= \left(\left(\Sigma^{\infty} X_n\right)^{C_{n/j}}\right)^{C_{j}}
	\xarr{\kappa} \left(\Sigma^{\infty} X_j \right)^{C_{j}}
	\xarr{t_j} \Sigma^{\infty} (X_j)_{h C_{j}}.
	\]
	This defines a map of homotopy limit systems from $(\Sigma^\infty X_\bullet)^{C_\bullet}$ to a product of homotopy limit systems over $j \geq 1$, the $j$th system having $n$th term $\Sigma^{\infty} (X_j)_{h C_{j}}$ when $j | n$ and $*$ when $j \nmid n$, with all identity maps between them. On each term this map of homotopy limit systems is an equivalence by the above discussion, so it induces an equivalence of homotopy limits. This gives the desired tom Dieck splitting of $\TR$.
      \end{proof}

      \begin{rmk}
	There is a generalization of the tom Dieck splitting theorem due to Gaunce Lewis \cite{lewis_splitting} that applies to the orthogonal $C_r$-spectrum $\bbP\Sigma^\infty K\left(\End^{(r)}\left(\tw{\cC}{L}{R}{\cD}\right)\right)$. We could use this together with the argument in \cref{naive_tr} below to identify the middle term of \eqref{map_of_trs} as the infinite product
	\begin{equation}\label{eq:TR_trace_out_of_product}
	\prod_{n \geq 1} K\left(\End^{(n)}\left(\tw{\cC}{L}{R}{\cD}\right)\right)_{hC_n}.
	\end{equation}
	The inclusion of underived $\TR$ is then just the first term in this product. Note that the map to the $\TR$ system described in \cref{naive_tr} does not respect this splitting.
\end{rmk}

\begin{example}\label{naive_tr_ghost}
	Along the tom~Dieck splitting, the ghost map is a map of products
	\[ g \colon \prod_{j \geq 1} (\Sigma^\infty X_j)_{hC_j} \arr \prod_{n \geq 1} \Sigma^\infty X_n. \]
	To compute its $n$th coordinate we use the product system from the proof of \cref{naive_tr}:
	\[ \xymatrix @R=1.5em{
		\TR(\Sigma^\infty X_\bullet) \ar[d]^-\simeq \ar[r] & (\Sigma^\infty X_n)^{C_n} \ar[d]^-\simeq \ar[r]^-F & \Sigma^\infty X_n \\
		\displaystyle\prod_{j \geq 1} (\Sigma^\infty X_j)_{hC_j} \ar[r] & \displaystyle\prod_{j | n} (\Sigma^\infty X_j)_{hC_j} \ar@{-->}[ur]
	} \]
	We then use \cite[4.4]{malkiewich_tc_ds1} to compute the dashed composite as the sum over all $j | n$ of the transfers and inclusions
	\begin{equation}\label{eq:transfer_inclusion}
		(\Sigma^{\infty} X_j)_{h C_{j}}
		\xarr{\mathrm{trf}} \Sigma^{\infty} X_j
		\cong \Sigma^\infty (X_n)^{C_{n/j}}
		\arr \Sigma^\infty X_n.
	\end{equation}
	On $\pi_0$, the composite takes every path component of $(X_j)_{h C_{j}}$ to the weighted sum of its preimage components in $X_j$ (weighted evenly so that the total weight is $j$), then maps forward to the corresponding components of $X_n$.  Here is a useful consequence of this description in the case $j = n$.
\end{example}

\begin{prop}\label{ghost_injectivity}
	For a suspension spectrum restriction system, the ghost map on $\pi_0$
	\[ g = (g_n) \colon \pi_0\TR(\Sigma^\infty X_\bullet) \arr \prod_{n \geq 1} \pi_0 \Sigma^{\infty}_{+} X_n \cong  \prod_{n \geq 1} H_0(X_n) \]
is injective.
\end{prop}

We now apply the tom Dieck splitting to $\TR$ of a spherical group ring. Let $G$ be a topological group or grouplike topological monoid and write $\bbS[G] = \Sigma^\infty_+ G$ for its suspension spectrum. Let $A$ be a topological space with commuting left and right $G$-actions. Assume that $G$ and $A$ are cofibrant as topological spaces. Then the restriction system $\THH^{\bullet}(\bbS[G];\bbS[A])$ from \cref{thh_system} is the suspension spectrum of a naive restriction system whose $n$th level is the bar construction in unbased spaces $B(G^{\times n};A^{\times n}_\rho)$. \cref{naive_tr} therefore gives a splitting (where the ${}_{C_j}$ denotes coinvariants)
\begin{align*}
	\TR(\bbS[G];\bbS[A]) &\simeq \prod_{j \geq 1} \Sigma^\infty_+ B(G^{\times j};A^{\times j}_\rho)_{h C_j}, \\
	\pi_0 \TR (\bbS[G]; \bbS[A]) &\cong \prod_{j \geq 1} \HH_0 (\bbZ[\pi_0 G]^{\otimes j}; \bbZ[\pi_0 A]^{\otimes j}_\rho)_{C_j}, \\
	\pi_0 \TR (\bbS[G]) &\cong \prod_{j \geq 1} \HH_0 (\bbZ[\pi_0 G]) \\
\end{align*}

For a based connected CW complex $X$, we take $G = \Omega X$ to be any well-based topological group modeling the loop space of $X$. Let $f\colon X \to X$ be a basepoint preserving self-map of $X$.  We let $A = \Omega^f X$ be $\Omega X$ with the usual right action, and left action twisted by $f$, i.e. given by the composite
\[
\Omega X \times \Omega X \xrightarrow{\Omega f \times \id} \Omega X \times \Omega X \xarr{\text{mult}} \Omega X.
\]
The tom Dieck coordinates can then be described as
\begin{align*}
	\pi_0 \TR (\bbS[\Omega X]; \bbS[\Omega^f X]) &\cong \prod_{j \geq 1} \HH_0 (\bbZ[\pi_1 X]; \bbZ[\pi_1 X]_{f^{\circ j}})_{f_*}
\end{align*}
where $(-)_{f_*}$ denotes coinvariants under the action of $f_*$ on each copy of $\pi_1 X$.

The above is true in general, but if $X$ is finitely dominated, then $\bbS$ is perfect as a left $\bbS[\Omega X]$-module, and we can pick out a distinguished class
\[ [f] \in K_0(\End(\bbS[\Omega X];\bbS[\Omega^f X])). \]
It is the twisted endomorphism of $(\bbS[\Omega X], \bbS)$-bimodule spectra
\begin{equation}\label{twisted_module_morphism}
\bbS \stackrel\simeq\arr \bbS[\Omega^f X] \sma_{\bbS[\Omega X]} \bbS
\end{equation}
that is homotopy inverse to the map that collapses the bar construction on the right back to $\bbS$.
This is the suspension of the canonical
isomorphism
\[ \bcr{\Omega X}{}{*} \stackrel\cong\arr \bcr{\Omega X}{f}{\Omega X} \odot \bcr{\Omega X}{}{*} \]
arising from the fact that $f$ and the identity agree after composing with $X \to *$. By \cite{ponto_thesis}, its bicategorical trace is the Reidemeister trace
\[ R(f)\colon \bbS \to \THH (\bbS[\Omega X]; \bbS[\Omega^f X]). \]

\begin{thm}\label{reidemeister_series}
	For any finitely dominated space $X$ and basepoint preserving self map $f\colon X \to X$, the image of the class $[f]$ from \eqref{twisted_module_morphism} under the composite
	\[
	\xymatrix{
	K_0(\bbS[\Omega X]; \bbS[\Omega^f X]) \ar[r]^-{\trc} & \pi_0\TR(\bbS[\Omega X]; \bbS[\Omega^f X]) \ar[r]^-{g} & \displaystyle\prod_{n \geq 1} \HH_0 (\bbZ[\pi_1 X]; \bbZ[\pi_1 X]_{f^{\circ n}})
}
	\]
	is the Reidemeister series $(R(f),R(f^{\circ 2}),R(f^{\circ 3}), \dots)$.
\end{thm}
\begin{proof}
	By \cref{thm:main_identify_pi_0_TR}, the $n$th factor of this map is the trace of the composite
	\[
	\bbS \cong \bbS[\Omega^f X] \sma_{\bbS[\Omega X]} \bbS \cong \bbS[\Omega^f X] \sma_{\bbS[\Omega X]} \bbS[\Omega^f X] \sma_{\bbS[\Omega X]} \bbS \cong \dotsm \cong \bbS[\Omega^f X]^{\sma_{\bbS[\Omega X]} (n)} \sma_{\bbS[\Omega X]} \bbS.
	\]
	Collapsing the copies of $\bbS[\Omega^f X]$ together gives the bimodule $\bbS[\Omega^{f^{\circ n}} X]$ and the twisted endomorphism \eqref{twisted_module_morphism} with $f$ replaced by $f^{\circ n}$. Therefore, along the resulting isomorphism
	\begin{equation}\label{composition_isomorphism}
	\THH(\bbS[\Omega X]; \bbS[\Omega^f X]^{\sma_{\bbS[\Omega X]} n}) \cong \THH(\bbS[\Omega X]; \bbS[\Omega^{f^{\circ n}} X])
	\end{equation}
	this trace is taken to the Reidemeister trace of $f^{\circ n}$.
\end{proof}

\begin{rmk}
	The $\TR$-trace without coefficients $A(X) \simeq K(\bbS[\Omega X]) \to \TR(\bbS[\Omega X])$ therefore gives traces of the identity map, as in \cite{lydakis}.
\end{rmk}

The ring map $\pi \colon \bbS[\Omega X] \to \bbS$ that collapses $\Omega X$ to a point and the corresponding bimodule map $\bbS[\Omega^f X] \to \bbS$ induce a map on topological restriction homology
\[ \xymatrix @R=1.5em{
	\TR(\bbS[\Omega X]; \bbS[\Omega^f X]) \ar[r]^-{\pi} & \TR(\bbS).
} \]
We may further compose with the ring map $\bbS \to H\bbZ$ to land in $\TR(\bbZ)$.

\begin{thm}\label{lefschetz_zeta_function}
  The composite
  \[
    K_0(\bbS[\Omega X], \bbS[\Omega^f X]) \xarr{\trc} \pi_0 \TR(\bbS[\Omega X]; \bbS[\Omega^f X]) \oarr{\pi} \pi_0 \TR (\bbS) \arr \pi_0 \TR (\bbZ) \cong (1 + t\bbZ[[t]])^\times
  \]
  maps the twisted module endomorphism $[f]$ from \eqref{twisted_module_morphism} to the Lefschetz zeta function
  \[
  \exp \left( \int \frac{1}{t} \left(\sum^\infty_{n=1} L(f^{\circ n}) t^n \right) \right) = \exp \left(\sum^\infty_{n=1} \frac{L(f^{\circ n})}{n} t^n \right).
  \]
\end{thm}

\begin{proof}
	Naturality of the ghost map implies that the following diagram commutes, where the vertical maps are ghost maps:
	\begin{equation}\label{three_ghosts_of_christmas}
	\xymatrix{
		\pi_0\TR(\bbS[\Omega X]; \bbS[\Omega^f X]) \ar[d]_-g \ar[r] & \pi_0 \TR(\bbS) \ar[d]_-{g} \ar[r] & (1 + \bbZ[[t]])^\times \ar[d]^-{-t \frac{d}{dt} \log} \\
		\displaystyle\prod_{n \geq 1} \HH_0 (\bbZ[\pi_1 X]; \bbZ[\pi_1 X]_{f^{\circ n}}) \ar[r] & \displaystyle\prod_{n \geq 1} \bbZ \ar@{<->}[r]_-\cong & t\bbZ[[t]]
	}
	\end{equation}
	By \cref{reidemeister_series}, the class $[f]$ in the upper-left goes to the Reidemeister traces $R(f^{\circ n})$ in the lower-left. The augmentation sends these to the Lefschetz numbers $L(f^{\circ n})$ in the lower-middle. This agrees with the image of the Lefschetz zeta function along the logarithmic derivative. Since the logarithmic derivative is injective, we conclude that the image of $[f]$ in the upper-right is the Lefschetz zeta function.
\end{proof}

In fact, the passage from $\bbS$ to $\bbZ$ in \cref{lefschetz_zeta_function} has no effect on $\pi_0\TR$:

\begin{prop}\label{crazy_isomorphism}
	The ring map $\bbS \to H\bbZ$ induces an isomorphism
	\[ \pi_0 \TR(\bbS) \cong \pi_0 \TR(\bbZ) \cong (1 + t\bbZ[[t]])^\times. \]
\end{prop}

\begin{proof}
	Expanding the right-hand square of \eqref{three_ghosts_of_christmas} using the tom~Dieck splitting and Witt coordinates, we get the commuting square
	\[ \xymatrix @R=1.5em{
		\displaystyle\prod_{j \geq 1} \bbZ \ar[d]_-{g} \ar[r] & \displaystyle\prod_{i \geq 1} \bbZ \ar[d]^-g && (b_j)_{j=1}^\infty \ar@{|->}[d] \ar@{|->}[r] & (a_i)_{i=1}^\infty \ar@{|->}[d] \\
		\displaystyle\prod_{n \geq 1} \bbZ \ar@{=}[r] & \displaystyle\prod_{n \geq 1} \bbZ && (w_n)_{n=1}^\infty \ar@{=}[r] & (w_n)_{n=1}^\infty.
	} \]
	The ghost coordinates $w_n$ are given by the formulas
	\[ \sum_{d|n} d b_d = w_n = \sum_{d|n} d a_d^{n/d}, \]
	the first arising from \cref{naive_tr_ghost} and the second from \cref{eq:witt_ghost_coords}. It is an easy observation that both of these maps are injective. Therefore the top horizontal map is injective.

	To prove it is surjective, we reverse-engineer the second formula to write $a_n$ as $\frac{1}{n}w_n$ plus a rational polynomial in the $a_d$ for $d | n$, $d < n$. Inductively, this implies that $a_n$ is expressed as a rational polynomial in the $b_d$ for $d | n$, whose $b_n$-term is $\frac{1}{n}(nb_n) = b_n$. Clearly changing the value of $b_n$ then allows us to attain any integer value of $a_n$, so by induction this collection of polynomials defines a surjective map to $\prod_{i \geq 1} \bbZ$.
\end{proof}

\begin{rmk}
	The isomorphism $\pi_0\TR(\bbS) \stackrel\cong\arr \pi_0\TR(\bbZ)$ is given from tom Dieck coordinates to Witt coordinates by the following polynomials in the first few degrees:
	\[
	\begin{array}{rclcrcl}
		a_1 &=& b_1								&\hspace{2em}& b_1 &=& a_1 \\
		a_2 &=& b_2 - \frac{b_1^2 - b_1}{2}		&& b_2 &=& a_2 + \frac{a_1^2 - a_1}{2} \\
		a_3 &=& b_3 - \frac{b_1^3 - b_1^3}{3}	&& b_3 &=& a_3 + \frac{a_1^3 - a_1}{3}
	\end{array}
	\]
	The polynomials for $a_4$ and $b_4$ have many more terms, for instance
	\[ a_4 = b_4 + \frac14\left( 2b_2 - b_2^2 + 2b_1^2b_2 - 2b_1b_2 - \frac{3}{2}b_1^4 + b_1^3 - \frac{1}{2}b_1^2 + b_1 \right). \]
	Though it is not directly apparent from the formula, this polynomial is integer valued on integer inputs.
\end{rmk}

\subsection{The relation to the free loop space and periodic-point theory}
For a space $X$ and self-map $f\colon X \to X$, let $\mathcal L^f X$ denote the twisted free loop space
\[ \mathcal{L}^f X = \{ \gamma\colon [0, 1] \to X| \gamma(0) = f(\gamma(1)) \}. \]
In this final subsection we describe how the work of the previous subsection is related to \cite{mp1}, which investigated the Reidemeister traces $R(f^{\circ n})$ as elements of
\[ \pi_0(\Sigma^\infty_+ \mathcal L^{f^{\circ n}} X) \cong H_0(\mathcal L^{f^{\circ n}} X), \]
rather than $\HH_0 (\bbZ[\pi_1 X]; \bbZ[\pi_1 X]_{f^{\circ n}})$. When $X$ is path-connected and $f$ preserves a chosen basepoint, the two definitions of the Reidemeister trace agree along the equivalence (see \cite[6.5]{vegetables}, \cite[Cor A.14]{cp}, \cite[8.2.8]{malkiewich_parametrized})
\begin{equation} \label{eq:THH_free_loops_basic_equiv}
 \THH( \bbS[\Omega X]; \bbS[\Omega^f X] ) \simeq \Sigma^\infty_+ \mathcal L^{f} X.
\end{equation}

We first lift this equivalence to an equivalence of restriction systems. Let $\Psi^n(f) = f^{\times n} \circ \rho^{-1}$ denote the $n$th Fuller construction of $f$ as in \cite{kw2,mp1}:
\[ \xymatrix @R=2pt {
	X\times \cdots \times X\ar[r]^{\Psi^n(f)} & X\times \cdots \times X \\
	(x_1,x_2,\ldots , x_n)\ar@{|->}[r] & (f(x_n), f(x_1), \ldots , f(x_{n-1}))
}\]
An element of the twisted free loop space $\mathcal{L}^{\Psi^n(f)} X^n$ consists of an $n$-tuple of points $x_1, \dotsc, x_n \in X$ and paths $\gamma_i$ from $f(x_i)$ to $x_{i+1}$, indices modulo $n$. As $n$ varies, the suspension spectra $\Sigma^{\infty}_{+} \mathcal{L}^{\Psi^n(f)} X^n$ form a naive restriction system (\cref{ex:naive_restriction_systems}.\ref{ex:fuller_maps}). Following \cite{mp1} and the precedent set by \cite{bokstedt_hsiang_madsen}, we denote its topological restriction homology by $\TR(X;f)$.

\begin{prop}\label{prop:TR_equiv_modules_to_freeloops}
	For any path-connected $X$ and basepoint-preserving $f\colon X \to X$, there is an equivalence of restriction systems
	\[ \THH^{(r)} (\bbS[\Omega X]; \bbS[\Omega^f X]) \simeq \Sigma^\infty_+ \mathcal{L}^{\Psi^r(f)} X^r \]
	and therefore an equivalence on $\TR$
	\[ \TR(\bbS[\Omega X]; \bbS[\Omega^f X]) \simeq \TR(X;f). \]
\end{prop}

To prove the proposition, note that both sides of the equivalence are homotopy invariant, and so we may without loss of generality assume that $X = BG$ for a cofibrant topological group $G$, and use $G$ as the model for $\Omega X$. We may also assume that $f\colon BG \to BG$ arises by applying the classifying space functor to a group homomorphism, which by abuse of notation we also denote $f\colon G \to G$.

The first step is to identify
\begin{equation}\label{thh_r_is_fuller}
	\THH^{(r)} (\bbS[G]; \bbS[\tensor[_f]{G}{}])
	= \THH (\bbS[\tensor[]{G}{^{\times r}}]; \bbS[\tensor[_{f^{\times r}}]{G}{^{\times r}_\rho}])
	\cong \THH (\bbS[\tensor[]{G}{^{\times r}}]; \bbS[\tensor[_{\Psi^r(f)}]{G}{^{\times r}}])
\end{equation}
by applying $\rho^{-1}$ once to the bimodule coordinate. This gives an isomorphism of restriction systems where the one on the right arises from the isomorphisms $(\Psi^{rs}(f))^{C_r} \cong \Psi^s(f)$. The proposition then follows from the next lemma by passage to suspension spectra.

\begin{lem}\label{lem:equiv_restrict_space_level}
There is a natural equivalence of restriction systems of spaces
\[
B^{\mathrm{cy}}(G^{\times r}; \tensor[_{\Psi^r(f)}]{G}{^{\times r}}) \oarr{\simeq} \mc{L}^{\Psi^r(f)}BG^{\times r}.
\]
\end{lem}
\begin{proof}
We start by constructing two equivalences of spaces
\begin{equation}\label{eq:two_equivs}
B^{\mathrm{cy}}(G; \tensor[_f]{G}{}) \oarr{\simeq} EG \times_{G} \tensor[_f]{G}{^{\mathrm{ad}}} \oarr{\simeq} \mc{L}^{f}BG
\end{equation}
where $\tensor[_f]{G}{^{\mathrm{ad}}}$ denotes $G$ with the left $G$-action $g \cdot a = f(g)ag^{-1}$. The first equivalence is actually an isomorphism, and is defined on $k$-simplices by
\[
(g_1, \dotsc, g_k;g) \longmapsto [g_1 \mid \dotsm \mid g_k] g g_1 \dotsm g_k.
\]
To define the second, we compare two different fibrant models for the base change 1-cell $\tensor[_{\id}]{BG}{_{f}} = [BG \oarr{f} BG]$ in the bicategory of parametrized spaces over varying base spaces. The first is the fibrant approximation of the parametrized space $(\id, f) \colon BG \arr BG^2$ given by the space of paths $(ev_0, f \circ ev_1) \colon BG^I \arr BG^2$. The second is $EG^2 \times_{G^2} \tensor[_f]{G}{}$, where $\tensor[_f]{G}{}$ is given the left $G^2$-action $(g,h)a = f(g)ah^{-1}$, and the map to $BG^2$ arises from the projection $(p_1,p_2)\colon EG^2 \to BG^2$. The equivalence
\begin{equation}\label{canonical_iso_of_base_change}
	BG \arr EG^2 \times_{G^2} \tensor[_f]{G}{}, \qquad [g_1 \mid \dotsm \mid g_k] \longmapsto [(g_1,f(g_1)) \mid \dotsm \mid (g_k,f(g_k))]e
\end{equation}
of spaces over $BG^2$ is the fibrant approximation map.
We form a commuting square of spaces over $BG^2$ where the horizontal maps are induced by \eqref{canonical_iso_of_base_change} and the vertical maps are inclusion of constant paths:
\[ \xymatrix @R=1.5em{
	BG \ar[d]_-\sim \ar[r]^-\sim & EG^2 \times_{G^2} \tensor[_f]{G}{} \ar[d]^-\sim \\
	BG^I \ar[r]_-\sim & (EG^2 \times_{G^2} \tensor[_f]{G}{})^I
} \]
The final space projects to $BG^2$ by $(p_1 \circ ev_0,p_2 \circ ev_1)$. If we remove the upper left instance of $BG$, the other three spaces are fibrant over $BG^2$, hence we can pull them back along the diagonal $\Delta\colon BG \to BG^2$ to get a zig-zag of equivalences of spaces over $BG$
\[ \mathcal L^f BG \overset\sim\arr \Delta^*\left[ (EG^2 \times_{G^2} \tensor[_f]{G}{})^I \right] \overset\sim\longleftarrow EG \times_{G} \tensor[_f]{G}{^{\mathrm{ad}}}. \]
The second equivalence in \eqref{eq:two_equivs} is this zig-zag.

Applying the construction \eqref{eq:two_equivs} to the group $G^{\times r}$ and the map $\Psi^r(f)\colon G^{\times r} \arr G^{\times r}$, we have a composite equivalence
\[
B^{\mathrm{cy}}(G^{\times r}; \tensor[_{\Psi^r(f)}]{G}{^{\times r}}) \oarr{\simeq} EG^{\times r} \times_{G^{\times r}} \tensor[_{\Psi^r(f)}]{G}{^{\mathrm{ad}\times r}} \oarr{\simeq} \mc{L}^{\Psi^r(f)}BG^{\times r},
\]
which defines each level of the equivalence of restriction systems in the statement of the lemma.
Since both of the maps are $C_r$-equivariant, and taking fixed points with respect to a subgroup of $C_r$ gives the same maps for a smaller value of $r$, the desired equivalence of restriction systems follows.
\end{proof}

\begin{rmk}\label{rmk:smbfs}
The map \eqref{eq:two_equivs} is the canonical equivalence between two different models for $r_{!} \Delta^* (\tensor[_{\id}]{BG}{_{f}})$ that are computed by deriving the base change functor $\Delta^*$ in two different ways.  It follows that each level of the equivalence of restriction systems in \cref{prop:TR_equiv_modules_to_freeloops} is a point set model for the comparison map of shadows induced by the equivalence of symmetric monoidal bifibrations from \cite[\S8.2]{malkiewich_parametrized} (see also \cite[\S14]{mp2}).
\end{rmk}

Recall from \cite{mp1} that the Fuller trace $R(\Psi^n(f))^{C_n}$ is defined as the $C_n$-equivariant Reidemeister trace of the map $\Psi^n(f)$. By the main theorem of \cite{mp1}, these assemble to define a class
\begin{equation*}
	R(\Psi^\infty(f)) \in \pi_0\TR(X;f)
\end{equation*}
called the infinite Fuller trace.

\begin{thm}\label{fuller_underneath_k_theory}
The composite
  \[
    K_0(\bbS[\Omega X], \bbS[\Omega^f X]) \xarr{\trc} \pi_0 \TR(\bbS[\Omega X]; \bbS[\Omega^f X]) \cong \pi_0 \TR(X; f)
  \]
of the $\TR$-trace and the equivalence from \cref{prop:TR_equiv_modules_to_freeloops} takes the class of the twisted module endomorphism $[f]$ from \eqref{twisted_module_morphism} to the infinite Fuller trace $R(\Psi^{\infty}(f))$.
\end{thm}
\begin{proof}
We examine the image of the class $[f]$ under the $\TR$-trace and the routes in the commutative diagram
\[
\begin{tikzcd}
\pi_0 \TR(\bbS[\Omega X]; \bbS[\Omega^f X]) \ar{r}{\cong} \ar{d}{g_n} & \pi_0 \TR(X; f) \ar{d}{g_n} \\
\pi_0 \THH(\bbS[\Omega X]; \bbS[\Omega^{\Psi^n(f)} X]) \ar{r}{\cong} \ar{d}{\cong} & \pi_0 \Sigma^{\infty}_+ \mc{L}^{\Psi^{n}(f)} X^{\times n} \\
\pi_0 \THH(\bbS[\Omega X]; \bbS[\Omega^{f^{\circ n}} X]),
\end{tikzcd}
\]
where the upper square commutes by the naturality of the ghost maps for the equivalence of restriction systems from \cref{prop:TR_equiv_modules_to_freeloops}, and the lower-left vertical map is the composition of \eqref{thh_r_is_fuller} and \eqref{composition_isomorphism}. \cref{reidemeister_series} states that the image in the lower left corner is the Reidemeister trace $R(f^{\circ n})$ of the iterate.  By the unwinding argument of \cite[Thm. 1.1]{mp1}, it follows that the image in the middle-left is the Fuller trace $R(\Psi^n(f))$. Therefore the image in the middle-right is $R(\Psi^n(f))$ as computed in parametrized spectra. By \cref{ghost_injectivity} the upper vertical maps are jointly injective, hence the image of $[f]$ in $\pi_0 \TR(X; f)$ is the infinite Fuller trace $R(\Psi^{\infty}(f))$.
\end{proof}

Since the infinite Fuller trace capture the behavior of $n$-periodic points for every $n \geq 1$, the theorem suggests that $K$-theory might capture deeper dynamical information. We plan to continue this investigation in future work.

\appendix
\section{Model categories of restriction systems}\label{sec:model_str_restriction_systems}

Recall that in \S\ref{sec:trace_to_TR} we introduced the notion of a ``restriction system'' to link the equivariant Dennis traces together and define the $\TR$-trace. A restriction system is like a cyclotomic spectrum, but more general. In this appendix, we show how to move these restriction systems between different models of spectra, and we establish a model structure for pre-restriction systems in the spirit of \cite{blumberg_mandell_cyclotomic}.

Recall the notion of an equivariant bispectrum from \cite[\swcref{Appendix A}{\cref{swc-sec:model_structures}}]{clmpz-specWaldCat}. We define genuine restriction systems of these bispectra by using the geometric fixed point functor from \cite[\swcref{Definition A.9}{\cref{swc-gfp_bispectra}}]{clmpz-specWaldCat}.

\begin{prop}\label{prolong_restriction_systems}
	If $X_\bullet$ is a genuine (pre-)restriction system of bispectra and each $X_n$ is cofibrant as a $C_n$-bispectrum then the prolongation $\bbP X_\bullet$ to orthogonal spectra is naturally a genuine \mbox{(pre-)restriction} system of orthogonal spectra.
\end{prop}
\begin{proof}
	We define the structure maps by
	\[ \xymatrix{
		\Phi^{C_r} \bbP X_{rs} & \ar[l]_-\cong \bbP \Phi^{C_r} X_{rs} \ar[r]^-{\bbP c_{r}} & \bbP X_s.
	} \]
	By \cite[\swcref{Proposition A.7}{\cref{swc-quillen-adjoints}}]{clmpz-specWaldCat}, $\bbP$ preserves equivalences of cofibrant spectra, so if the maps $c_r$ are stable equivalences then so are these new maps. To check these form a pre-restriction system it suffices to show the diagram below commutes. The top-right region commutes because the bispectra $X_\bullet$ form a restriction system. The bottom-right region commutes by naturality and the left-hand region commutes by \cite[\swcref{Lemma A.13}{\cref{swc-lem:rigidity}}]{clmpz-specWaldCat}.
	\[ \xymatrix @R=1.5em @C=5em {
		\Phi^{C_{rs}} \bbP X_{rst} \ar[dd]^-{\textup{it}} & \ar[l]_-\cong \bbP \Phi^{C_{rs}} X_{rst} \ar[d]^-{\textup{it}} \ar[r]^-{\bbP c_{rs}} & \bbP X_t \\
		& \bbP \Phi^{C_{s}}\Phi^{C_{r}} X_{rst} \ar[d]^-\cong \ar[r]^-{\bbP \Phi^{C_s} c_r} & \bbP \Phi^{C_{s}} X_{st} \ar[u]_-{\bbP c_s} \ar[d]^-\cong \\
		\Phi^{C_{s}}\Phi^{C_{r}} \bbP X_{rst} & \ar[l]_-\cong \Phi^{C_{s}} \bbP \Phi^{C_{r}} X_{rst} \ar[r]^-{\Phi^{C_{s}}\bbP c_r} & \Phi^{C_{s}}\bbP X_{st}
	} \]
\end{proof}

\begin{example}\label{restriction_system_pushup_1}
	For every termwise cofibrant naive restriction system of symmetric spectra $X_\bullet$, the orthogonal suspension spectra $\Sigma^\infty X_\bullet$ form a genuine restriction system of bispectra. In fact, the restriction map from the categorical fixed points
	\[ \kappa\colon \left(\Sigma^\infty X_{rs}\right)^{C_r} \arr \Phi^{C_r} \Sigma^\infty X_{rs} \]
	is an isomorphism, so we define the maps of the genuine restriction system using the inverse of the naive restriction system maps:
	\[ \Phi^{C_r} \Sigma^\infty X_{rs} \cong \left(\Sigma^\infty X_{rs}\right)^{C_r} \cong \Sigma^\infty (X_{rs}^{C_r}) \overset{\cong}{\longleftarrow} \Sigma^\infty X_s. \]
	Their compatibility follows immediately from rigidity.
\end{example}

\begin{example}\label{restriction_system_pushup_2}
	Similarly, for every genuine restriction system of orthogonal spectra $X_\bullet$, the symmetric suspension spectra $\Sigma^\infty X_\bullet$ form a genuine restriction system of bispectra with structure maps
	\[ \Phi^{C_r} \Sigma^\infty X_{rs} \cong \Sigma^\infty \Phi^{C_r} X_{rs} \oarr{\Sigma^{\infty}c_{r}} \Sigma^\infty X_s. \]
\end{example}

Next we place model structures on the various categories of restriction systems. It is simple enough to do this for naive restriction systems $\{X_n\}$ because they are equivalent to spectra with a $\bbZ$-action and no free $\bbZ$-orbits. The weak equivalences and fibrations are measured on each term $X_n$ separately, and the cofibrations are generated by the maps of naive restriction systems that are $*$ at all levels $n$ not divisible by $a$, and the shift-desuspensions of the canonical inclusions
\[ F_m(S^{k-1} \times \bbZ/a\bbZ)_+ \arr F_m(D^k \times \bbZ/a\bbZ)_+  \]
at all levels $n$ where $a \mid n$.

\begin{prop}\label{naive_restriction_system_model_structure}
	This defines a model structure on naive restriction systems of symmetric spectra.
\end{prop}

Next we turn to the model structure on genuine restriction systems. The idea, as in \cite[\S 5]{blumberg_mandell_cyclotomic}, is to build a model structure on genuine pre-restriction systems by thinking of them as algebras over a monad $\bbC$. As in that paper, this is not completely correct because $\bbC$ is only a monad on cofibrant inputs, but we can still use it to create free pre-restriction systems, which is enough to build the model structure.

Consider the category of sequences $\{X_n\}_{n \geq 0}$ in which the $n$th term $X_n$ is a $C_n$-equivariant orthogonal spectrum or bispectrum. Such a sequence is termwise cofibrant if each $X_n$ is cofibrant in the 
stable model structure.
We define $\bbC\{X_\bullet\}$ to be the sequence with $n$th term
\[  \bbC\{X_\bullet\}_{n} = \bigvee_{m \geq 1} \Phi^{C_m} X_{mn}. \]

If the sequence $\{X_\bullet\}$ is termwise cofibrant then we make $\bbC\{X_\bullet\}$ into a genuine pre-restriction system by defining $c_r$ to be the composite
\[ \xymatrix @R=1em{
	\Phi^{C_r} \bbC\{X_\bullet\}_{rs} \ar@{=}[d] && \bbC\{X_\bullet\}_{s} \\
	\Phi^{C_r} \left( \displaystyle\bigvee_{m \geq 1} \Phi^{C_m} X_{mrs} \right) &
	\ar@{<-}[l]_-\cong \displaystyle\bigvee_{m \geq 1} \Phi^{C_r} \Phi^{C_m} X_{mrs} &
	\ar@{<-}[l]_-\cong \displaystyle\bigvee_{m \geq 1} \Phi^{C_{mr}} X_{mrs}. \ar[u]
	} \]
Note that the final map is an inclusion of some but not all of the summands of $\bbC\{X_\bullet\}_{s}$. Hence this is a not a restriction system, only a pre-restriction system. The compatibility check for the structure maps $c_r$ can be done on each summand of the source separately, where it follows from rigidity.

\begin{lem}\label{free_restriction_system}
	On cofibrant inputs, $\bbC$ is the left adjoint of the forgetful functor from pre-restriction systems to sequences of equivariant spectra.
\end{lem}

\begin{proof}
	A map of pre-restriction systems $\bbC X \arr Y$ is given by maps of $C_n$-spectra
\[f_{m,n}\colon \Phi^{C_m} X_{mn} \arr Y_n\]
 for all $m,n \geq 1$ that are compatible along the structure maps $c_{r}$.  In the case of $r = m$, one finds that the compatibility condition implies that the maps $f_{1, n} \colon X_{n} \to Y_{n}$ determine all of the others.
\end{proof}

The generating cofibrations for the model structure on pre-restriction systems are constructed using the sets $I_n$ of generating cofibrations for $C_{n}$-spectra, for all $n \geq 1$, by considering each map as a morphism of equivariant sequences that is only nontrivial at the $n$th term, and then applying $\bbC$ to get a map of pre-restriction systems.
Concretely, these are the maps of pre-restriction systems that are trivial on the $k$th term unless $k \mid n$, in which case they are given by the maps of $C_k$-spectra
\[ \Phi^{C_{n/k}} F_{(m,V)}\left( C_n/C_a \times S^{k-1} \right)_+ \arr \Phi^{C_{n/k}} F_{(m,V)}\left( C_n/C_a \times D^k \right)_+. \]
Call the collection of such maps $\bbC I$. By the preservation properties of geometric fixed points detailed in \cite[\swcref{Lemma A.10}{\cref{swc-lem:gfp_properties}}]{clmpz-specWaldCat}, any $\bbC I$-cell complex is at the $n$th term of the restriction system an $I_n$-cell complex.

We perform the same construction to the generating acyclic cofibrations,
and in the case of bispectra to the generating cofibrations for the model structure on $C_n$-equivariant bispectra from \cite[\swcref{Proposition A.5}{\cref{swc-stable_model_structure}}]{clmpz-specWaldCat}.  Verifying that our definitions define a model structure is now straightforward by checking the properties termwise.

\begin{prop}\label{restriction_system_model_structure}
	These generating cofibrations and acyclic cofibrations, together with the termwise stable equivalences, define a stable model structure on pre-restriction systems of orthogonal spectra or of bispectra.
\end{prop}

\bibliography{references}
\bibliographystyle{amsalpha2}

\end{document}